\documentclass[onefignum,onetabnum]{amsart}

\usepackage{lipsum}
\usepackage{amssymb,amsmath,amsthm,dsfont,amsfonts,empheq}
\usepackage{graphicx}
\usepackage{subcaption}
\usepackage{wrapfig}
\usepackage{epstopdf}
\usepackage{algorithmic}
\usepackage[many]{tcolorbox}
\usepackage{dsfont}
\usepackage{nicefrac}
\usepackage{todonotes}
\usepackage{mathtools}
\usepackage{geometry}
\usepackage{hyperref}
\usepackage[capitalise]{cleveref}
\usepackage{cite}
\usepackage{graphicx}
\usepackage{enumerate}
\usepackage{bm}
\usepackage{stmaryrd}
\usepackage{array}
\usepackage{esint}

\usepackage[lang=english]{fixme}
	\fxsetup{draft}
	\fxusetheme{color} 
\ifpdf
  \DeclareGraphicsExtensions{.eps,.pdf,.png,.jpg}
\else
  \DeclareGraphicsExtensions{.eps}
\fi

\newtheorem{theorem}{Theorem}[section]
\newtheorem{lemma}[theorem]{Lemma}

\newtheorem{definition}[theorem]{Definition}

\newtheorem{remark}[theorem]{Remark}

\usepackage{amsopn}

\newcommand{\e}{\varepsilon}
\newcommand{\Z}{\mathbb{Z}}
\newcommand{\N}{\mathbb{N}}
\newcommand{\R}{\mathbb{R}}

\newcommand{\omg}{{\Omega_\e^g}}
\newcommand{\omf}{{\Omega_\e^f}}
\newcommand{\dive}{\operatorname{div}}

\newcommand{\di}[1]{\,\mathrm{d}#1}

\newcommand{\ID}{\operatorname{Id}}
\newcommand{\interior}{\operatorname{int}}
\newcommand{\derh}{\partial_t^h}

\newcommand{\weakstar}{\mathrel{\overset{*}{\rightharpoonup}}}

\newtcolorbox{mybox}[1]{%
    tikznode boxed title,
    enhanced,
    arc=0mm,
    interior style={white},
    attach boxed title to top left= {yshift=-\tcboxedtitleheight/2-0.05cm, xshift=0.7cm},
    fonttitle=\small\bfseries,
    colbacktitle=white,coltitle=black,
    boxed title style={size=small,colframe=white,boxrule=0pt},
    title={#1}}

\ifpdf
\hypersetup{
  pdftitle={Homogenization of a Poroelasticity Model for Fibre-Reinforced Hydrogels},
  pdfauthor={M.~Eden and H.~S.~Mahato}
}
\fi

\title[A Two-Component System for Fibre-Reinforced Hydrogels]{A Two-Component Poro-viscoelastic System for Fibre-Reinforced Hydrogels: Analysis and Homogenization
}

\author{Michael Eden}
\address{Department of Mathematics, University of Regensburg, Germany}

\author{Hari Shankar Mahato}
\address{Department of Mathematics, IIT Kharagpur, India}

\begin{document}

\maketitle

\begin{abstract}
We study the multiscale behavior of a coupled visco-poroelastic system arising in the modelling of fibre-reinforced hydrogels (FIHs) used in tissue engineering scaffolds.
The composite material consists of a periodic fibre scaffold, which is governed by quasi-static linear elasticity, and a hydrogel phase saturating the interstitial space, which is modelled as a Biot linear poroelastic medium enhanced with Kelvin--Voigt structural damping.
The two phases are coupled through continuity of displacement and traction across their shared interface.
Both the fibre scaffold and the hydrogel phase are connected, so that mechanical forces can be transmitted through the composite and interstitial fluid can flow directly through the hydrogel network.

Starting from a microscopic ($\e$-scale) model, we derive uniform a-priori estimates and establish well-posedness via a Rothe time-discretisation argument for both the case of standard Biot fluid content \(\eta=0\) and the case of viscous fluid content \(\eta=\alpha\delta>0\).
We then perform a rigorous two-scale homogenization in the limit $\e \to 0$ using periodic unfolding.
In addition to the usual effective elasticity, storage, coupling, and permeability coefficients, the homogenized constitutive laws contain nonlocal-in-time memory terms generated by the microscopic viscoelastic relaxation.
All effective coefficients and memory kernels are explicitly characterized in terms of the microscale geometry and material parameters.
\end{abstract}



\section{Introduction}
\label{sec:intro}

Fibre-reinforced hydrogels (FIHs) are complex biomaterials used in tissue engineering (e.g., for growing cartilage).
In a typical FIH (see \cref{fig:1}), a periodic scaffold of fibres is saturated with a hydrogel with the goal of simultaneously replicating the structural stiffness of load-bearing tissues as well as the hydrophilic, nutrient-permeable environment required for cell viability and proliferation \cite{ARC13,CK18,CYB18,KMH12,LN17}.
They are mainly used in different biomedical applications for the in-vitro growth of tissues which can then be used to repair cartilage, meniscus, or cardiac tissue patches \cite{PW17,KMH12}.
In that context, understanding and controlling the mechanical behavior of FIHs is important to ensure that they act as good surrogates for actual tissue \cite{CYB18,KMH12}.

\begin{figure}[h]
\centering
\includegraphics[width=.7\textwidth]{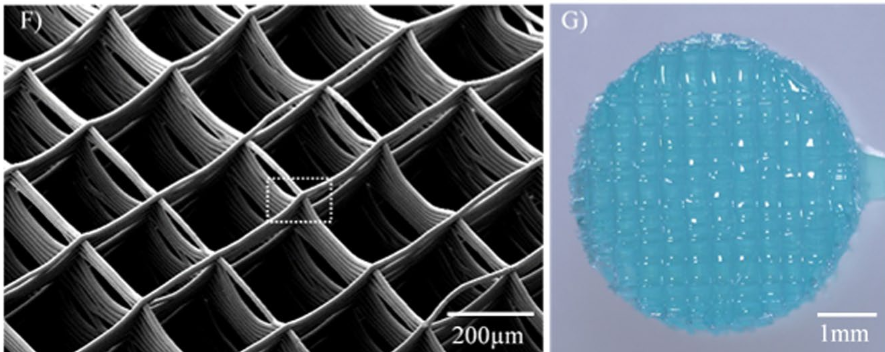}%
\caption{The periodic fibre scaffold: empty (left) and saturated with hydrogel (right).
This figure is taken from \cite{CH18} under a Creative Commons license (see: \url{http://creativecommons.org/license/by/4.0/})}%
\label{fig:1}%
\end{figure}

However, their mechanical properties are quite complex due to their composite nature where some processes are intrinsically coupled across different length scales:
At the microscale ($\e \sim 10$--$100\,\mu$m), e.g., 
individual fibre structures and gel pockets are able to exchange both mechanical forces and interstitial fluid.
At the macroscale ($\sim$ mm--cm), the FIH experiences deformations and diffusive transport governed by effective, homogenized properties.
Over the last few years, there has been quite some interest in different approaches of describing, modeling, and calculating the effective material properties of different FIHs under sometimes very different underlying assumptions on the fibres, the hydrogel, and their connection:
In these models, the fibre scaffolding is generally described as either linearly elastic (e.g.,\cite{hollister_computational_2007,CK18}) or hyperelastic (e.g., \cite{CH18,nian_constitutive_2023}).
The hydrogel, which is itself a composite network of hydrophilic polymer chains saturated with water-based fluids, is sometimes modeled as an elastic material \cite{CH18} or as a Newtonian fluid \cite{hollister_computational_2007}; but generally it is better understood as a saturated porous structure where, e.g., Biot poroelasticity applies (see \cite{CK18,EM22}).

Effective models for FIHs strongly depend on these underlying modeling assumptions and are often based on non-rigorous averaging techniques; as examples we point to \cite{CH18}, where a two-scale model is proposed, and to \cite{hollister_computational_2007,CK18}, where effective models are calculated via expansions.
In \cite{EM22}, an effective material model was rigorously derived using two-scale convergence based on a coupled elastic-poroelastic model, where, crucially, the gel pockets inside the fibre scaffolding were assumed to be well-separated.
Recently, this was combined with a dimension reduction in \cite{Chakrabortty2025}.

Beyond the specific setting of fibre-reinforced hydrogels, there are several structurally related homogenization results for heterogeneous viscoelastic and poro-viscoelastic media.
In \cite{ptashnyk_seguin_2017}, for instance, elastic microfibrils are embedded in a Kelvin--Voigt viscoelastic matrix in a model for plant cell wall biomechanics where the homogenization result also includes macroscopic memory terms. 
In contrast, the matrix phase in the present work is poroelastic, so that the mechanical relaxation is additionally coupled to pore pressure and Darcy transport.
Poro-viscoelastic macroscopic models have also been derived by
upscaling fluid-saturated solids with a viscoelastic skeleton; see
\cite{rohan_poro-viscoelasticity_2014}.
There, the solid constituent already exhibits fading-memory viscoelasticity at the microscopic level, whereas in the present model the memory terms result from the heterogeneous coupling of a purely elastic fibre phase with a Kelvin--Voigt poroelastic gel phase.

In this current work, we are building upon these earlier works to address the following two aspects:
\begin{itemize}
    \item[$(a)$] In reality, the gel pockets are not well separated and allow for exchange (see the empty scaffold in \cref{fig:1} where holes in the fibre walls are clearly visible) of mass and momentum.
    In this work, we allow for such exchanges by working with a connected-connected geometry where both the gel phase and the fibre scaffolding are connected.
    \item[$(b)$] The polymer-chain networks making up the hydrogel can exhibit \emph{viscoelastic} creep and stress relaxation.
    This viscoelastic behavior is not captured by poroelasticity but is an important part of the mechanical behavior of tissue \cite{moutos_composite_2008}.
    In this work, we propose a visco-poroelasticity model for the gel phase, where the stress tensor additionally depends on the rate of change of deformations.
\end{itemize}

To be more precise, we propose a microscale model for a highly heterogeneous FIH (modeled as a domain $\Omega\subset\R^3$) consisting of a periodically arranged fibre scaffold phase occupying a domain $\omf\subset\R^3$ and a hydrogel phase $\omg\subset\R^3$ saturating the connected interstitial space.
The fibre phase is governed by quasi-static linear elasticity, i.e.,
\begin{subequations}
\begin{equation}
-\dive(\mathcal{C}^fe(u_\e^f))=f_\e^f\quad\text{in}\quad S\times\omf
\end{equation}
and the hydrogel is modeled as a linear visco-poroelastic medium (using Kelvin-Voigt damping) via
\begin{align}
-\dive\left(C^ge(u_\e^g+\delta \partial_tu_\e^g)-\alpha p_\e\mathds{I}\right)&=f_\e^g&\quad &\text{in}\quad S\times\omg,\\
\partial_t\left(cp_\e+\alpha\dive u_\e^g+\eta \dive \partial_tu_\e^g\right)-\dive\left(K\nabla p_\e\right)&=g_\e&\quad &\text{in}\quad S\times\omg.
\end{align}
\end{subequations}
Here, $(u_\e^f,u_\e^g)$ are the deformations in their respective phases and $p_\e$ is the pore pressure inside the hydrogel.
The parameters $\delta$ and $\eta$ govern the strength of the damping and the processes in the two phases are coupled via transmission conditions at the interface $\partial\Omega^g_\e\cap\partial\Omega_\e^f$.
In this work, we consider the two structurally special cases
\[
    \eta=0
    \qquad\text{and}\qquad
    \eta=\alpha\delta.
\]
For $\eta=0$, the fluid content has the standard Biot form. 
If $\eta=\alpha\delta$, the same auxiliary displacement variable
\[
    w_\e=u_\e+\delta\partial_tu_\e
\]
appears in both the mechanical stress and the fluid content.
The details and general idea behind this model are presented in \cref{sec:modeling}.

We note that the switch from elasticity and poro-viscoelasticity behavior between the fibre phase and the gel phase creates the main analytical difficulty in the analysis of the problem.
The auxiliary displacement $u+\delta\partial_tu$, which is commonly used in the analysis of poro-viscoelastic systems, is naturally defined only in the gel phase.
Because of that, the standard global reduction to an elliptic problem coupled with an evolution equation is not available.
To get around this issue, we use tailored Rothe discretisations to establish well-posedness of the microscopic transmission problem for the two structurally specific choices $\eta=0$ and $\eta=\alpha\delta$ and derive estimates that are uniform with respect to $\e$.

Now, the main goal of this work is to rigorously derive an effective macroscopic description of the material behavior which is done via periodic homogenization.
To that end, we use the small parameter $\e>0$ to describe the periodicity of the geometry and show that the above problem has a unique solution with uniform estimates in $\e$.
These estimates provide compactness for the displacement and pressure, but, crucially, also for the unfolded displacement gradients and their time derivatives. 
Using periodic unfolding, we pass to a coupled macro--micro limit system, where the well-prepared initial data are used to identify the initial trace of the microscopic displacement corrector.
And after eliminating this corrector via a number of stationary and relaxation cell problems we arrive at the following effective model (the details are spelled out in \cref{thm:homog}) when $\eta=0$:
\begin{align}\label{eq:limit_problem}
  -\dive\left(\mathbb C^{\rm h}e(u)-A^{\rm h}p +\delta\mathcal S^{\rm h}\left[e(\partial_tu),\partial_tp)\right]\right)
  &= f
  && \text{in } S\times\Omega,\\
  \partial_t\left( c^{\rm h}p +A^{\rm h}:e(u) +\delta\, \mathcal G^{\rm h} \bigl[e(\partial_tu),\partial_tp\bigr]\right)
  -\dive\bigl(K^{\rm h}\nabla p\bigr) &= |Y^g|g &&\text{in }S\times\Omega.
\end{align}
Here, $\mathbb C^{\rm h}$, $A^{\rm h}$, $c^{\rm h}$, and $K^{\rm h}$ are the standard effective coefficients defined through stationary cell problems.
In this respect, the resulting limit system is consistent with the homogenization limit in the poroelastic case in \cite{EM22} when $\delta=\eta=0$ in the microscopic problem.
The two additional operators $\mathcal S^{\rm h}$ and $\mathcal G^{\rm h}$, however, are rather complex and contain both local and memory terms for the time derivatives.
They are given by
\begin{align}\label{eq:coupling_operators}
\mathcal S^{\rm h}[E,r]&=\mathbb D^{\rm h}E+\partial_t\mathbb M^{\rm h}*E+\partial_t\mathbb{R}^{\rm h}*r,\\
\mathcal G^{\rm h}[E,r]&=-\partial_t R^{\rm h}*E +H^{\rm h}*r. \label{eq:hom-fluid-content}
\end{align}
%
For the precise definition of these coefficients and memory kernels, we refer to \cref{sec:homog}.
The emergence of memory terms under homogenization is a known phenomenon in heterogeneous viscoelastic materials. 
We refer to \cite{abdessamad_memory_2009} where it was shown for a Kelvin--Voigt composite with rapidly varying space- and time-dependent coefficients that a constitutive law which is local in time at the microscopic level may lead to a nonlocal-in-time effective response.

In the present setting, this mechanism is coupled to the poroelastic pressure and mass-balance equations. 
Consequently, the effective memory terms occur not only in the macroscopic stress but also in the macroscopic fluid content. 
They involve the histories of $e(\partial_tu)$ and $\partial_tp$ and encode the interaction between the Kelvin--Voigt gel phase and the purely elastic fibre scaffold.

The case $\eta=\delta\alpha$, expectedly, leads to a more involved effective fluid content, while the mechanical equation remains unchanged.
In this case, the homogenised mass balance takes the form
\begin{multline}\label{eq:limit-pressure-eta-delta-alpha}
    \partial_t\Biggr[c^{\rm h}p + A^{\rm h}:e(u) + \delta\, \mathcal G^{\rm h} \bigl[e(\partial_tu),\partial_tp\bigr] \\ 
    + \delta\,\partial_t\Bigl( \bigl(c^{\rm h}-c|Y^g|\bigr)p + A^{\rm h}:e(u) + \delta\, \mathcal G^{\rm h} \bigl[e(\partial_tu),\partial_tp\bigr] \Bigr) \Biggr] - \dive\bigl(K^{\rm h}\nabla p\bigr) = |Y^g|h \quad\text{in }S\times\Omega.
\end{multline}

We want to point out that the mathematical theory of purely poro-viscoelastic PDE models is well understood and we point to \cite{BGSW16,BGSV19,BS22,BMW23} for further details.
The main mathematical difficulty in this work is the treatment of the jump in viscous behavior across the interface and the corresponding jump in regularity.
The standard global reduction to an elliptic poroelastic problem coupled with an ODE (as outlined in \cite[Section 3.2]{BGSW16}) is not directly applicable, since the natural auxiliary displacement may jump across the fibre--gel interface. 
We overcome this by constructing tailored Rothe discretisations for the two admissible choices of the viscous storage coefficient.
For the limit process $\e\to0$, we rely on the method of periodic unfolding \cite{CDG02,CDG08} (which is equivalent to the notion of two-scale convergence \cite{Ngu89,All92,LNW02}).
Due to the connected-connected geometry (often in homogenization it is assumed that one phase is disconnected) we are employing in our setup, we additionally rely on more involved extension operators as introduced in \cite{Hoepker2014}.

The paper is organised as follows:
In~\cref{sec:modeling}, we introduce the geometric setup and the microscale model. We also collect the mathematical assumptions on the data and coefficients.
This is followed up, in \cref{sec:analysis}, with the analysis (well-posedness and $\e$-uniform energy estimates) of the microscale model where we distinguish between the two cases $\eta=0$ (\cref{sssec:eta0}) and $\eta=\alpha\delta$ (\cref{sssec:eta_alpha}).
Finally, \cref{sec:homog} then focuses on the homogenization process and presents the homogenized model (\cref{thm:homog}).

\section{Modeling and Setup}\label{sec:modeling}
In this section, we introduce the mathematical model we are considering and collect the main assumptions.
The geometric setup we have in mind models a porous medium where both phases (solid phase and fluid phase) are connected (see also \cref{fig:fibre-geometry-overview}).

\begin{figure}[htbp]
    \centering
    \begin{subfigure}[c]{0.58\textwidth}
        \centering
        \includegraphics[width=.75\textwidth]{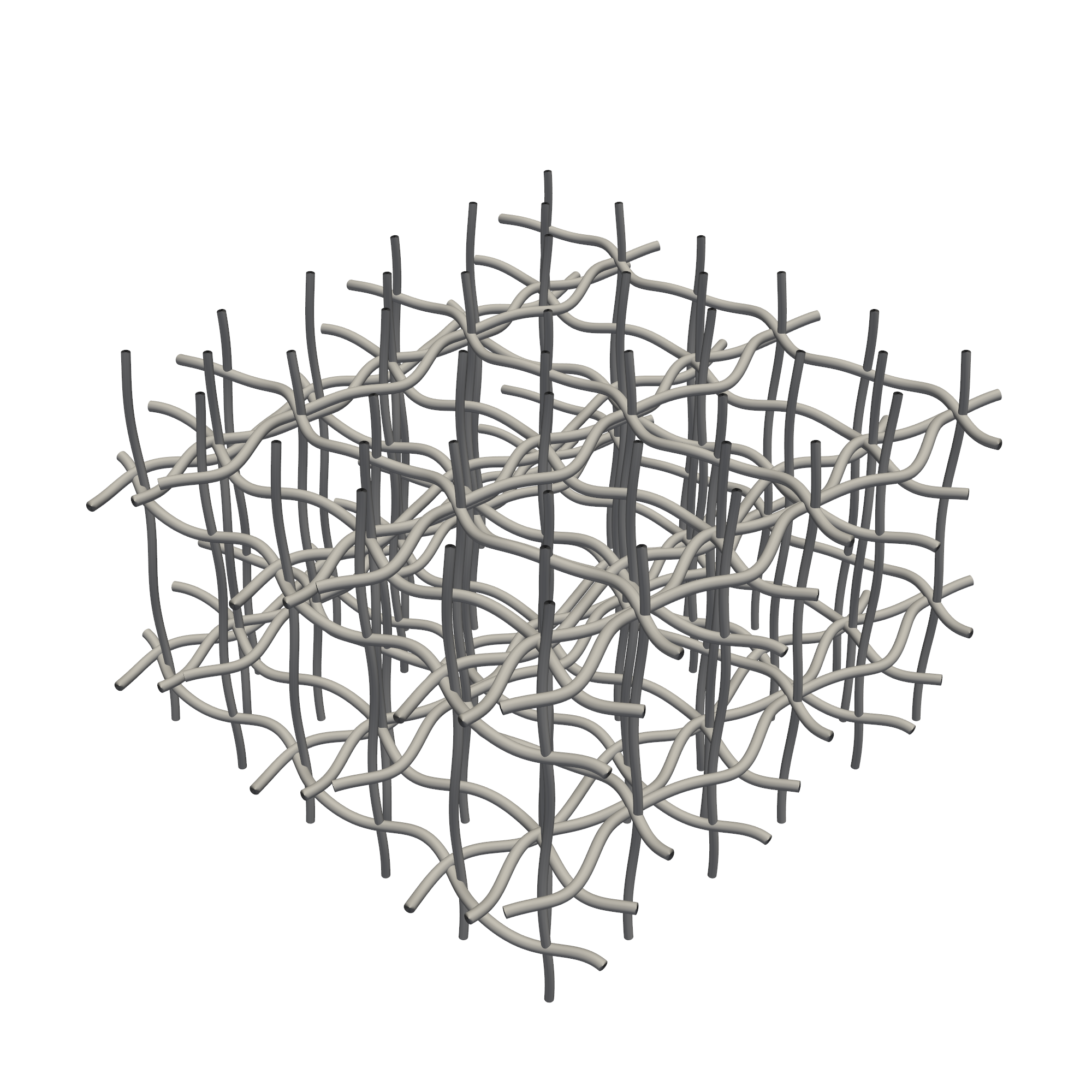}
    \end{subfigure}
    \hfill
    \begin{subfigure}[c]{0.38\textwidth}
        \centering
        \includegraphics[width=.8\textwidth]{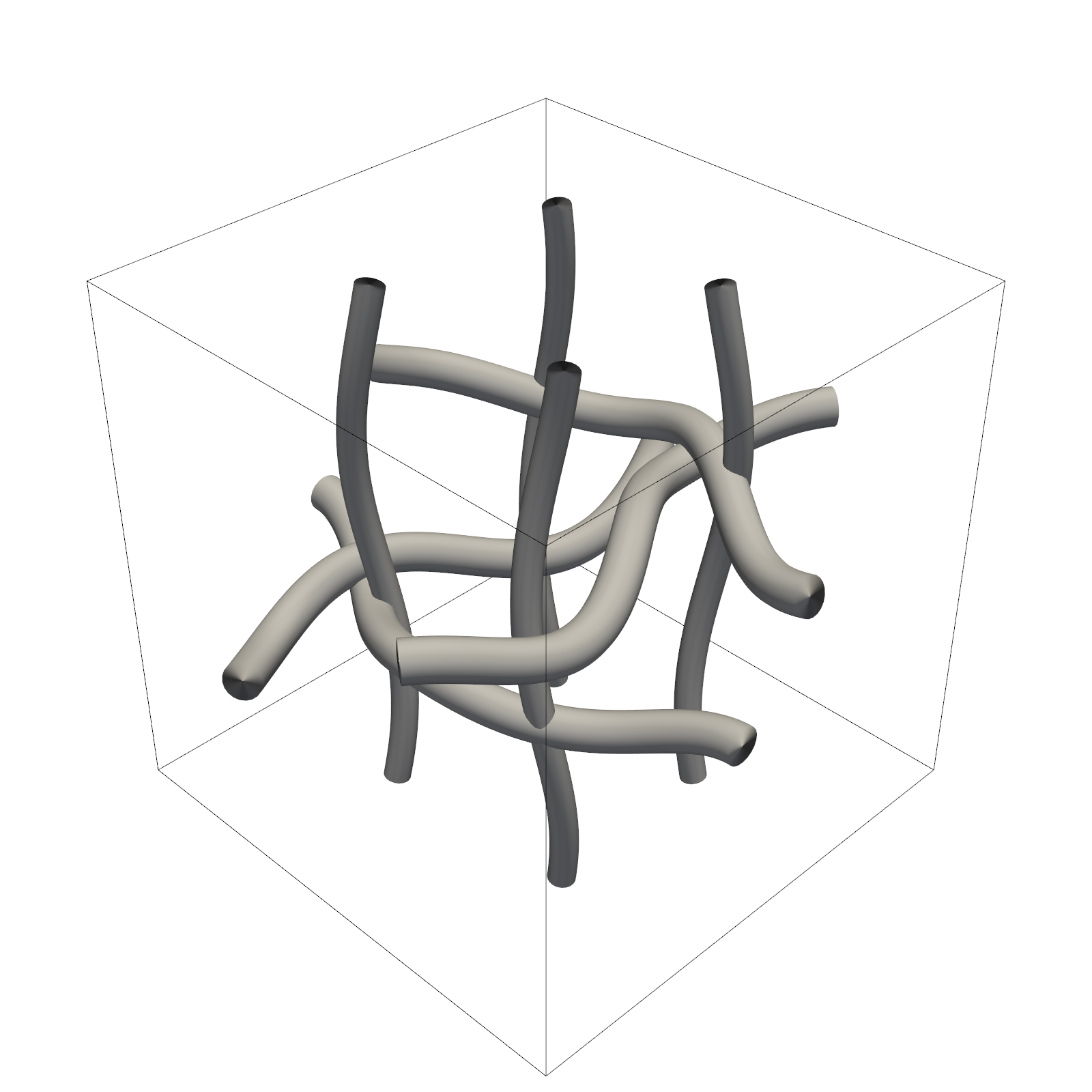}
        \label{fig:periodic-geometry}
    \end{subfigure}
    \caption{Schematic representation of the periodic geometry (left) and the unit cell (right). Note that both the fibres and the fluid space are connected.
    For visibility reasons, we have drawn only a few fibres.}
    \label{fig:fibre-geometry-overview}
\end{figure}

To make this precise, let $\Omega\subset\R^3$ be a bounded Lipschitz domain representing the overall system and let $S=(0,T)$, $T>0$, represent the time interval of interest.
We denote the outer normal vector of $\Omega$ with $\nu=\nu(x)$.
Let $Y=(0,1)^3$ be the open unit cell in $\R^3$ and take two disjoint Lipschitz domains $Y^{f},$ $Y^{g}\subset Y$ such that $\overline{Y}=\overline{Y^f}\cup\overline{Y^g}$.
We set $\Gamma:=\overline{Y^f}\cap\overline{Y^g}$.
With $n_\Gamma=n_\Gamma(y)$, $y\in\Gamma$, we denote the normal vector of $\Gamma$ pointing outwards of $Y^{g}$.
Now, let $\e_0$ be chosen such that $\Omega$ can be perfectly tiled with $\e_0Y$ cells, set $\e_n=2^{-n}\e_0$,\footnote{In the following, we suppress the index $n$ and note that always $\e\in\{\e_n\ : \ n\in\N\}$ and $\e\to0$ is understood as $\lim_{n\to\infty}\e_n$.} and let $I_{\e}\subset\Z^3$ be the corresponding index sets such that
\[
\Omega=\interior\left(\bigcup_{k\in I_{\e}}\e(\overline{Y}+k)\right).
\]
We introduce the periodic sets
\[
\Omega_\e^f=\interior\left(\bigcup_{k\in I_{\e}}\e(\overline{Y^f}+k)\right),\quad \Omega_\e^g=\interior\left(\bigcup_{k\in I_{\e}}\e(\overline{Y^g}+k)\right),\quad \Gamma_\e=\partial\Omega_\e^f\cap\partial\Omega_\e^g\cap\Omega
\]
representing the fibre and gel domain  as well as the internal boundary between fibre and gel, respectively.
We additionally assume that, for every admissible $\e>0$, the sets $\Omega_\e^f$ and $\Omega_\e^g$ are connected Lipschitz domains.
This imposes an additional geometric condition on $Y^f$ and $Y^g$, namely that their periodic extensions form connected phases and meet across a Lipschitz interface.
We also introduce the outer boundaries
\[
(\partial\Omega)_\e^l=\partial\Omega_\e^l\cap\partial\Omega\quad (l=f,g).
\]
In the following, we take $\chi_\e^{i}\colon\Omega\to\{0,1\}$ ($i=f,g$) to denote the characteristic functions corresponding to $\Omega_\e^i$.
Now, let $u_\e^f\colon S\times\omf\to\R^3$ represent the deformation in the fibre part.
Assuming that the mechanical response of the fibre scaffold is governed by quasi-stationary elasticity, we then have
\begin{subequations}\label{eq:weak_form}
\begin{align}
-\dive(\mathcal{C}^fe(u_\e^f))=f_\e^f\quad\text{in}\quad S\times\Omega_\e^f.\label{eq:2.1a}
\end{align}
Here, $e(w)=1/2(\nabla w+\nabla w^T)$ is the linearized strain tensor, $\mathcal{C}^f\in\R^{3\times3\times3\times3}$ the elasticity tensor, and $f_\e^f$ possible volume forces.

The hydrogel, on the other hand, is itself a composite (polymer networks saturated with water) with complex mechanical properties.
These are sometimes modeled as a linear poroelastic material, see, e.g., \cite{CYB18,CK18,KMH12,LN17,EM22}.
While this approach is good at capturing fluid movement inside the hydrogel, it neglects the effect of relaxation of the polymer networks \cite{Hu2012,Wang2014}.
In this work, we therefore propose a more complex visco-poroelasticity model, which is able to also account for relaxation.

To that end, we introduce the pore pressure $p\colon S\times\omg\to\R$ and the gel deformation $u_\e^g\colon S\times\omg\to\R^3$.
Following the Kelvin-Voigt approach of including strong (or “structural”) damping, cf. \cite{CHR82,CHT89}, we introduce the relaxation parameter $\delta>0$ and write the stress tensor $\sigma$ as
\[
\sigma=\mathcal{C}^ge(u_\e^g+\delta\partial_tu_\e^g)-\alpha p_\e I
\]
where $\mathcal{C}^g\in\R^{3\times3\times3\times3}$ is the elasticity tensor of the hydrogel and $\alpha>0$ the Biot-Willis parameter which measures how much fluid pressure contributes to the deformation of the hydrogel compared to the total applied stress.
We point out that this structure can be obtained as a limiting case of the poro-viscoelastic modeling in \cite{CO04,BGSW16,BMW23}.

We model the effect of viscoelasticity on the fluid content $\zeta$ by considering
\begin{equation*}
\zeta=c p_\e+\alpha \nabla \cdot u_\e+\eta \nabla \cdot \partial_tu_\e.
\end{equation*}
where $c>0$ is the poroelastic storage coefficient and where for the viscous storage coefficient $\eta$ we take either $\eta=0$ or $\eta=\alpha\delta$.
Here, the case $ \eta = 0$ corresponds to the standard Biot definition of the fluid content, which commonly appears in the literature for linearized poro-elastic systems with and without viscoelastic effects.
This assumes that the pore space adjusts instantaneously to any deformation.
In the second case, $\eta=\alpha\delta$, on the other hand, the fluid inside the hydrogel also has an internal relaxation.
The visco-poroelastic response of the hydrogel is then modeled as
\begin{align}
-\dive\left(C^ge(u_\e+\delta \partial_tu_\e)-\alpha p_\e\mathds{I}\right)&=f_\e^g&\quad &\text{in}\quad S\times\Omega_\e^g,\label{eq:2.1b}\\
\partial_t\left(cp_\e+\alpha\nabla\cdot u_\e+\eta \nabla \cdot \partial_tu_\e\right)-\dive\left(K\nabla p_\e\right)&=g_\e&\quad &\text{in}\quad S\times\Omega_\e^g,\label{eq:2.1c}
\end{align}
where $K\in\R^{3\times3}$ is the Darcy permeability and $g_\e$ denotes a potential volume source of fluid.

Now, at the interface between the fibre and gel parts, we assume continuity of both deformation and stresses and no fluid penetration into the fibre:
\begin{alignat}{2}
\left(\mathcal{C}^ge(u_\e^g+\delta \partial_tu_\e^g)-\alpha p_\e\mathds{I}\right)n_\e&=(\mathcal{C}e(u_\e^f))n_\e&\quad &\text{on}\quad S\times\Gamma_\e,\label{eq:2.1d}\\
u_\e^f&=u_\e^g&\quad &\text{on}\quad S\times\Gamma_\e,\label{eq:2.1e}\\
-K\nabla p_\e\cdot n_\e&=0&\quad &\text{on}\quad S\times\Gamma_\e\label{eq:2.1f}
\end{alignat}
Finally, we close the system with homogeneous conditions at the outer boundaries as well as initial conditions:
\begin{alignat}{2}
u_\e^f&=0&\quad &\text{on}\quad S\times(\partial\Omega)_\e^f,\label{eq:2.1g}\\
u_\e^g&=0&\quad &\text{on}\quad S\times(\partial\Omega)_\e^g,\label{eq:2.1h}\\
-K\nabla p_\e\cdot n_\e&=0&\quad &\text{on}\quad S\times(\partial\Omega)_\e^g,\label{eq:2.1i}\\
u_\e(0)&=u_{\e,0}&\quad &\text{in}\quad \Omega\label{eq:2.1j},\\
p_\e(0)&=p_{\e,0}&\quad &\text{in}\quad \Omega_\e^g\label{eq:2.1k}.
\end{alignat}
In the case $\eta=\delta\alpha$, we additionally have the initial velocity
\[
\partial_tu_\e(0)=v_{\e,0}\quad \text{in}\quad \Omega\label{eq:2.1l}.
\]
\end{subequations}
%
In the following, we write $\mathcal{C}(x)=\chi_{\omg}(x)\mathcal{C}^g+\chi_{\omf}(x)\mathcal{C}^f$ and similarly for $u_\e,f_\e$.
Please note that with the continuity condition \eqref{eq:2.1e} together with the boundary conditions \eqref{eq:2.1g} and \eqref{eq:2.1h}, we can expect the regularity $u_\e\in H_0^1(\Omega)$.
With this convention, \cref{eq:2.1a,eq:2.1b} can be written together as
\[
-\nabla\cdot\left(Ce(u_\e)+\delta C^ge(\partial_tu_\e)-\alpha p_\e\mathds{I}\right)=f_\e\quad \text{in}\quad S\times(\omf\cup\omg).
\]

Below we summarize our main assumptions regarding the coefficients and data in our problem:
\begin{itemize}
\item[\textbf{(A1)}]\textbf{Coefficient assumptions.} The fourth-order stiffness tensors $\mathcal{C}^f,\mathcal{C}^g\in\R^{3\times3\times3\times3}$ are symmetric and positive definite, i.e.,
    \[
    \mathcal{C}^l_{ijkl}=\mathcal{C}^l_{jikl}=\mathcal{C}^l_{ijlk}=\mathcal{C}^l_{klij},\quad \mathcal{C}^lM:M\ge cM:M
    \]
    for all symmetric matrices $M\in\R^{3\times3}_{\rm sym}$ and $l=f,g$ for some constant $c>0$.
    Also, the permeability matrix $K\in\R^{3\times3}$ is symmetric and positive definite, i.e., 
    \[
    K\xi\cdot\xi\ge c|\xi|^2
    \]
    for all $\xi\in\R^3$ for some constant $c>0$.
    Finally, the Biot--Willis parameter \(\alpha>0\), the storage coefficient \(c>0\), and the relaxation time \(\delta>0\) are positive.
    For the viscous storage coefficient, we distinguish between the
    cases
    \[
        \eta=0
        \qquad\text{and}\qquad
        \eta=\alpha\delta.
    \]
    \item[\textbf{(A2)}] \textbf{Data and initial conditions.}
    We impose the following assumptions on the data.

    \begin{enumerate}
        \item[\textbf{(i)}]
        In the case \(\eta=0\), we assume
        \[
            f_\e
            \in
            L^\infty\bigl(S;L^2(\Omega)^3\bigr),
            \qquad
            \partial_t f_\e^f
            \in
            L^2\bigl(S;L^2(\Omega_\e^f)^3\bigr),
        \]
        and
        \[
            g_\e
            \in
            L^2\bigl(S;L^2(\Omega_\e^g)\bigr),
            \qquad
            u_{\e,0}
            \in
            H_0^1(\Omega)^3,
            \qquad
            p_{\e,0}
            \in
            L^2(\Omega_\e^g).
        \]
        Here and below,
        \[
            f_\e^l
            :=
            f_\e|_{\Omega_\e^l},
            \qquad l\in\{f,g\}.
        \]
        Moreover, the initial displacement is compatible with the
        quasistatic equilibrium in the fibre phase
        \[
            -\dive\bigl(
                \mathcal C^f e(u_{\e,0})
            \bigr)
            =
            f_\e^f(0)
            \quad\text{in }\Omega_\e^f
        \]
        in the weak sense.

        \item[\textbf{(ii)}]
        In the case \(\eta=\alpha\delta\), we assume
        \[
            f_\e^f
            \in
            H^2\bigl(S;L^2(\Omega_\e^f)^3\bigr),
            \qquad
            f_\e^g
            \in
            H^1\bigl(S;L^2(\Omega_\e^g)^3\bigr),
        \]
        and
        \[
            g_\e
            \in
            L^2\bigl(S;L^2(\Omega_\e^g)\bigr).
        \]
        Furthermore,
        \[
            u_{\e,0},v_{\e,0}
            \in
            H_0^1(\Omega)^3,
            \qquad
            p_{\e,0}
            \in
            H^1(\Omega_\e^g).
        \]
        We set
        \[
            w_{\e,0}
            :=
            u_{\e,0}+\delta v_{\e,0}
        \]
        and assume the full mechanical compatibility condition
        \begin{empheq}[left=\empheqlbrace]{alignat*=2}
        -\dive\bigl(\mathcal C^f e(u_{\e,0})\bigr)
            &=f_\e^f(0)
            &\quad&\text{in }\Omega_\e^f,\\
        -\dive\bigl(\mathcal C^g e(w_{\e,0})-\alpha p_{\e,0}I\bigr)
            &=f_\e^g(0)
            &\quad&\text{in }\Omega_\e^g,\\
        \bigl(\mathcal C^g e(w_{\e,0})-\alpha p_{\e,0}I\bigr)n_\e
            &=\mathcal C^f e(u_{\e,0})n_\e
            &\qquad&\text{on }\Gamma_\e.
        \end{empheq}
        in the weak sense.
    \end{enumerate}
    In both cases, we define the initial microscopic fluid content by
    \[
        \zeta_{\e,0}
        :=
        c p_{\e,0}
        +\alpha\dive u_{\e,0}
        +\eta\dive v_{\e,0}
        \qquad\text{in }\Omega_\e^g,
    \]
    where the last term is omitted if $\eta=0$.
    \item[\textbf{(A3)}] \textbf{Limits of data and initial values.}
    The data and initial values are uniformly bounded with respect to
    $\e$ in the spaces specified in \textbf{(A2)}. Moreover, there exist
    \[
        f\in L^2\bigl(S;L^2(\Omega)^3\bigr),
        \qquad
        g\in L^2(S\times\Omega),
        \qquad
        u_0\in H_0^1(\Omega)^3,
    \]
    \[
        p_0\in L^2(\Omega),
        \qquad
        \zeta_0\in L^2(\Omega),
    \]
    such that, as $\e\to0$,
    \[
        f_\e\rightharpoonup f
        \quad\text{in }L^2\bigl(S;L^2(\Omega)^3\bigr),
        \qquad
        \widehat g_\e\rightharpoonup |Y^g|g
        \quad\text{in }L^2(S\times\Omega),
    \]
    and
    \[
        u_{\e,0}\rightharpoonup u_0
        \quad\text{in }H_0^1(\Omega)^3,
        \qquad
        \mathcal T_\e(\widehat{p_{\e,0}})
        \rightharpoonup
        \chi^g p_0
        \quad\text{in }L^2(\Omega\times Y),
    \]
    as well as
    \[
        \widehat{\zeta_{\e,0}}
        \rightharpoonup \zeta_0
        \quad\text{in }L^2(\Omega).
    \]
    Here, the hat denotes extension by zero from $\Omega_\e^g$ to
    $\Omega$.
    \item[\textbf{(A4)}] \textbf{Well-prepared initial data.}
    Setting
    \[
        u_{1,0}^{\rm eq}(x,y)
        :=
        \boldsymbol\eta(y):e(u_0(x))+q(y)p_0(x),
    \]
    we assume
    \[
        \mathcal T_\e(\nabla u_{\e,0})
        \to
        \nabla_xu_0+\nabla_yu_{1,0}^{\rm eq}
        \quad\text{in }L^2(\Omega\times Y)^{3\times3}.
    \]
    For $\eta=\alpha\delta$, we additionally assume that
    \[
        v_{\e,0}\rightharpoonup v_0
        \quad\text{in }H_0^1(\Omega)^3
    \]
    and that, with $w_0:=u_0+\delta v_0$, there exists
    \[
        w_{1,0}\in
        L^2\bigl(\Omega;H_\#^1(Y)^3\bigr)
    \]
    such that
    \[
        \mathcal T_\e(\nabla w_{\e,0})
        \to
        \nabla_xw_0+\nabla_yw_{1,0}
        \quad\text{in }L^2(\Omega\times Y)^{3\times3}.
    \]
\end{itemize}
Note that the convergence assumptions in \textup{(A3)}--\textup{(A4)} imply the
following compatibility relations for the limiting initial fluid
content. If $\eta=0$, then
\[
    \zeta_0
    =
    \int_{Y^g}
    \left[
        cp_0
        +
        \alpha\bigl(
            \dive_x u_0
            +
            \dive_y u_{1,0}^{\rm eq}
        \bigr)
    \right]\di y
    \quad\text{in }L^2(\Omega).
\]
If $\eta=\alpha\delta$, then
\[
    \zeta_0
    =
    \int_{Y^g}
    \left[
        cp_0
        +
        \alpha\bigl(
            \dive_x w_0
            +
            \dive_y w_{1,0}
        \bigr)
    \right]\di y
    \quad\text{in }L^2(\Omega).
\]




With this we are able to formulate the solution concept for the problem given by \cref{eq:2.1a,eq:2.1b,eq:2.1c,eq:2.1d,eq:2.1e,eq:2.1f,eq:2.1g,eq:2.1h,eq:2.1i,eq:2.1j,eq:2.1k}.

\begin{definition}[Weak solution]\label{def:weak_solution}
We say that
\[
    u_\e
    \in
    H^1(S;H_0^1(\Omega)^3),
    \qquad
    p_\e
    \in
    H^1(S;H^1(\omg)^*)
    \cap
    L^2(S;H^1(\omg)),
\]
is a weak solution to the microscopic problem if
\[
    \zeta_\e
    :=
    cp_\e
    +
    \alpha\dive u_\e
    +
    \eta\dive\partial_tu_\e
    \in
    H^1(S;H^1(\omg)^*)
\]
and the following conditions are satisfied.
For almost every \(t\in S\),
\begin{equation}\label{eq:weak_form_mechanic}
    \int_{\Omega}
        \mathcal{C}e(u_\e):e(\psi)
    \,\di{x}
    +
    \int_{\omg}
        \left(
            \delta\mathcal{C}^g e(\partial_tu_\e):e(\psi)
            -
            \alpha p_\e\dive\psi
        \right)
    \,\di{x}
    =
    \int_{\Omega}
        f_\e\cdot\psi
    \,\di{x}
\end{equation}
for every \(\psi\in H_0^1(\Omega)^3\).
Moreover,
\begin{multline}\label{eq:weak_visco_porous}
    -\int_S
        \left\langle
            \zeta_\e,
            \partial_t\varphi
        \right\rangle_{H^1(\omg)^*,H^1(\omg)}
    \,\di{t}
    +
    \int_S\int_{\omg}
        K\nabla p_\e\cdot\nabla\varphi
    \,\di{x}\di{t}
    \\
    =
    \int_S\int_{\omg}
        g_\e\varphi
    \,\di{x}\di{t}
    +
    \int_\omg
        \zeta_{\e,0}
        \varphi(0)\di x
\end{multline}
for every \(\varphi\in H^1(S;H^1(\omg))\) with \(\varphi(T)=0\).
Finally,
\[
    u_\e(0)=u_{\e,0}
    \quad\text{in }H_0^1(\Omega)^3,
    \qquad
    p_\e(0)=p_{\e,0}
    \quad\text{in }L^2(\Omega_\e^g).
\]
If $\eta=\alpha\delta$, we additionally require
\[
    w_\e:=u_\e+\delta\partial_tu_\e
    \in H^1\bigl(S;H^1(\Omega_\e^g)^3\bigr),\qquad  w_\e(0)=w_{\e,0}
    \quad\text{in }H^1(\Omega_\e^g)^3.
\]
\end{definition}

\begin{remark}
    Time regularity is a delicate issue in poroelastic systems, even in
    the purely poroelastic case $\delta=\eta=0$.
    At the natural energy level, one generally only expects the fluid content
    \[
        \zeta_\e=
        cp_\e+\alpha\dive u_\e+\eta\dive\partial_tu_\e
    \]
    to satisfy
    \(
        \zeta_\e
        \in
        H^1\bigl(S;H^1(\omg)^*\bigr).
    \)
    Consequently,
    \(
        \zeta_\e
        \in
        C\bigl([0,T];H^1(\omg)^*\bigr),
    \)
    and the natural initial condition associated with the mass balance is
    therefore
    \(
        \zeta_\e(0)=\zeta_{\e,0}
        \ \text{in }H^1(\omg)^*.
    \)
    This condition is already encoded in the time-integrated weak
    formulation of the mass-balance equation through the term involving
    $\zeta_{\e,0}$.
    At this level of regularity, neither $p_\e$ nor $u_\e$ needs to possess
    an individual time derivative, and their initial values need not be
    meaningful separately. For the corresponding weak theory of
    poroelasticity, we refer to \cite{showalter_diffusion_2000}.

    In the present work, we impose stronger assumptions on the data and
    obtain correspondingly stronger time regularity. In the case
    $\eta=0$, we have
    \[
        u_\e\in H^1\bigl(S;H_0^1(\Omega)^3\bigr),
        \qquad
        p_\e\in
        L^2\bigl(S;H^1(\omg)\bigr)
        \cap
        H^1\bigl(S;H^1(\omg)^*\bigr).
    \]
    Hence,
    \[
        u_\e\in C\bigl([0,T];H_0^1(\Omega)^3\bigr),
        \qquad
        p_\e\in C\bigl([0,T];L^2(\omg)\bigr),
    \]
    so that the initial conditions for displacement and pressure are
    well defined separately.

    In the case $\eta=\alpha\delta$, the higher-order estimates yield
    \[
        p_\e\in
        L^\infty\bigl(S;H^1(\omg)\bigr)
        \cap
        H^1\bigl(S;L^2(\omg)\bigr),
    \]
    and therefore again
    \[
        p_\e\in C\bigl([0,T];L^2(\omg)\bigr).
    \]
    Moreover, with
    \[
        w_\e:=u_\e+\delta\partial_tu_\e,
    \]
    the analysis provides sufficient time regularity to prescribe the
    corresponding initial value $w_\e(0)=w_{\e,0}$.

    The stronger assumptions on the force in the fibre phase are needed
    because the fibre deformation is governed by a quasistatic elastic
    equation and therefore does not gain time regularity from a viscous
    term. In the gel phase, by contrast, the Kelvin--Voigt contribution
    provides the additional smoothing in time.
\end{remark}

\section{Analysis of the microscale problem}\label{sec:analysis}
In this section, we tend to the analysis of the $\e$-problem and show that there is a unique weak solution in the sense of \cref{def:weak_solution} and that this solution satisfies certain $\e$-independent energy estimates.
We start by collecting some auxiliary results like extension operators and Korn inequalities.
In the analysis itself, which is outlined in \cref{sec:analysis_problem}, we distinguish between the two cases $\eta=0$ and $\eta=\delta\alpha$ as they ask for slightly different approaches.
The result for $\eta=0$ is presented in \cref{thm:existence_eta0} and the result for $\eta=\delta\alpha$ is given with \cref{lemma:existence_eta_alpha}.
The a priori estimates are given in \cref{lem:apriori} for both cases.

\subsection{Auxiliary results}\label{sec:auxiliary_results}
We start by collecting some auxiliary results which are used in the analysis of the visco-poroelasticit model.
\begin{lemma}[Extension operators]\label{lem:extension}
There exist linear extension operators
$\mathcal{E}_\e^f : H^1(\omf) \to H^1(\Omega)$ and
$\mathcal{E}_\e^g : H^1(\omg) \to H^1(\Omega)$
satisfying
\begin{alignat*}{2}
\|\mathcal{E}_\e^f\phi^f\|_{L^2(\Omega)} &\le C\|\phi^f\|_{L^2(\omf)} ,&\quad \|\mathcal{E}_\e^g\phi^g\|_{L^2(\Omega)} &\le C\|\phi^g\|_{L^2(\omg)},\\
\|\nabla(\mathcal{E}_\e^f\phi^f)\|_{L^2(\Omega)} &\le C\|\nabla\phi^f\|_{L^2(\omf)} ,&\qquad \|\nabla(\mathcal{E}_\e^g\phi^g)\|_{L^2(\Omega)} &\le C\|\nabla\phi^g\|_{L^2(\omg)},\\
\|e(\mathcal{E}_\e^f\phi^f)\|_{L^2(\Omega)} &\le C\|e(\phi^f)\|_{L^2(\omf)} ,&\quad \|e(\mathcal{E}_\e^g\phi^g)\|_{L^2(\Omega)} &\le C\|e(\phi^g)\|_{L^2(\omg)},
\end{alignat*}
for all $(\phi^f,\phi^g)\in H^1(\omf)\times H^1(\omg)$
for some $\e$-independent constant $C>0$.
\end{lemma}
\begin{proof}
First, extension operators must exist since $\omf$ and $\omg$ are Lipschitz domains.
The $\e$-uniform estimates are standard as long as we are away from the outer boundary (see \cite{ACP92}) but the estimates get more technical at the boundary because both domains are connected. 
We refer to \cite[Theorem 3.5 for the $H^1$ norms and Theorem 4.3 for the linearized gradients]{hopker_extension_2016}.
\end{proof}
\begin{lemma}[Uniform Korn inequalities]\label{lem:korn}
There exists a constant \(C>0\), independent of \(\e\), such that
\[
    \|v\|_{H^1(\Omega)}
    \leq
    C\|e(v)\|_{L^2(\Omega)}
    \qquad
    \text{for all }v\in H_0^1(\Omega)^3,
\]
and, for \(r\in\{f,g\}\),
\[
    \|v\|_{H^1(\Omega_\e^r)}
    \leq
    C\|e(v)\|_{L^2(\Omega_\e^r)}
    \qquad
    \text{for all }
    v\in H_{(\partial\Omega)_\e^r}^1(\Omega_\e^r)^3.
\]
\end{lemma}
\begin{proof}
On $\Omega$, this is the standard Korn inequality; the uniform version on subdomains follows via the uniform extension operators in \cref{lem:extension}:
\[
\|\nabla v\|_{L^2(\omf)}\le \|\nabla (\mathcal{E}_\e^fv)\|_{L^2(\Omega)}\lesssim \|e( \mathcal{E}_\e^fv)\|_{L^2(\Omega)}\lesssim\|e(v)\|_{L^2(\omf)}
\]
and the same argument over $\omg$.
\end{proof}

\begin{lemma}[Uniform Poincaré inequalities]
\label{lem:poincare}
There exists a constant \(C>0\), independent of \(\e\), such that
\[
    \|v\|_{H^1(\Omega)}
    \leq
    C\|\nabla v\|_{L^2(\Omega)}
    \qquad
    \text{for all }v\in H_0^1(\Omega),
\]
and, for \(r\in\{f,g\}\),
\[
    \|v\|_{H^1(\Omega_\e^r)}
    \leq
    C\|\nabla v\|_{L^2(\Omega_\e^r)}
    \qquad
    \text{for all }
    v\in H_{(\partial\Omega)_\e^r}^1(\Omega_\e^r).
\]
\end{lemma}
\begin{proof}
For $\Omega$ this is clear since the domain does not depend on $\e$. For $\omf$ and $\omg$, this follows again by using the uniform extensions from \cref{lem:extension}.
\end{proof}

\begin{lemma}[Discrete Grönwall inequality]
\label{lem:gronwall}
Let \(\xi\geq0\), and let \((u_n)_{n\geq1}\) and
\((a_n)_{n\geq1}\) be sequences of non-negative real numbers satisfying
\[
    u_n
    \leq
    \xi
    +
    \sum_{i=1}^{n-1}a_i u_i,
    \qquad n\geq1.
\]
Then
\[
    u_n
    \leq
    \xi\exp\left(\sum_{i=1}^{n-1}a_i\right),
    \qquad n\geq1.
\]
\end{lemma}
\begin{proof}
    We refer to \cite[Lemma 7.1]{thomee_galerkin_2006}.
\end{proof}

\begin{lemma}[Uniform Poincaré--Wirtinger inequality]
\label{lem:poincare_wirtinger}
There exists a constant \(C>0\), independent of \(\e\), such that
\[
    \left\|
        v-\fint_{\Omega_\e^r}v\,\di x
    \right\|_{L^2(\Omega_\e^r)}
    \leq
    C\|\nabla v\|_{L^2(\Omega_\e^r)}
\]
for every \(v\in H^1(\Omega_\e^r)\) and \(r\in\{f,g\}\).
Equivalently,
\[
    \|v\|_{L^2(\Omega_\e^r)}
    \leq
    C\|\nabla v\|_{L^2(\Omega_\e^r)}
\]
for every
\[
    v\in H_\diamond^1(\Omega_\e^r)
    :=
    \left\{
        v\in H^1(\Omega_\e^r):
        \fint_{\Omega_\e^r}v\,\di x=0
    \right\}.
\]
\end{lemma}
\begin{proof}
It suffices to consider
\(v\in H_\diamond^1(\Omega_\e^r)\).
Let
\[
    \widetilde v:=E_\e^r v\in H^1(\Omega)
\]
be the uniform extension from \(\Omega_\e^r\) to \(\Omega\), and denote
\[
    \overline{\widetilde v}
    :=
    \fint_\Omega \widetilde v\,\di x.
\]
The Poincaré--Wirtinger inequality on the fixed domain \(\Omega\) gives
\[
    \|\widetilde v-\overline{\widetilde v}\|_{L^2(\Omega)}
    \leq
    C\|\nabla\widetilde v\|_{L^2(\Omega)}.
\]
Since \(\widetilde v=v\) on \(\Omega_\e^r\) and
\(\fint_{\Omega_\e^r}v\,\di x=0\), we have
\[
    \overline{\widetilde v}
    =
    -\fint_{\Omega_\e^r}
        \bigl(\widetilde v-\overline{\widetilde v}\bigr)\,\di x.
\]
Consequently,
\[
    |\overline{\widetilde v}|
    \leq
    |\Omega_\e^r|^{-1/2}
    \|\widetilde v-\overline{\widetilde v}\|_{L^2(\Omega_\e^r)}
    \leq
    C\|\nabla\widetilde v\|_{L^2(\Omega)}.
\]
Here, we have used that
\[
    |\Omega_\e^r|=|Y^r|\,|\Omega|
\]
and hence \(|\Omega_\e^r|\) is bounded away from zero uniformly in
\(\e\). Therefore,
\[
\begin{aligned}
    \|v\|_{L^2(\Omega_\e^r)}
    &\leq
    \|\widetilde v-\overline{\widetilde v}\|_{L^2(\Omega)}
    +
    |\Omega|^{1/2}|\overline{\widetilde v}| \\
    &\leq
    C\|\nabla\widetilde v\|_{L^2(\Omega)}
    \leq
    C\|\nabla v\|_{L^2(\Omega_\e^r)},
\end{aligned}
\]
where the last inequality follows from the uniform estimate for the
extension operator.
\end{proof}

\subsection{Analysis of the visco-poroelasticity problem}\label{sec:analysis_problem}

For the ease of notation, we suppress the parameter $\e$ in the functions an the definition of the operators below, i.e., we write $u$ instead of $u_\e$ and $\mathcal{E}$ instead of $\mathcal{E}_\e$.

We start by introducing the linear and continuous operators $\mathcal{E}^r\colon H_0^1(\Omega)^3\to H^{-1}(\Omega)^3$ $(r=f,g)$ defined via
\[
\langle \mathcal{E}^ru,\psi\rangle_{H^{-1}(\Omega)^3}=\int_{\Omega_\e^r}C^re(u):e(\psi)\di{x}\quad (\psi\in H^1_0(\Omega)^3).
\]
as well as the overall elasticity operator
\[
\mathcal{E}\colon H_0^1(\Omega)^3\to H^{-1}(\Omega)^3,\quad \mathcal{E}=\mathcal{E}^f+\mathcal{E}^g.
\]
As $C^f,C^g$ are positive definite, $\mathcal{E}$ is continuously invertible as a function from $H_0^1(\Omega)^3$ to $H^{-1}(\Omega)^3$.
Similarly, $\mathcal{E}^r$ are continuously invertible when viewed as operators for $H_0^1(\Omega_\e^r)\to H^{-1}(\Omega_\e^r)$.
We also introdue
\begin{alignat*}{2}
\nabla^g&\colon L^2(\Omega_\e^g)\to H^{-1}(\Omega)^3,&\quad\langle\nabla^g p,\psi\rangle_{H^{-1}(\Omega)^3}&=-\int_\omg p\dive\psi\di{x},\\
\dive^g
&\colon H_0^1(\Omega)^3\to L^2(\Omega_\e^g),&\qquad\langle\dive^g u,\phi\rangle&:=\int_{\Omega_\e^g}\dive u\,\phi\,\di{x},\\
\mathcal{K}&\colon H^1(\Omega_\e^g)\to H^1(\Omega_\e^g)^*,&\quad \langle\mathcal{K}p,\phi\rangle_{H^1(\Omega_\e^g)^*}&=\int_{\omg}K\nabla p\cdot\nabla\phi\di{x}.
\end{alignat*}
Using these operators, we can write \cref{eq:weak_form_mechanic,eq:weak_visco_porous} abstractly as
\begin{subequations}\label{system:abstract_operator_problem}
\begin{align}
    \mathcal{E}u+\delta\mathcal{E}^g\partial_tu+\alpha\nabla^gp&=f\quad\text{in}\ H^{-1}(\Omega)^3,\ t\in S,\label{system:abstract_operator_problem:VE}\\
    \partial_t(cp+\dive^g (\alpha u+\eta\partial_tu))+\mathcal{K}p&=g\quad\text{in}\ H^1(\omg)^*,\ t\in S.\label{system:abstract_operator_problem:PE}
\end{align}
\end{subequations}
Please note: for any $p\in L^2(\Omega_\e^g)$ and any $f\in L^2(\Omega)$, there is a unique function $v\in H^1_0(\Omega)^3$ solving the purely elastic problem
\[
\mathcal{E}v+\alpha\nabla^g p=f\quad\text{in}\ H^{-1}(\Omega)^3
\]
or alternatively
\begin{equation*}
v=\mathcal{E}^{-1}f-\alpha\mathcal{E}^{-1}\nabla^gp\quad\text{in}\ H_0^1(\Omega)^3.
\end{equation*}
This can be naturally transferred to the quasi-static case:  for any $p\in L^2(S\times\Omega_\e^g)$ and any $f\in L^2(S\times\Omega)$, $v\in L^2(S;H^1_0(\Omega))$ satisfies
\begin{equation}\label{eq:solution_v}
v(t)=\mathcal{E}^{-1}f(t)-\alpha\mathcal{E}^{-1}\nabla^gp(t)\quad\text{in}\ H_0^1(\Omega)^3,\ t\in S.
\end{equation}

We are tempted to go to the viscoelastic case from here via $v=u+\delta\partial_tu$ by solving the resulting ODE (see, e.g., \cite[Section 3.2]{BMW23} where this strategy is outlined in detail for a viscoelasticity problem).
However, this is not possible in our case due to the jump between elastic and viscoelastic material behavior at $\Gamma_\e$.
Note that the function 
\[
v=\begin{cases}
    u \quad&\text{in}\ S\times\omf\\
    u+\delta\partial_tu\quad&\text{in}\ S\times\omg
\end{cases}
\]
may jump across the interface $\Gamma_\e$;
consequently, $v\notin H^1(\Omega)$ is possible and $v$ can not be the function given in \cref{eq:solution_v}.
We resolve this issue by following a Rothe-type approach where we approximate the time derivative by difference quotients; similar to what has been done in \cite{BGSW16}.

To that end, for $N\in\mathbb{N}$, let $h=T/N>0$ denote the (uniform) time step size and $t^i=ih$ the time steps for $i=0,..., N$.
We introduce 
\[
f^{i}(x):=\frac{1}{h}\int_{t^i-h}^{t^i}f(\tau,x)\di{\tau},\quad g^{i}(x):=\frac{1}{h}\int_{t^i-h}^{t^i} g(\tau,x)\di{\tau}
\]
and see that $f^i\in L^2(\Omega)^3$ and $g^i\in L^2(\omg)$.
Moreover, the piecewise constant interpolations
\[
f_{h}(t,x):=\sum_{i=1}^Nf^{i}(x)\mathds{1}_{(t^i-h,t^i)}(t),\quad g_{h}(t,x):=\sum_{i=1}^Ng^{i}(x)\mathds{1}_{(t^i-h,t^i)}(t)
\]
converge strongly in $L^2$ for $h\to0$, i.e., $f_h\to f$ in $L^2(S\times\Omega)$ and $g_h\to g$ in $L^2(S\times\omg)$.
We consider slightly different Rothe discretizations for the two cases  which are tailored to their specific coupling structure.
We start with the easier case where $\eta=0$.

\subsubsection{Case $(i)$, $\eta=0$}\label{sssec:eta0}
In this case, we look at the following time discrete problem $(P_h)$ (after multiplication with $h$ and $h^2$, repectively): Set $u^0=u_0$ and $p^0=p_0$.
Then, find $(u^i)_{i=1}^N\subset H_0^1(\Omega)^3$ and $(p^{i})_{i=1}^{N-1}\subset H^1(\omg)$ such that
\begin{subequations}\label{eq:rothe_discretization}
\begin{empheq}[left=(P_h)\empheqlbrace]{alignat=2}
h\mathcal{E}u^{i+1}+\delta\mathcal{E}^g(u^{i+1}-u^{i})+h\alpha\nabla^gp^{i}&=hf^{i+1}&\quad&\text{in}\ H^{-1}(\Omega)^3,\label{eq:rothe_discretization:mechanics}\\
ch(p^{i}-p^{i-1})+h\alpha\dive^g(u^{i+1}-u^{i})+h^2\mathcal{K}p^{i}
&=h^2g^{i}&\quad&\text{in}\ H^{1}(\omg)^*-.\label{eq:rothe_discretization:visco_porous}
\end{empheq}
\end{subequations}
\begin{remark}
    The scheme \((P_h)\) is staggered in time: the displacement and
    pressure are associated with shifted time levels. After determining
    \(u^1\) from the initial data, each subsequent step consists of a
    coupled problem for \((u^{i+1},p^i)\), with
    \((u^i,p^{i-1})\) given.
\end{remark}

\begin{lemma}[Well-posedness of the discrete problem $(P_h)$]
\label{lemma:existence_eta0_discrete}
The discrete problem \((P_h)\) admits a unique family of solutions
\[
    (u^i)_{i=1}^N\subset H_0^1(\Omega)^3,
    \qquad
    (p^i)_{i=1}^{N-1}\subset H^1(\omg),
\]
Moreover, there exists a constant \(C>0\), independent of \(h\), such that
\begin{equation*}
    \max_{0\leq i\leq N-1}\left(
        \|u^{i+1}\|_{H^1(\Omega)}^2
    +
        \|p^i\|_{L^2(\omg)}^2\right)
    \\
    +
    h\sum_{i=1}^{N-1}
    \left(
        \|\nabla p^i\|_{L^2(\omg)}^2
        +
        \delta
        \|e(\derh u^{i+1})\|_{L^2(\omg)}^2
    \right)
    \leq C.
\end{equation*}
In addition,
\[
    \|u^1-u^0\|_{H^1(\omg)}\longrightarrow0
    \qquad\text{as }h\to0.
\]
\end{lemma}
\begin{proof}
    \textit{\underline{Step 1: Well-posedness of the discrete system.}}
    The existence of solutions to the discretized problem $(P_h)$ follows by induction:
    First, we immediately get $u^{1}\in H_0^1(\Omega)^3$ from \cref{eq:rothe_discretization:mechanics} using $u^0=u_0$ and $p^0=p_0$ with the Lax-Milgram lemma.
    Now, let \(1\leq i\leq N-1\), and assume that
    \[
        (u^i,p^{i-1})
        \in
        H_0^1(\Omega)^3
        \times H^1(\omg)
    \]
    is already known.
    We introduce the operator $\mathcal{L}\colon H_0^1(\Omega)^3\times H^1(\omg)\to H^{-1}(\Omega)^3\times H^1(\omg)^*$ by
    \[
    \mathcal{L}=\begin{pmatrix}
          h\mathcal{E}+\delta\mathcal{E}^g  &\alpha h\nabla^g\\
          h\alpha\dive^g& ch\ID +h^2\mathcal{K}
        \end{pmatrix}
    \]
    and note that $\mathcal{L}$ is linear and continuous.
    Moreover, since
    \[
        \langle\nabla^g p,u\rangle
        =
        -\langle\dive^g u,p\rangle,
    \]
    the coupling terms cancel, and therefore
    \[
    \begin{aligned}
        \left\langle
            \mathcal L
            \binom{u}{p},
            \binom{u}{p}
        \right\rangle
        &=
        h\langle\mathcal Eu,u\rangle
        +\delta\langle\mathcal E^gu,u\rangle
        +ch\|p\|_{L^2(\omg)}^2
        +h^2\langle\mathcal Kp,p\rangle.
    \end{aligned}
    \]
    Hence, \(\mathcal L\) is coercive.
    Since the corresponding right-hand side depends only on the already known quantities \(u^i\), \(p^{i-1}\), and on the data \(f^{i+1}\), \(g^i\), the Lax--Milgram lemma yields a unique solution \((u^{i+1},p^i)\).

    \textit{\underline{Step 2: Energy estimates.}}
    We first consider \(u^1\). Testing
    \cref{eq:rothe_discretization:mechanics} for \(i=0\) with
    \(u^1-u^0\) and using the polarization identity, we obtain
    \begin{multline*}
        \frac{h}{2}\Bigl(
            \langle\mathcal Eu^1,u^1\rangle
            -\langle\mathcal Eu^0,u^0\rangle
            +\langle\mathcal E(u^1-u^0),u^1-u^0\rangle
        \Bigr)
        +\delta
        \langle
            \mathcal E^g(u^1-u^0),u^1-u^0
        \rangle
        \\
        =
        h(f^1,u^1-u^0)_{L^2(\omf)}
        +h(f^1,u^1-u^0)_{L^2(\omg)}
        -h\alpha
        \langle\nabla^gp^0,u^1-u^0\rangle .
    \end{multline*}
    By coercivity, Korn's inequality, and Young's inequality,
    \begin{align*}
        h\left|
            (f^1,u^1-u^0)_{L^2(\omf)}
        \right|
        &\leq
        \frac{h}{8}
        \langle
            \mathcal E(u^1-u^0),u^1-u^0
        \rangle
        +Ch\|f^1\|_{L^2(\omf)}^2,
        \\
        h\left|
            (f^1,u^1-u^0)_{L^2(\omg)}
        \right|
        &\leq
        \frac{\delta}{4}
        \langle
            \mathcal E^g(u^1-u^0),u^1-u^0
        \rangle
        +C\frac{h^2}{\delta}\|f^1\|_{L^2(\omg)}^2,
        \\
        h\alpha\left|
            \langle\nabla^gp^0,u^1-u^0\rangle
        \right|
        &\leq
        \frac{\delta}{4}
        \langle
            \mathcal E^g(u^1-u^0),u^1-u^0
        \rangle
        +C\frac{h^2}{\delta}\|p^0\|_{L^2(\omg)}^2.
    \end{align*}
    After absorbing these terms and dividing by \(h\), we obtain
    \begin{equation}\label{eq:estimate_u1}
    \begin{aligned}
        \|u^1\|_{H^1(\Omega)}^2
        \leq C\Bigl(
            \|u_0\|_{H^1(\Omega)}^2
            +\|f^1\|_{L^2(\omf)}^2
            +\frac{h}{\delta}\|f^1\|_{L^2(\omg)}^2
            +\frac{h}{\delta}\|p_0\|_{L^2(\omg)}^2
        \Bigr).
    \end{aligned}
    \end{equation}
    Moreover, the preceding estimate also yields
    \[
        \|u^1-u^0\|_{H^1(\omg)}^2\leq C_\delta h.
    \]
    Consequently,
    \[
        u^1-u^0\longrightarrow0
        \quad\text{in }H^1(\omg)^3,
    \]
    and hence
    \[
        cp^0+\alpha\dive^g u^1
        \longrightarrow
        cp_0+\alpha\dive^g u_0
        \quad\text{in }H^1(\omg)^*
    \]
    for $h\to0$.

    For the estimates of $(u^{i+1},p^i)$ for $i\ge1$, we set
    \[
        \derh u^i:=h^{-1}(u^i-u^{i-1}),
        \qquad i=1,\ldots,N,
    \]
    and
    \[
        \derh p^i:=h^{-1}(p^i-p^{i-1}),
        \qquad i=1,\ldots,N-1.
    \]
    For \(i=1,\ldots,N-1\), we test \eqref{eq:rothe_discretization} with \(\derh u^{i+1}\) and \(h^{-1}p^i\).
    This yields
    \begin{align*}
        h\langle\mathcal Eu^{i+1},\derh u^{i+1}\rangle_{H^{-1}(\Omega)}
        &+\delta h
        \langle\mathcal E^g\derh u^{i+1},
        \derh u^{i+1}\rangle_{H^{-1}(\Omega)}
        \\
        &+h\alpha
        \langle\nabla^gp^i,\derh u^{i+1}\rangle_{H^{-1}(\Omega)}
        =
        h(f^{i+1},\derh u^{i+1})_{L^2(\Omega)},
        \\
        ch(\derh p^i,p^i)_{L^2(\omg)}
        &+h\alpha
        (\dive\derh u^{i+1},p^i)_{L^2(\omg)}
        \\
        &+h\langle\mathcal Kp^i,p^i\rangle_{H^1(\omg)^*}
        =
        h(g^i,p^i)_{L^2(\omg)}.
    \end{align*}
    Noting that
    \[
    \langle\nabla^gp^i,\derh u^{i+1}\rangle_{H^{-1}(\Omega)}
    =
    -(\dive\derh u^{i+1},p^i)_{L^2(\omg)},  
    \]
    and summing the weak forms, we arrive at
    \begin{multline*}
    ch(\derh p^i,p^i)_{L^2(\omg)}
    +h\langle\mathcal Eu^{i+1},\derh u^{i+1}\rangle_{H^{-1}(\Omega)}
    \\
    +\delta h
    \langle\mathcal E^g\derh u^{i+1},
    \derh u^{i+1}\rangle_{H^{-1}(\Omega)}
    +h\langle\mathcal Kp^i,p^i\rangle_{H^1(\omg)^*}
    \\
    =
    h(g^i,p^i)_{L^2(\omg)}
    +h(f^{i+1},\derh u^{i+1})_{L^2(\Omega)}.
    \end{multline*}
    Using the product rule for difference quotients (and its analogue for the pressure), 
    \begin{multline}\label{eq:product_rule_difference}
        2h\langle\mathcal Eu^{i+1},\derh u^{i+1}\rangle_{H^{-1}(\Omega)}\\
        =        \langle\mathcal Eu^{i+1},u^{i+1}\rangle_{H^{-1}(\Omega)}
        -\langle\mathcal Eu^i,u^i\rangle_{H^{-1}(\Omega)}
        +h^2
        \langle\mathcal E\derh u^{i+1},
        \derh u^{i+1}\rangle_{H^{-1}(\Omega)},
    \end{multline}
    we are led to
    \begin{multline}\label{eq:estimates_mechanics}
    c\left(
        \|p^i\|_{L^2(\omg)}^2
        -\|p^{i-1}\|_{L^2(\omg)}^2
        +h^2\|\derh p^i\|_{L^2(\omg)}^2
    \right)
    \\
    +\langle\mathcal Eu^{i+1},u^{i+1}\rangle_{H^{-1}(\Omega)}
    -\langle\mathcal Eu^i,u^i\rangle_{H^{-1}(\Omega)}
    +h^2
    \langle\mathcal E\derh u^{i+1},
    \derh u^{i+1}\rangle_{H^{-1}(\Omega)}
    \\
    +2\delta h
    \langle\mathcal E^g\derh u^{i+1},
    \derh u^{i+1}\rangle_{H^{-1}(\Omega)}
    +2h\langle\mathcal Kp^i,p^i\rangle_{H^1(\omg)^*}
    \\
    =
    2h(g^i,p^i)_{L^2(\omg)}
    +2h(f^{i+1},\derh u^{i+1})_{L^2(\Omega)}.
    \end{multline}
    Summing \cref{eq:estimates_mechanics} over \(i=1,\ldots,j\), where
    \(1\leq j\leq N-1\), and using the positivity assumptions on the
    stiffness tensors and the permeability, together with Korn's
    inequalities on \(\Omega\) and \(\omg\), see \cref{lem:korn}, we obtain
    \begin{multline*}
        \|p^j\|_{L^2(\omg)}^2
        -\|p^0\|_{L^2(\omg)}^2
        +\|\nabla u^{j+1}\|_{L^2(\Omega)}^2
        -\|\nabla u^1\|_{L^2(\Omega)}^2
        \\
        +h\sum_{i=1}^j
        \Bigl(
            h\|\derh p^i\|_{L^2(\omg)}^2
            +h\|\nabla\derh u^{i+1}\|_{L^2(\Omega)}^2
            +\delta\|\nabla\derh u^{i+1}\|_{L^2(\omg)}^2
            +\|\nabla p^i\|_{L^2(\omg)}^2
        \Bigr)
        \\
        \leq
        C\left(h\sum_{i=1}^j
            \left|(g^i,p^i)_{L^2(\omg)}\right|
            +
            h\left|\sum_{i=1}^j(f^{i+1},\derh u^{i+1})_{L^2(\Omega)}\right|
        \right).
    \end{multline*}
    For the term involving \(f\), we split the integral over the two
    subdomains and obtain
    \begin{multline*}
        h\sum_{i=1}^j
        (f^{i+1},\derh u^{i+1})_{L^2(\Omega)}=
        h\sum_{i=1}^j
        (f^{i+1},\derh u^{i+1})_{L^2(\omf)}
        +
        h\sum_{i=1}^j
        (f^{i+1},\derh u^{i+1})_{L^2(\omg)}\\
        =
        (f^{j+1},u^{j+1})_{L^2(\omf)}
        -(f^2,u^1)_{L^2(\omf)}
        -\sum_{i=1}^{j-1}
        (f^{i+2}-f^{i+1},u^{i+1})_{L^2(\omf)}
        +
        h\sum_{i=1}^j
        (f^{i+1},\derh u^{i+1})_{L^2(\omg)}
        \end{multline*}
        and estimate
    \[
        \left|
            (f^{i+2}-f^{i+1},u^{i+1})_{L^2(\omf)}
        \right|
        \leq
        \sqrt{h}\,
        \|\partial_t f\|_{L^2((t^{i+1},t^{i+2})\times\omf)}
        \|u^{i+1}\|_{L^2(\omf)}.
    \]
    Using the continuous embedding
    \[
        H^1((a,b);L^2(\omf))
        \hookrightarrow
        C([a,b];L^2(\omf)),
    \]
    we obtain
    \[
    \begin{aligned}
        \|f^{j+1}\|_{L^2(\omf)}^2
        +\|f^2\|_{L^2(\omf)}^2
        &\leq
        2\|f\|_{C([0,t^{j+1}];L^2(\omf))}^2
        \\
        &\leq C_T\|f\|_{H^1(S;L^2(\omf))}^2.
    \end{aligned}
    \]
    Consequently, the discrete integration-by-parts formula and Young's
    inequality yield
    \begin{multline*}
        h\left|\sum_{i=1}^j
            (f^{i+1},\derh u^{i+1})_{L^2(\Omega)}
        \right|\\
        \leq \gamma\left(\|u^{j+1}\|_{L^2(\omf)}^2
            +\|u^1\|_{L^2(\omf)}^2+h
        \sum_{i=1}^{j-1}
        \|u^{i+1}\|_{L^2(\omf)}^2
        +h
        \sum_{i=1}^j
        \|\derh u^{i+1}\|_{L^2(\omg)}^2\right)
        \\
        +C_\gamma\left(C_T\|f\|_{H^1(S;L^2(\omf))}^2+
        \sum_{i=1}^{j-1}
        \|\partial_t f\|_{L^2((t^{i+1},t^{i+2})\times\omf)}^2+h
        \sum_{i=1}^j
        \|f^{i+1}\|_{L^2(\omg)}^2\right).
    \end{multline*}
    For $\gamma$ sufficiently small, we get
    \begin{multline}\label{eq:estimates_with_raw_pressure}
    \|p^j\|_{L^2(\omg)}^2
    +
    \|\nabla u^{j+1}\|_{L^2(\Omega)}^2
    \\
    +h\sum_{i=1}^j
    \Bigl(
        h\|\derh p^i\|_{L^2(\omg)}^2
        +h\|\nabla\derh u^{i+1}\|_{L^2(\Omega)}^2
        +\delta\|\nabla\derh u^{i+1}\|_{L^2(\omg)}^2
        +\|\nabla p^i\|_{L^2(\omg)}^2
    \Bigr)
    \\
    \leq
    C\Bigl(
        \|g\|_{L^2((0,t^{j+1})\times\omg)}^2
        +
        \|f\|_{H^1(0,t^{j+1};L^2(\omf))}^2
        +
        \|f\|_{L^2((0,t^{j+1})\times\omg)}^2
        +
        \|p^0\|_{L^2(\omg)}^2
        +
        \|u^1\|_{H^1(\Omega)}^2
    \Bigr)
    \\
    +Ch\sum_{i=1}^j
    \left(
        \|p^i\|_{L^2(\omg)}^2
        +
        \|u^{i+1}\|_{L^2(\Omega)}^2
    \right),
    \end{multline}
    By the discrete Grönwall inequality, see \cref{lem:gronwall}, and the estimates for $u^1$ \eqref{eq:estimate_u1}, we obtain
    \begin{multline*}
        \|p^j\|_{L^2(\omg)}^2
        +
        \|u^{j+1}\|_{H^1(\Omega)}^2
        \\
        +h\sum_{i=1}^j
        \Bigl(
            \|\nabla p^i\|_{L^2(\omg)}^2
            +
            \delta\|e(\derh u^{i+1})\|_{L^2(\omg)}^2
            +h\|\derh p^i\|_{L^2(\omg)}^2
            +h\|\derh u^{i+1}\|_{H^1(\Omega)}^2
        \Bigr)
        \\
        \leq
        C\exp(Ct^{j+1})
        \Bigl(
            \|g\|_{L^2(S\times\omg)}^2
            +
            \|f\|_{H^1(S;L^2(\omf))}^2
            +
            \|f\|_{L^2(S\times\omg)}^2
            +
            \|p_0\|_{L^2(\omg)}^2
            +
            \|u_0\|_{H^1(\Omega)}^2
        \Bigr).
    \end{multline*}
\end{proof}

\begin{theorem}[Well-posedness for $\eta=0$]\label{thm:existence_eta0}
There are unique functions 
\[
u\in L^\infty(S;H_0^1(\Omega)^3)
      \cap H^1(S;H_0^1(\Omega)^3),\quad p\in L^\infty(S;L^2(\omg))
      \cap L^2(S;H^1(\omg))
      \cap H^1(S;H^1(\omg)^*)
\]
solving System~\eqref{system:abstract_operator_problem} in the sense of \cref{def:weak_solution}.
\end{theorem}
\begin{proof}
We introduce the piecewise constant interpolants associated with the
discrete solutions from \cref{lemma:existence_eta0_discrete}:
\[
    u_h(t)
    :=
    \sum_{i=0}^{N-1}
    u^{i+1}\mathds{I}_{(t^i,t^{i+1}]}(t),
\]
and\footnote{Here, we introduce two pressure interpolants because \(p^0=p_0\) is only assumed to belong to \(L^2(\omg)\), whereas \(p^i\in H^1(\omg)\) for \(i\geq1\). Thus, \(p_h^-\) retains the correct initial value, while \(p_h^+\) is uniformly bounded in \(L^2(S;H^1(\omg))\).
Moreover,
\[
    \|p_h^+-p_h^-\|_{L^2(S;L^2(\omg))}\to0
    \qquad\text{as }h\to0.
\]}
\[
    p_h^-(t)
    :=
    \sum_{i=0}^{N-1}
    p^i\mathds{I}_{(t^i,t^{i+1}]}(t),
    \qquad
    p_h^+(t)
    :=
    \sum_{i=1}^{N-1}
    p^i\mathds{I}_{(t^{i-1},t^i]}(t)
    +
    p^{N-1}\mathds{I}_{(t^{N-1},T]}(t).
\]
Moreover, we define
\[
    \derh u_h(t)
    :=
    \sum_{i=0}^{N-1}
    \derh u^{i+1}
    \mathds{I}_{(t^i,t^{i+1}]}(t).
\]
The estimate for the initial step and the discrete energy estimate imply
\[
    \|u_h\|_{L^\infty(S;H_0^1(\Omega)^3)}
    +
    \|\derh u_h\|_{L^2(S;H^1(\omg)^3)}
    +
    \|p_h^-\|_{L^\infty(S;L^2(\omg))}
    +
    \|p_h^+\|_{L^2(S;H^1(\omg))}
    \leq C.
\]
where $C>0$ does not depend on $h$.

With the uniform estimates derived above, we may extract a subsequence,
not relabeled, and find functions
\[
    u\in L^\infty(S;H_0^1(\Omega)^3),
    \qquad
    p\in L^\infty(S;L^2(\omg))
       \cap L^2(S;H^1(\omg)),
\]
such that
\[
    u_h \weakstar u
    \quad\text{in }L^\infty(S;H_0^1(\Omega)^3),
\]
and
\[
    p_h^-\weakstar p
    \quad\text{in }L^\infty(S;L^2(\omg)),
    \qquad
    p_h^+\rightharpoonup p
    \quad\text{in }L^2(S;H^1(\omg)).
\]
To identify the discrete time derivative, we introduce the piecewise
affine interpolant
\[
    \widehat u_h(t)
    :=
    u^i+(t-t^i)\derh u^{i+1},
    \qquad t\in(t^i,t^{i+1}].
\]
so that
\[
    \partial_t\widehat u_h
    =
    \derh u^{i+1}
    \qquad\text{on }(t^i,t^{i+1}].
\]
The energy estimate implies
\[
    \partial_t\widehat u_h
    \rightharpoonup v
    \quad\text{in }L^2(S;H^1(\omg)^3)
\]
for some \(v\). Moreover,
\[
    \|u_h-\widehat u_h\|_{L^2(S;H^1(\omg))}
    \leq
    h\|\partial_t\widehat u_h\|_{L^2(S;H^1(\omg))}
    \longrightarrow0.
\]
Consequently, \(v=\partial_tu\) on \(\omg\), and hence
\[
    u\in H^1(S;H^1(\omg)^3).
\]
Since the affine interpolant satisfies
\[
    \widehat u_h(0)=u^0=u_0
\]
and
\[
    \widehat u_h
    \rightharpoonup u
    \qquad\text{in }H^1(S;H^1(\omg)^3),
\]
the continuity of the linear trace operator at \(t=0\) yields
\[
    u(0)=u_0
    \qquad\text{in }H^1(\omg)^3.
\]
After division by \(h\), the discrete mechanical equation can be
written on every interval \((t^i,t^{i+1}]\), \(i=1,\ldots,N-1\), as
\[
    \mathcal E u_h
    +\delta\mathcal E^g\partial_t\widehat u_h
    +\alpha\nabla^g p_h
    =
    f_h,
\]
where \(f_h(t)=f^{i+1}\) on \((t^i,t^{i+1}]\).
Since
\[
    f_h\longrightarrow f
    \quad\text{in }L^2(S;H^{-1}(\Omega)^3),
\]
we may test this identity with arbitrary
\(\psi\in L^2(S;H_0^1(\Omega)^3)\) and pass to the limit in each term.
Using the weak convergences above and the continuity of
\(\mathcal E\), \(\mathcal E^g\), and \(\nabla^g\), we obtain
\[
    \int_S
    \left\langle
        \mathcal Eu
        +\delta\mathcal E^g\partial_tu
        +\alpha\nabla^gp-f,
        \psi
    \right\rangle
    \,\di t
    =0.
\]
Therefore,
\[
    \mathcal Eu
    +\delta\mathcal E^g\partial_tu
    +\alpha\nabla^gp
    =
    f
    \quad\text{in }L^2(S;H^{-1}(\Omega)^3).
\]

    To improve the time regularity of \(u\) in the elastic part, we first
    restrict the mechanical equation to test functions supported in
    \(\omf\). Since the operators \(\mathcal E^g\) and \(\nabla^g\) vanish
    on such test functions, we obtain
    \[
        \mathcal E^f u=f
        \quad\text{in }L^2\bigl(S;H^{-1}(\omf)^3\bigr).
    \]
    Moreover, since
    \[
        u\in H^1\bigl(S;H^1(\omg)^3\bigr),
    \]
    the trace theorem yields
    \[
        u|_{\Gamma_\e}
        \in H^1\bigl(S;H^{1/2}(\Gamma_\e)^3\bigr),
        \qquad
        \partial_t(u|_{\Gamma_\e})
        =
        (\partial_tu)|_{\Gamma_\e}.
    \]
    Since, for fixed $\e>0$, the elliptic problem in $\Omega_\e^f$ depends continuously and linearly on the volume force and on the Dirichlet trace prescribed on $\Gamma_\e$, the regularity
    \[
        f^f\in H^1\bigl(S;L^2(\Omega_\e^f)^3\bigr),
        \qquad
        u|_{\Gamma_\e}\in
        H^1\bigl(S;H^{1/2}(\Gamma_\e)^3\bigr),
    \]
    implies
    \[
        u|_{\Omega_\e^f}
        \in H^1\bigl(S;H^1(\Omega_\e^f)^3\bigr).
    \]
    Since the traces of the time derivatives from $\Omega_\e^f$ and
    $\Omega_\e^g$ agree on $\Gamma_\e$, the two restrictions can be glued,
    and hence
    \[
        u\in H^1\bigl(S;H_0^1(\Omega)^3\bigr).
    \]
    The identity
    \[
        u(0)=u_0
        \quad\text{in }H^1(\omg)^3
    \]
    has already been obtained from the affine interpolants.
    By the quasistatic equilibrium in the fibre phase, the compatibility condition for $u_0$, and the continuity of the corresponding elliptic solution operator, the same identity holds in $\omf$.
    Hence,
    \[
        u(0)=u_0
        \quad\text{in }H_0^1(\Omega)^3.
    \]

    Finally, we pass to the limit in the discrete mass-balance equation.
    For this purpose, we set
    \[
        q^i:=cp^i+\alpha\dive^g u^{i+1},
        \qquad i=0,\ldots,N-1.
    \]
    Due to the staggered time discretisation, the initial discrete fluid content is given by
    \[
        q^0=cp_0+\alpha\dive^g u^1,
    \]
    and therefore does not coincide exactly with the prescribed initial fluid content
    \[
        \zeta_0=cp_0+\alpha\dive^g u_0.
    \]
    However, the estimate for the initial step yields
    \[
        u^1-u_0\longrightarrow0
        \quad\text{in }H^1(\omg)^3,
    \]
    and hence
    \[
        q^0\longrightarrow\zeta_0
        \quad\text{in }L^2(\omg)
    \]
    as $h\to0$.
    Then, after division by \(h^2\),
    \cref{eq:rothe_discretization:visco_porous} can be written as
    \[
        \frac{q^i-q^{i-1}}{h}
        +\mathcal Kp^i
        =
        g^i
        \quad\text{in }H^1(\omg)^*,
        \qquad i=1,\ldots,N-1.
    \]
    We take \(\varphi\in C_c^1([0,T);H^1(\omg))\) and test the discrete equation with suitable time averages \(\varphi^i\) of \(\varphi\).
    After summing over \(i\) and applying discrete
    summation by parts, we obtain
    \begin{align*}
        -\sum_{i=1}^{N-2}
        \left\langle
            q^i,\varphi^{i+1}-\varphi^i
        \right\rangle+
        h\sum_{i=1}^{N-1}
        \left\langle
            \mathcal Kp^i,\varphi^i
        \right\rangle=
        h\sum_{i=1}^{N-1}
        (g^i,\varphi^i)_{L^2(\omg)}
        +
        \left\langle q^0,\varphi^1\right\rangle .
    \end{align*}
    Using the convergences established above, the strong convergence of the piecewise constant approximation \(g_h\) to \(g\), and
    \[
        \frac{\varphi^{i+1}-\varphi^i}{h}
        \longrightarrow
        \partial_t\varphi
        \quad\text{uniformly in time},
    \]
    we may pass to the limit in the discrete weak formulation.
    \begin{align*}
        -\int_0^T
        \left\langle
            cp+\alpha\dive^g u,\partial_t\varphi
        \right\rangle
        \di{t}
        +
        \int_0^T
        \left\langle
            \mathcal Kp,\varphi
        \right\rangle
        \di{t}
        =
        \int_0^T
        (g,\varphi)_{L^2(\omg)}
        \di{t}
        +
        (
            \zeta_0,\varphi(0))_{L^2(\omg)} .
    \end{align*}
    The preceding identity is the weak formulation of
\[
    \partial_t\bigl(cp+\alpha\dive^g u\bigr)
    +\mathcal Kp
    =
    g
    \quad\text{in }L^2\bigl(S;H^1(\omg)^*\bigr),
\]
together with the initial condition
\[
    \bigl(cp+\alpha\dive^g u\bigr)(0)
    =
    \zeta_0
    \quad\text{in }H^1(\omg)^*.
\]
Since
\[
    \dive^g\partial_tu
    \in L^2\bigl(S;H^1(\omg)^*\bigr),
\]
it follows that
\[
    \partial_tp
    =
    \frac1c
    \left(
        g-\mathcal Kp-\alpha\dive^g\partial_tu
    \right)
    \in
    L^2\bigl(S;H^1(\omg)^*\bigr).
\]
Together with
\[
    p\in L^2\bigl(S;H^1(\omg)\bigr),
\]
the Lions--Magenes lemma yields
\[
    p\in C\bigl([0,T];L^2(\omg)\bigr).
\]
Moreover,
\[
    u\in C\bigl([0,T];H_0^1(\Omega)^3\bigr)
\]
and $u(0)=u_0$. Hence both sides of
\[
    \bigl(cp+\alpha\dive^g u\bigr)(0)=\zeta_0
\]
belong to $L^2(\omg)$, so the identity also holds in
$L^2(\omg)$. Since
\[
    \zeta_0=cp_0+\alpha\dive^g u_0,
\]
we conclude that
\[
    p(0)=p_0
    \quad\text{in }L^2(\omg).
\]

    Finally, uniqueness follows by applying the standard energy estimate to the difference of two solutions.
    Testing the mechanical equation with \(\partial_t u\) and the mass-balance equation with \(p\) (which can be justified by a standard time-regularisation argument), the coupling terms cancel, and coercivity implies that the difference vanishes.
    \end{proof}

\subsubsection{Case $(ii)$, $\eta=\delta\alpha$}\label{sssec:eta_alpha}

To exploit the special coupling structure of the case
\(\eta=\delta\alpha\), we introduce the auxiliary function
\[
    w:=u+\delta\partial_tu.
\]
Since \(\mathcal E=\mathcal E^f+\mathcal E^g\), the mechanical equation
can then be rewritten as
\[
    \mathcal Eu+\delta\mathcal E^g\partial_tu
    +\alpha\nabla^gp
    =
    \mathcal E^fu+\mathcal E^gw+\alpha\nabla^gp.
\]
Moreover, since \(\eta=\delta\alpha\),
\[
    \alpha\dive^g u+\eta\dive^g\partial_tu
    =
    \alpha\dive^g\bigl(u+\delta\partial_tu\bigr)
    =
    \alpha\dive^g w.
\]
Consequently, the original system is equivalently written as
\begin{subequations}\label{system:uvp}
\begin{align}
    \mathcal E^fu+\mathcal E^gw+\alpha\nabla^gp
    &=
    f
    &&\text{in }H^{-1}(\Omega)^3,
    \label{system:uvp:mechanics}
    \\
    \partial_t\bigl(cp+\alpha\dive^gw\bigr)
    +\mathcal Kp
    &=
    g
    &&\text{in }H^1(\omg)^*,
    \label{system:uvp:pressure}
    \\
    \delta\partial_tu+u
    &=
    w
    &&\text{in }H_0^1(\Omega)^3.
    \label{system:uvp:ode}
\end{align}
\end{subequations}
The system is supplemented with
\[
    u(0)=u_0,
    \qquad
    w(0)=w_0:=u_0+\delta v_0,
    \qquad
    p(0)=p_0.
\]

We again follow a Rothe type approach (using the same notation as in the previous section) and discretize this equivalent system in time. 
To that end, we initialise by
\[
    u^0=u_0,
    \qquad
    w^0=w_0:=u_0+\delta v_0,
    \qquad
    p^0=p_0.
\]
Then, for \(i=0,\ldots,N-1\), we seek
\[
    (u^{i+1},w^{i+1},p^{i+1})
    \in
    H_0^1(\Omega)^3
    \times H_0^1(\Omega)^3
    \times H^1(\omg)
\]
such that
\begin{align*}
    \mathcal E^fu^{i+1}
    +\mathcal E^gw^{i+1}
    +\alpha\nabla^gp^{i+1}
    &=
    f^{i+1},
    \\
    c\derh p^{i+1}
    +\alpha\dive^g\derh w^{i+1}
    +\mathcal Kp^{i+1}
    &=
    g^{i+1},
    \\
    \delta\derh u^{i+1}+u^{i+1}
    &=
    w^{i+1}.
\end{align*}
for given $(u^i,w^i,p^i)$.
The last equation is equivalent to
\[
    u^{i+1}
    =
    a_hw^{i+1}+b_hu^i,
    \qquad
    a_h:=\frac{h}{h+\delta},
    \qquad
    b_h:=\frac{\delta}{h+\delta}.
\]
Consequently, \((w^{i+1},p^{i+1})\) solves
\begin{subequations}\label{eq:rothe_discretization_eta}
\begin{empheq}[left=(P_{h,\eta})\empheqlbrace]{alignat=2}
    a_h\mathcal E^fw^{i+1}
    +\mathcal E^gw^{i+1}
    +\alpha\nabla^gp^{i+1}
    &=
    f^{i+1}-b_h\mathcal E^fu^i,&\quad&\text{in}\ H^{-1}(\Omega)^3
    \label{eq:discrete_vp_mechanics}
    \\
    cp^{i+1}
    +\alpha\dive^gw^{i+1}
    +h\mathcal Kp^{i+1}
    &=
    cp^i+\alpha\dive^gw^i+hg^{i+1}&\quad&\text{in}\ H^{1}(\omg)^*.
    \label{eq:discrete_vp_pressure}
\end{empheq}
\end{subequations}

\begin{lemma}[Well-posedness of the discrete problem]
\label{lem:discrete_wellposedness_eta_alpha}
For every \(h>0\) and every
\[
    (u^i,w^i,p^i)
    \in
    H_0^1(\Omega)^3
    \times H_0^1(\Omega)^3
    \times H^1(\omg),
\]
there exists a unique solution
\[
    (u^{i+1},w^{i+1},p^{i+1})
    \in
    H_0^1(\Omega)^3
    \times H_0^1(\Omega)^3
    \times H^1(\omg)
\]
of the discrete system \((P_{h,\eta})\).
\end{lemma}

\begin{proof}
We consider the product space
\[
    X:=H_0^1(\Omega)^3\times H^1(\omg)
\]
and introduce the bilinear form \(B_h\colon X\times X\to\R\)
\[
\begin{aligned}
B_h((w,p),(\varphi,\psi))
:={}&
a_h\langle\mathcal E^fw,\varphi\rangle
+\langle\mathcal E^gw,\varphi\rangle
+\alpha\langle\nabla^gp,\varphi\rangle
\\
&+
c(p,\psi)_{L^2(\omg)}
+\alpha(\dive^gw,\psi)_{L^2(\omg)}
+h\langle\mathcal Kp,\psi\rangle .
\end{aligned}
\]
The corresponding linear functional \(L_h\colon X\to\R\) is given by
\[
L_h(\varphi,\psi)
:=
\left\langle
    f^{i+1}-b_h\mathcal E^fu^i,
    \varphi
\right\rangle+
\left(
    cp^i+\alpha\dive^gw^i+hg^{i+1},
    \psi
\right)_{L^2(\omg)}.
\]
Both \(B_h\) and \(L_h\) are continuous.
Moreover, by the adjoint relation between \(\nabla^g\) and \(\dive^g\),
the coupling terms cancel, and hence
\[
B_h((w,p),(w,p))
=
a_h\langle\mathcal E^fw,w\rangle
+\langle\mathcal E^gw,w\rangle+
c\|p\|_{L^2(\omg)}^2
+h\langle\mathcal Kp,p\rangle .
\]
Since \(a_h>0\), the coercivity assumptions on the elasticity tensors
and Korn's inequality from \cref{lem:korn} imply
\[
    a_h\langle\mathcal E^fw,w\rangle
    +
    \langle\mathcal E^gw,w\rangle
    \geq
    C_h\|w\|_{H^1(\Omega)}^2
\]
for some \(C_h>0\). Moreover, since \(c>0\),
\[
    c\|p\|_{L^2(\omg)}^2
    +
    h\langle\mathcal Kp,p\rangle
    \geq
    C_h\|p\|_{H^1(\omg)}^2.
\]
Thus, \(B_h\) is coercive on \(X\) for every fixed \(h>0\).
The Lax--Milgram lemma therefore yields a unique pair
\[
    (w^{i+1},p^{i+1})\in X
\]
solving \((P_h)\). Finally, we define
\[
    u^{i+1}:=a_hw^{i+1}+b_hu^i.
\]
Since \(u^i,w^{i+1}\in H_0^1(\Omega)^3\), we have
\(u^{i+1}\in H_0^1(\Omega)^3\). This value is uniquely determined and
satisfies
\[
    \delta\derh u^{i+1}+u^{i+1}=w^{i+1}.
\]
Hence, the discrete solution is unique.
\end{proof}

\begin{lemma}[Preliminary uniform estimates]
\label{lem:preliminary_estimate_eta_alpha}
Let \((u^i,w^i,p^i)\) be the solution of the discrete problem and set
\[
    q^i:=cp^i+\alpha\dive^g w^i.
\]
Then there exists a constant \(C>0\), independent of \(h\), such that
\[
    \max_{0\leq i\leq N}
    \left(
        \|u^i\|_{H^1(\omf)}^2
        +
        \|q^i\|_{H^1(\omg)^*}^2
    \right)
    +h\sum_{i=0}^{N-1}
\left(
    \|w^{i+1}\|_{H^1(\omg)}^2
    +
    \|p^{i+1}\|_{L^2(\omg)}^2
\right)
    \leq C.
\]
\end{lemma}
\begin{proof}
    We start by introducing spaces
\[
    H_\diamond^1(\omg)
    :=
    \left\{
        \psi\in H^1(\omg):
        \int_{\omg}\psi\,\di x=0
    \right\},\quad
    L_\diamond^2(\omg)
    :=
    \left\{
        \psi\in L^2(\omg):
        \int_{\omg}\psi\,\di x=0
    \right\}
\]
and note the decompositions
\[
H^1(\omg)=H_\diamond^1(\omg)\oplus\R,\quad L^2(\omg)=L^2_\diamond(\omg)\oplus\R
\]
where \(\R\) is identified with the space of constant functions.
For \(z\in L^2(\omg)\), we set
\[
    \overline z
    :=
    \fint_{\omg}z\,\di x,
    \qquad
    z^\circ:=z-\overline z.
\]
Thus
\[
    z=z^\circ+\overline z,
    \qquad
    z^\circ\in L_\diamond^2(\omg),
    \qquad
    \overline z\in\R.
\]
We introduce the discrete fluid content variable
\[
q^i=cp^i+\alpha\dive^gw^i
\]
and decompose
\[
    q^i=q^{i,\circ}+\overline q^i,
    \qquad
    p^i=p^{i,\circ}+\overline p^i,
\]
The discrete mass balance
\[
    q^{i+1}-q^i+h\mathcal Kp^{i+1}
    =
    hg^{i+1}
    \quad\text{in }H^1(\omg)^*
\]
is equivalent to the system
\begin{alignat}{2}
    \overline q^{i+1}-\overline q^i
    &=
    h\overline g^{i+1}\quad &&\text{in }\R
    \label{eq:mean_q}
    \\
    q^{i+1,\circ}-q^{i,\circ}
    +
    h\mathcal K p^{i+1,\circ}
    &=
    hg^{i+1,\circ}
    \quad&&\text{in }H_\diamond^1(\omg)^*.
    \label{eq:mean_free_q}
\end{alignat}
Summing \eqref{eq:mean_q} over $i=0,\dots,j$ and applying Cauchy-Schwarz, we find that
\[
    \max_{0\leq i\leq N}|\overline q^i|
    \leq|\overline q^0|+\left(\frac{T}{|\omg|}\right)^{\frac12}\|g\|_{L^2(S\times\omg)}\le C.
\]
We test \eqref{eq:discrete_vp_mechanics} with
\(h w^{i+1}\) and \eqref{eq:mean_free_q}
with \(\mathcal{K}^{-1}q^{i+1,\circ}\) where we note that $\mathcal{K}$ is an isomorphism from $H_\diamond^1(\omg)$ onto its dual (by the positivity of \(K\) and the Poincaré--Wirtinger inequality):
\begin{align}
    h\langle\mathcal E^fu^{i+1},w^{i+1}\rangle
    +
    h\langle\mathcal E^gw^{i+1},w^{i+1}\rangle
    -
    h\alpha
    (p^{i+1},\dive^gw^{i+1})_{L^2(\omg)}
    &=
    h(f^{i+1},w^{i+1})_{L^2(\Omega)},
    \label{eq:mechanical_energy_test}\\
    \left\langle
        q^{i+1,\circ}-q^{i,\circ},
        \mathcal K^{-1}q^{i+1,\circ}
    \right\rangle
    +
    h(p^{i+1,\circ},q^{i+1,\circ})_{L^2(\omg)}
    &=
    h\left\langle
        g^{i+1,\circ},
        \mathcal K^{-1}q^{i+1,\circ}
    \right\rangle.
    \label{eq:q_energy_test}
\end{align}
Here, we have used the reconstruction 
\[
    u^{i+1}=a_hw^{i+1}+b_hu^i.
\]
Since
\[
    q^{i+1,\circ}
    =
    cp^{i+1,\circ}
    +
    \alpha
    \bigl(\dive^gw^{i+1}\bigr)^\circ,
\]
we have
\[
    (p^{i+1,\circ},q^{i+1,\circ})_{L^2(\omg)}
    =
    c\|p^{i+1,\circ}\|_{L^2(\omg)}^2
    +
    \alpha
    (p^{i+1,\circ},\dive^gw^{i+1})_{L^2(\omg)}
\]
where we used that \(p^{i+1,\circ}\) has zero mean in the last term.
Hence, the mean-free part of the coupling term cancels when
\eqref{eq:mechanical_energy_test} and \eqref{eq:q_energy_test}
are added.
For the remaining mean-value contribution, we use
\[
    \overline q^{i+1}
    =
    c\overline p^{i+1}
    +
    \alpha
    \overline{\dive^gw^{i+1}},
\]
which implies
\[
\begin{aligned}
    -\alpha|\omg|\,
    \overline p^{i+1}
    \overline{\dive^gw^{i+1}}
    &=
    c|\omg|\,|\overline p^{i+1}|^2
    -
    |\omg|\,
    \overline p^{i+1}\overline q^{i+1}.
\end{aligned}
\]

We define
\[
    \|v\|_{\mathcal E^f}^2
    :=
    \langle\mathcal E^fv,v\rangle,
    \qquad
    \|v\|_{\mathcal E^g}^2
    :=
    \langle\mathcal E^gv,v\rangle,
\]
and, for \(r\in H_\diamond^1(\omg)^*\),
\[
    \|r\|_{\mathcal K^{-1}}^2
    :=
    \langle r,\mathcal K^{-1}r\rangle.
\]
Using
\[
    w^{i+1}
    =
    u^{i+1}
    +
    \delta\derh u^{i+1},
\]
we obtain
\begin{align}
    h\langle\mathcal E^fu^{i+1},w^{i+1}\rangle
    &=
    h\|u^{i+1}\|_{\mathcal E^f}^2
    +
    \frac{\delta}{2}
    \Bigl(
        \|u^{i+1}\|_{\mathcal E^f}^2
        -
        \|u^i\|_{\mathcal E^f}^2
        +
        \|u^{i+1}-u^i\|_{\mathcal E^f}^2
    \Bigr).
    \label{eq:fibre_energy_identity}
\end{align}
Moreover,
\begin{align}
    \left\langle
        q^{i+1,\circ}-q^{i,\circ},
        \mathcal K^{-1}q^{i+1,\circ}
    \right\rangle
    =
    \frac12
    \Bigl(
        \|q^{i+1,\circ}\|_{\mathcal K^{-1}}^2
        -
        \|q^{i,\circ}\|_{\mathcal K^{-1}}^2
        +
        \|q^{i+1,\circ}-q^{i,\circ}\|_{\mathcal K^{-1}}^2
    \Bigr).
    \label{eq:q_energy_identity}
\end{align}
Combining
\eqref{eq:mechanical_energy_test}--\eqref{eq:q_energy_identity},
we arrive at
\begin{multline}
\frac12
\Bigl(
    \|q^{i+1,\circ}\|_{\mathcal K^{-1}}^2
    -
    \|q^{i,\circ}\|_{\mathcal K^{-1}}^2
    +
    \|q^{i+1,\circ}-q^{i,\circ}\|_{\mathcal K^{-1}}^2
\Bigr)\\
+
\frac{\delta}{2}
\Bigl(
    \|u^{i+1}\|_{\mathcal E^f}^2
    -
    \|u^i\|_{\mathcal E^f}^2
    +
    \|u^{i+1}-u^i\|_{\mathcal E^f}^2
\Bigr)
+
h\|u^{i+1}\|_{\mathcal E^f}^2
+
h\|w^{i+1}\|_{\mathcal E^g}^2
+
hc\|p^{i+1,\circ}\|_{L^2(\omg)}^2
\\
+hc|\omg|\,|\overline p^{i+1}|^2
=
h(f^{i+1},w^{i+1})_{L^2(\Omega)}
+
h\left\langle
    g^{i+1,\circ},
    \mathcal K^{-1}q^{i+1,\circ}
\right\rangle
+
h|\omg|\,
\overline p^{i+1}\overline q^{i+1}.
\label{eq:discrete_q_energy}
\end{multline}
We now estimate the terms on the right-hand side. First,
\[
\begin{aligned}
    h|\omg|
    \left|
        \overline p^{i+1}\overline q^{i+1}
    \right|
    &\leq
    \frac{hc|\omg|}{2}
    |\overline p^{i+1}|^2
    +
    \frac{h|\omg|}{2c}
    |\overline q^{i+1}|^2.
\end{aligned}
\]
The second term is uniformly summable since
\[
    \max_{0\leq i\leq N}|\overline q^i|\leq C.
\]
Furthermore, for every \(\gamma>0\),
\[
    h\left|
        \left\langle
            g^{i+1,\circ},
            \mathcal K^{-1}q^{i+1,\circ}
        \right\rangle
    \right|
    \leq
    \gamma h
    \|q^{i+1,\circ}\|_{\mathcal K^{-1}}^2
    +
    C_\gamma h
    \|g^{i+1}\|_{L^2(\omg)}^2.
\]
To treat the force term, we use
\[
    w^{i+1}
    =
    u^{i+1}
    +
    \delta\derh u^{i+1}
\]
and decompose
\begin{multline}
    h(f^{i+1},w^{i+1})_{L^2(\Omega)}\\
    =
    h(f^{i+1},w^{i+1})_{L^2(\omg)}
    +
    h(f^{i+1},u^{i+1})_{L^2(\omf)}
    +
    \delta h
    (f^{i+1},\derh u^{i+1})_{L^2(\omf)}.
    \label{eq:force_decomposition}
\end{multline}
The first two terms satisfy
\[
\begin{aligned}
    h\left|
        (f^{i+1},w^{i+1})_{L^2(\omg)}
    \right|
    &\leq
    \gamma h
    \|w^{i+1}\|_{\mathcal E^g}^2
    +
    C_\gamma h
    \|f^{i+1}\|_{L^2(\omg)}^2,
    \\
    h\left|
        (f^{i+1},u^{i+1})_{L^2(\omf)}
    \right|
    &\leq
    \gamma h
    \|u^{i+1}\|_{\mathcal E^f}^2
    +
    C_\gamma h
    \|f^{i+1}\|_{L^2(\omf)}^2.
\end{aligned}
\]
The last term in \eqref{eq:force_decomposition} is treated after
summation in time. Discrete summation by parts yields
\begin{multline}
\delta h\sum_{i=0}^j
(f^{i+1},\derh u^{i+1})_{L^2(\omf)}\\
=
\delta
(f^{j+1},u^{j+1})_{L^2(\omf)}
-
\delta
(f^{1},u^{0})_{L^2(\omf)}
-
\delta
\sum_{i=1}^j
(f^{i+1}-f^{i},u^{i})_{L^2(\omf)}.
\label{eq:force_summation_by_parts}
\end{multline}
For the terminal term, Young's inequality gives
\[
\begin{aligned}
    \delta
    \left|
        (f^{j+1},u^{j+1})_{L^2(\omf)}
    \right|
    &\leq
    \frac{\delta}{4}
    \|u^{j+1}\|_{\mathcal E^f}^2
    +
    C\delta
    \|f^{j+1}\|_{L^2(\omf)}^2.
\end{aligned}
\]
Moreover, writing
\[
    f^{i+1}-f^{i}
    =
    h\derh f^{i+1},
\]
we obtain, for every \(\gamma>0\),
\[\delta
\sum_{i=1}^j
\left|
    (f^{i+1}-f^{i},u^{i})_{L^2(\omf)}
\right|
\leq
\gamma h\sum_{i=1}^j
\|u^i\|_{\mathcal E^f}^2
+
C_{\gamma,\delta}
h\sum_{i=1}^j
\|\derh f^{i+1}\|_{L^2(\omf)}^2.
\]
Summing \eqref{eq:discrete_q_energy} over
\(i=0,\ldots,j\), choosing \(\gamma>0\) sufficiently small, and
using the coercivity assumptions and Korn's inequality, we obtain
\begin{multline}
\|u^{j+1}\|_{\mathcal E^f}^2
+
\|q^{j+1,\circ}\|_{\mathcal K^{-1}}^2\\
+
h\sum_{i=0}^j
\left(
    \|u^{i+1}\|_{\mathcal E^f}^2
    +
    \|w^{i+1}\|_{\mathcal E^g}^2
    +
    \|p^{i+1}\|_{L^2(\omg)}^2
\right)
+
\sum_{i=0}^j
\left(
    \|u^{i+1}-u^i\|_{\mathcal E^f}^2
    +
    \|q^{i+1,\circ}-q^{i,\circ}\|_{\mathcal K^{-1}}^2
\right)\\
\leq
C
+
Ch\sum_{i=0}^j
\left(
    \|u^i\|_{\mathcal E^f}^2
    +
    \|q^{i,\circ}\|_{\mathcal K^{-1}}^2
\right),
\label{eq:preliminary_energy_estimate}
\end{multline}
where \(C\) depends on the data.
The discrete Grönwall inequality therefore yields
\begin{multline}
\max_{0\leq i\leq N}
\left(
    \|u^i\|_{\mathcal E^f}^2
    +
    \|q^{i,\circ}\|_{\mathcal K^{-1}}^2
\right)
+
h\sum_{i=0}^{N-1}
\left(
    \|u^{i+1}\|_{\mathcal E^f}^2
    +
    \|w^{i+1}\|_{\mathcal E^g}^2
    +
    \|p^{i+1}\|_{L^2(\omg)}^2
\right)\\
+
\sum_{i=0}^{N-1}
\left(
    \|u^{i+1}-u^i\|_{\mathcal E^f}^2
    +
    \|q^{i+1,\circ}-q^{i,\circ}\|_{\mathcal K^{-1}}^2
\right)
\leq C.
\label{eq:preliminary_uniform_estimate}
\end{multline}
Estimate~\eqref{eq:preliminary_uniform_estimate}, Korn's inequality,
and the uniform bound for the mean value \(\overline q^i\) imply
\[
    \max_{0\leq i\leq N}
    \left(
        \|u^i\|_{H^1(\omf)}^2
        +
        \|q^i\|_{H^1(\omg)^*}^2
    \right)
    \leq C.
\]
Moreover,
\[
    h\sum_{i=0}^{N-1}
    \left(
        \|w^{i+1}\|_{H^1(\omg)}^2
        +
        \|p^{i+1}\|_{L^2(\omg)}^2
    \right)
    \leq C.
\]
\end{proof}

\begin{remark}
    \begin{itemize}
        \item These estimates already imply, up to a subsequence,
        \[
            u_h\xrightharpoonup{*}u
            \quad\text{in }L^\infty\bigl(S;H^1(\omf)^3\bigr),\quad
            w_h\rightharpoonup w
            \quad\text{in }L^2\bigl(S;H^1(\omg)^3\bigr),
        \]
        \[
            p_h\rightharpoonup p
            \quad\text{in }L^2\bigl(S;L^2(\omg)\bigr),
            \qquad
            q_h\xrightharpoonup{*}q
            \quad\text{in }L^\infty\bigl(S;H^1(\omg)^*\bigr)
        \]
        which is enough to identify
        \[
            q=cp+\alpha\dive^g w.
        \]
        At this stage, however, we neither control the pressure gradient nor
        the relevant discrete time derivatives. Thus, we cannot yet pass to
        the term \(\mathcal Kp_h\), identify
        \[
            w=u+\delta\partial_tu,
        \]
        or obtain an \(L^2(S;H^1(\omg)^*)\)-bound for
        \(\partial_tq_h\). These bounds are established in the following
        higher-order estimate.  
        \item Note that these estimates hold without the higher regularity and compatibility conditions on the data.
        These only become necessary for the missing estimates.
    \end{itemize}
\end{remark}

\begin{lemma}[Higher-order uniform estimates]
\label{lem:higher_estimate_eta_alpha}
Assume in addition that
\[
    f^f\in H^2\bigl(S;L^2(\omf)^3\bigr),
    \qquad
    f^g\in H^1\bigl(S;L^2(\omg)^3\bigr),
\]
that
\[
    p_{\e,0}\in H^1(\omg),
\]
and that the initial data satisfy
\[
    w^0=u_{0}+\delta v_{0}
\]
and
\[
    \mathcal E^fu_{0}
    +
    \mathcal E^gw^0
    +
    \alpha\nabla^gp_{0}
    =
    f(0)
    \quad\text{in }H^{-1}(\Omega)^3.
\]
Then there exist \(h_0>0\) and \(C>0\), independent of
\(h\in(0,h_0]\), such that
\[
    \max_{0\leq i\leq N}
    \|p^i\|_{H^1(\omg)}^2
    +
    h\sum_{i=0}^{N-1}\left(\|\derh u^{i+1}\|_{H^1(\Omega)}^2+
    \|\derh p^{i+1}\|_{L^2(\omg)}^2+\|\derh q^{i+1}\|_{H^1(\omg)^*}^2\right)
    \leq C.
\]
\end{lemma}
\begin{proof}
    Set
\[
    v^{i+1}:=\derh u^{i+1},
    \qquad
    v^0:=v_{\e,0}.
\]
Then
\[
    w^{i}=u^{i}+\delta v^{i}
\]
and hence
\[
    \derh w^{i}
    =
    v^{i}+\delta\derh v^{i}.
\]
Due to the initial compatibility condition, the discrete mechanical
equation also holds at \(i=0\). Taking the difference of the mechanical
equations at two consecutive time levels gives
\begin{equation}
    \mathcal E^f v^{i+1}
    +
    \mathcal E^g\derh w^{i+1}
    +
    \alpha\nabla^g\derh p^{i+1}
    =
    \derh f^{i+1}
    \quad\text{in }H^{-1}(\Omega)^3.
    \label{eq:discrete_differentiated_mechanics}
\end{equation}
The discrete mass balance can be written as
\begin{equation}
    c\derh p^{i+1}
    +
    \alpha\dive^g\derh w^{i+1}
    +
    \mathcal Kp^{i+1}
    =
    g^{i+1}
    \quad\text{in }H^1(\omg)^*.
    \label{eq:discrete_mass_second_estimate}
\end{equation}

We test \eqref{eq:discrete_differentiated_mechanics} with
\(h\derh w^{i+1}\) and
\eqref{eq:discrete_mass_second_estimate} with
\(h\derh p^{i+1}\). The coupling terms cancel due to the adjoint
relation between \(\nabla^g\) and \(\dive^g\). We obtain
\begin{multline}
h\langle\mathcal E^fv^{i+1},\derh w^{i+1}\rangle
+
h\|\derh w^{i+1}\|_{\mathcal E^g}^2
+
ch\|\derh p^{i+1}\|_{L^2(\omg)}^2
+
h\langle\mathcal Kp^{i+1},\derh p^{i+1}\rangle\\
=
h(\derh f^{i+1},\derh w^{i+1})_{L^2(\Omega)}
+
h(g^{i+1},\derh p^{i+1})_{L^2(\omg)}.
\label{eq:second_discrete_energy}
\end{multline}
Using
\[
    \derh w^{i+1}
    =
    v^{i+1}+\delta\derh v^{i+1},
\]
we have
\begin{align}
h\langle\mathcal E^fv^{i+1},\derh w^{i+1}\rangle
=
h\|v^{i+1}\|_{\mathcal E^f}^2+
\frac{\delta}{2}
\Bigl(
    \|v^{i+1}\|_{\mathcal E^f}^2
    -
    \|v^i\|_{\mathcal E^f}^2
    +
    \|v^{i+1}-v^i\|_{\mathcal E^f}^2
\Bigr).
\label{eq:velocity_energy_identity}
\end{align}
Moreover,
\begin{align}
h\langle\mathcal Kp^{i+1},\derh p^{i+1}\rangle
=
\frac12
\Bigl(
    \|p^{i+1}\|_{\mathcal K}^2
    -
    \|p^i\|_{\mathcal K}^2
    +
    \|p^{i+1}-p^i\|_{\mathcal K}^2
\Bigr),
\label{eq:pressure_gradient_identity}
\end{align}
where
\[
    \|p\|_{\mathcal K}^2
    :=
    \langle\mathcal Kp,p\rangle
    =
    \int_{\omg}K\nabla p\cdot\nabla p\,\di x.
\]
The term involving \(g^{i+1}\) is estimated exactly as in
\cref{thm:existence_eta0}:
\[
\begin{aligned}
    h\left|
        (g^{i+1},\derh p^{i+1})_{L^2(\omg)}
    \right|
    &\leq
    \frac{ch}{4}
    \|\derh p^{i+1}\|_{L^2(\omg)}^2
    +
    Ch\|g^{i+1}\|_{L^2(\omg)}^2.
\end{aligned}
\]
For the force term, we distinguish between the gel and fibre parts.
In the gel part, Young's inequality and the coercivity of
\(\mathcal E^g\) yield
\[
\begin{aligned}
    h\left|
        (\derh f^{i+1},\derh w^{i+1})_{L^2(\omg)}
    \right|
    &\leq
    \gamma h
    \|\derh w^{i+1}\|_{\mathcal E^g}^2
    +
    C_\gamma h
    \|\derh f^{i+1}\|_{L^2(\omg)}^2.
\end{aligned}
\]
In the fibre part, we use
\[
\begin{aligned}
h(\derh f^{i+1},\derh w^{i+1})_{L^2(\omf)}
={}&
h(\derh f^{i+1},v^{i+1})_{L^2(\omf)}
\\
&+
\delta
(\derh f^{i+1},v^{i+1}-v^{i})_{L^2(\omf)}.
\end{aligned}
\]
The first term satisfies
\[
\begin{aligned}
h\left|
    (\derh f^{i+1},v^{i+1})_{L^2(\omf)}
\right|
\leq
\gamma h\|v^{i+1}\|_{\mathcal E^f}^2
+
C_\gamma h
\|\derh f^{i+1}\|_{L^2(\omf)}^2.
\end{aligned}
\]
The remaining term is treated after summation in time. Discrete
summation by parts gives
\begin{align}
&\delta\sum_{i=0}^j
(\derh f^{i+1},v^{i+1}-v^{i})_{L^2(\omf)}
\notag\\
&\qquad
=
\delta
(\derh f^{j+1},v^{j+1})_{L^2(\omf)}
-
\delta
(\derh f^{1},v^{0})_{L^2(\omf)}
\notag\\
&\qquad\quad
-
\delta\sum_{i=1}^j
(\derh f^{i+1}-\derh f^{i},v^{i})_{L^2(\omf)}.
\label{eq:second_force_summation}
\end{align}
The terminal term is estimated by
\[
\begin{aligned}
\delta
\left|
    (\derh f^{j+1},v^{j+1})_{L^2(\omf)}
\right|
\leq
\frac{\delta}{4}
\|v^{j+1}\|_{\mathcal E^f}^2
+
C\delta
\|\derh f^{j+1}\|_{L^2(\omf)}^2.
\end{aligned}
\]
The initial term is controlled by the prescribed initial velocity
\(v_{\e,0}\) and the trace of \(\partial_tf^f\) at \(t=0\).
Since
\[
    \derh f^{i+1}-\derh f^{i}
    =
    h(\derh)^2f^{i+1},
\]
we obtain
\begin{align}
&\delta\sum_{i=1}^j
\left|
    (\derh f^{i+1}-\derh f^{i},v^{i})_{L^2(\omf)}
\right|
\notag\\
&\qquad
\leq
\gamma h\sum_{i=1}^j
\|v^i\|_{\mathcal E^f}^2
+
C_{\gamma,\delta}
h\sum_{i=1}^j
\|(\derh)^2f^{i+1}\|_{L^2(\omf)}^2.
\label{eq:second_force_difference_estimate}
\end{align}
The discrete first and second derivatives of \(f\) are uniformly
controlled by
\[
    f^f\in H^2\bigl(S;L^2(\omf)^3\bigr),
    \qquad
    f^g\in H^1\bigl(S;L^2(\omg)^3\bigr).
\]
Summing \eqref{eq:second_discrete_energy} over
\(i=0,\ldots,j\), using
\eqref{eq:velocity_energy_identity}--%
\eqref{eq:second_force_difference_estimate}, and choosing
\(\gamma>0\) sufficiently small, we obtain
\begin{align}
&\|v^{j+1}\|_{\mathcal E^f}^2
+
\|p^{j+1}\|_{\mathcal K}^2
\notag\\
&\quad
+
h\sum_{i=0}^j
\left(
    \|v^{i+1}\|_{\mathcal E^f}^2
    +
    \|\derh w^{i+1}\|_{\mathcal E^g}^2
    +
    \|\derh p^{i+1}\|_{L^2(\omg)}^2
\right)
\notag\\
&\quad
+
\sum_{i=0}^j
\left(
    \|v^{i+1}-v^i\|_{\mathcal E^f}^2
    +
    \|p^{i+1}-p^i\|_{\mathcal K}^2
\right)
\notag\\
&\leq
C
+
Ch\sum_{i=0}^j
\|v^i\|_{\mathcal E^f}^2,
\label{eq:second_pre_gronwall_estimate}
\end{align}
where \(C\) is independent of \(h\). The discrete Grönwall inequality
therefore gives
\begin{align}
&\max_{0\leq i\leq N}
\left(
    \|v^i\|_{\mathcal E^f}^2
    +
    \|p^i\|_{\mathcal K}^2
\right)
\notag\\
&\quad
+
h\sum_{i=0}^{N-1}
\left(
    \|v^{i+1}\|_{\mathcal E^f}^2
    +
    \|\derh w^{i+1}\|_{\mathcal E^g}^2
    +
    \|\derh p^{i+1}\|_{L^2(\omg)}^2
\right)
\notag\\
&\quad
+
\sum_{i=0}^{N-1}
\left(
    \|v^{i+1}-v^i\|_{\mathcal E^f}^2
    +
    \|p^{i+1}-p^i\|_{\mathcal K}^2
\right)
\leq C.
\label{eq:second_uniform_estimate}
\end{align}
By Korn's inequality, \eqref{eq:second_uniform_estimate} implies
\[
    h\sum_{i=0}^{N-1}
    \|v^{i+1}\|_{H^1(\omf)}^2
    \leq C
\]
and
\[
    h\sum_{i=0}^{N-1}
    \|\derh w^{i+1}\|_{H^1(\omg)}^2
    \leq C.
\]
Together with the preliminary estimate
\[
    h\sum_{i=0}^{N-1}
    \|w^{i+1}\|_{H^1(\omg)}^2
    \leq C,
\]
this yields a uniform \(H^1\)-in-time estimate for the piecewise affine
interpolant of \(w^g\). In particular,
\[
    \max_{0\leq i\leq N}
    \|w^{i}\|_{H^1(\omg)}^2
    \leq C.
\]
The discrete relaxation equation gives
\[
    v^{i}
    =
    \frac1\delta\bigl(w^{i}-u^{i}\bigr).
\]
Since
\[
    u^{i+1}=a_hw^{i+1}+b_hu^i,
\]
the stability of the implicit Euler scheme, together with the preceding
bound for \(w^{i}\), yields
\[
    \max_{0\leq i\leq N}
    \|u^{i}\|_{H^1(\omg)}^2
    \leq C.
\]
Consequently,
\[
    \max_{0\leq i\leq N}
    \|v^{i}\|_{H^1(\omg)}^2
    \leq C.
\]
Combining the fibre and gel estimates, we conclude that
\begin{equation}
    h\sum_{i=0}^{N-1}
    \|\derh u^{i+1}\|_{H^1(\Omega)}^2
    \leq C.
    \label{eq:global_discrete_velocity_estimate}
\end{equation}
Moreover, \eqref{eq:second_uniform_estimate}, the bound for the mean
value of \(p^i\) obtained from the preliminary estimate, and the
Poincaré--Wirtinger inequality give
\[
    \max_{0\leq i\leq N}
    \|p^i\|_{H^1(\omg)}^2
    +
    h\sum_{i=0}^{N-1}
    \|\derh p^{i+1}\|_{L^2(\omg)}^2
    \leq C.
\]
Finally, the discrete mass balance implies
\[
    \derh q^{i+1}
    =
    g^{i+1}-\mathcal Kp^{i+1}.
\]
Hence,
\[
    h\sum_{i=0}^{N-1}
    \|\derh q^{i+1}\|_{H^1(\omg)^*}^2
    \leq C.
    \label{eq:discrete_fluid_content_derivative}
\]
\end{proof}

\begin{theorem}[Well-posedness for \(\eta=\alpha\delta\)]
\label{lemma:existence_eta_alpha}
There exists a unique weak solution \((u,p)\) in the sense of \cref{def:weak_solution} for $\eta=\alpha\delta$satisfying
\[
    u\in H^1\bigl(S;H_0^1(\Omega)^3\bigr),
    \qquad
    \partial_tu\in L^\infty\bigl(S;H^1(\omg)^3\bigr),
\]
and
\[
    p\in
    L^\infty\bigl(S;H^1(\omg)\bigr)
    \cap
    H^1\bigl(S;L^2(\omg)\bigr).
\]
Moreover,
\[
    q:=cp+\alpha\dive^g\bigl(u+\delta\partial_tu\bigr)
    \in H^1\bigl(S;H^1(\omg)^*\bigr).
\]
\end{theorem}
\begin{proof}
The estimates of \cref{lem:preliminary_estimate_eta_alpha,lem:higher_estimate_eta_alpha} provide the uniform bounds required for the construction of the piecewise constant and affine time interpolants associated with
\[
    (u^i,w^i,p^i,q^i).
\]
The compactness argument and the passage to the limit in the discrete
mechanical and mass-balance equations are analogous to the proof of
\cref{thm:existence_eta0}.
We therefore only indicate the additional identifications specific to the present case.

Let \(\widehat u_h\) denote the piecewise affine interpolant of
\((u^i)\). Since
\[
    \delta\partial_t\widehat u_h+u_h=w_h,
\]
and the three terms are uniformly bounded in the corresponding spaces,
we may pass to the limit and obtain
\[
    w=u+\delta\partial_tu.
\]
Moreover, the weak convergences of \(p_h\) and \(w_h\), together with the continuity of \(\dive^g\), imply
\[
    q=cp+\alpha\dive^g w.
\]
Passing to the limit in the discrete equations yields
\[
    E^fu+E^gw+\alpha\nabla^gp=f
    \quad\text{in }L^2(S;H^{-1}(\Omega)^3),
\]
and
\[
    \partial_tq+Kp=g
    \quad\text{in }L^2(S;H^1(\omg)^*).
\]
The time-continuity of the corresponding affine interpolants, together
with Lemma~\ref{lem:higher_estimate_eta_alpha} and the initialization of
the discrete scheme, gives
\[
    u(0)=u_0,\qquad
    w(0)=w_0,\qquad
    q(0)=\zeta_0.
\]
Since
\[
    q=cp+\alpha\dive^g w,
    \qquad
    \zeta_0=cp_0+\alpha\dive^g w_0,
\]
we conclude that
\[
    p(0)=p_0
    \quad\text{in }L^2(\omg).
\]
Finally, since $w=u+\delta\partial_tu$,
\[
    cp+\alpha\dive^g(u+\delta\partial_tu)
    =
    q
    \in H^1\bigl(S;H^1(\omg)^*\bigr),
\]
and $(u,p)$ is a weak solution.

Now, let
\[
    (u_1,p_1)
    \qquad\text{and}\qquad
    (u_2,p_2)
\]
be two weak solutions corresponding to the same data, and set
\[
    u:=u_1-u_2,
    \qquad
    p:=p_1-p_2,\qquad 
    w:=u+\delta\partial_tu,
    \qquad
    q:=cp+\alpha\dive^gw.
\]
Then
\begin{subequations}
\label{eq:uniqueness_eta_alpha}
\begin{align}
    \mathcal E^fu+\mathcal E^gw+\alpha\nabla^gp
    &=0
    &&\text{in }H^{-1}(\Omega)^3,
    \label{eq:uniqueness_mechanics}
    \\
    \partial_tq+\mathcal Kp
    &=0
    &&\text{in }H^1(\omg)^*,
    \label{eq:uniqueness_mass}
    \\
    \delta\partial_tu+u
    &=w
    &&\text{in }H_0^1(\Omega)^3.
    \label{eq:uniqueness_relaxation}
\end{align}
\end{subequations}
Moreover,
\[
    u(0)=0,
    \qquad
    q(0)=0.
\]
Testing \eqref{eq:uniqueness_mass} with the constant function \(1\)
shows that
\[
    \frac{\di}{\di t}\overline q=0.
\]
Since \(\overline q(0)=0\), it follows that
\[
    \overline q(t)=0
    \qquad\text{for a.e. }t\in S.
\]
Thus, \(q(t)\in H_\diamond^1(\omg)^*\), and we may test
\eqref{eq:uniqueness_mass} with \(\mathcal K^{-1}q\). This gives
\[
    \frac12\frac{\di}{\di t}
    \|q\|_{\mathcal K^{-1}}^2
    +(p,q)_{L^2(\omg)}
    =0.
\]
Since
\[
    q=cp+\alpha\dive^gw,
\]
we have
\[
    (p,q)_{L^2(\omg)}
    =
    c\|p\|_{L^2(\omg)}^2
    +
    \alpha(p,\dive^gw)_{L^2(\omg)}.
\]
On the other hand, testing
\eqref{eq:uniqueness_mechanics} with \(w\) yields
\[
    \langle\mathcal E^fu,w\rangle
    +
    \|w\|_{\mathcal E^g}^2
    -
    \alpha(p,\dive^gw)_{L^2(\omg)}
    =0.
\]
Adding the preceding two identities, the coupling terms cancel, and
we obtain
\[
    \frac12\frac{\di}{\di t}
    \|q\|_{\mathcal K^{-1}}^2
    +
    \langle\mathcal E^fu,w\rangle
    +
    \|w\|_{\mathcal E^g}^2
    +
    c\|p\|_{L^2(\omg)}^2
    =0.
\]
By \eqref{eq:uniqueness_relaxation},
\[
\begin{aligned}
    \langle\mathcal E^fu,w\rangle
    &=
    \langle\mathcal E^fu,u+\delta\partial_tu\rangle
    \\
    &=
    \|u\|_{\mathcal E^f}^2
    +
    \frac{\delta}{2}
    \frac{\di}{\di t}\|u\|_{\mathcal E^f}^2.
\end{aligned}
\]
Consequently,
\[
\begin{aligned}
    \frac12\frac{\di}{\di t}
    \left(
        \delta\|u\|_{\mathcal E^f}^2
        +
        \|q\|_{\mathcal K^{-1}}^2
    \right)
    +
    \|u\|_{\mathcal E^f}^2
    +
    \|w\|_{\mathcal E^g}^2
    +
    c\|p\|_{L^2(\omg)}^2
    =0.
\end{aligned}
\]
Integration over \((0,t)\), together with \(u(0)=0\) and \(q(0)=0\),
gives
\[
\delta\|u(t)\|_{\mathcal E^f}^2
+\|q(t)\|_{\mathcal K^{-1}}^2
+
2\int_0^t
\left(
    \|u\|_{\mathcal E^f}^2
    +
    \|w\|_{\mathcal E^g}^2
    +
    c\|p\|_{L^2(\omg)}^2
\right)\di\tau
=0.
\]
Hence,
\[
    u=0\quad\text{in }S\times\omf,
    \qquad
    w=0\quad\text{in }S\times\omg,
    \qquad
    p=0\quad\text{in }S\times\omg.
\]
Finally, in the gel phase,
\[
    \delta\partial_tu+u=w=0,
    \qquad
    u(0)=0,
\]
and therefore \(u=0\) also in \(S\times\omg\)
\end{proof}

\begin{remark}
     This approach does not work for general $\eta\ne\alpha\delta$:
     the fluid content instead contains the combination
        \[
            u+\frac{\eta}{\alpha}\partial_tu,
        \]
        which coincides with \(w=u+\delta\partial_tu\) only if
        \(\eta=\alpha\delta\).
        Introducing
        \[
            z:=u+\frac{\eta}{\alpha}\partial_tu
        \]
        would therefore lead to the two distinct auxiliary variables \(w\) and \(z\).

    From a modelling perspective, the cases $\eta=0$ and $\eta=\alpha\delta$ are also the most natural ones in the present setting: the former corresponds to the standard Biot fluid content, while in the latter the viscous deformation entering the stress also enters the fluid content with the same relaxation time scale.
\end{remark}

\subsection{A-priori estimates}\label{sec:apriori}
With \cref{lemma:existence_eta_alpha,thm:existence_eta0}, we have unique weak solutions satisfying
\[
\begin{aligned}
    &u_\varepsilon
    \in
    L^\infty(S;H_0^1(\Omega)^3)
    \cap
    H^1(S;H_0^1(\Omega)^3),\\
    &p_\varepsilon
    \in
    L^\infty(S;L^2(\Omega_g^\varepsilon))
    \cap
    L^2(S;H^1(\Omega_g^\varepsilon))
    \cap
    H^1(S;H^1(\Omega_g^\varepsilon)^*)
\end{aligned}
\]
if \(\eta=0\). If \(\eta=\alpha\delta\), the solution additionally
satisfies
\[
    \partial_tu_\varepsilon
    \in L^\infty(S;H^1(\Omega_g^\varepsilon)^3),
    \qquad
    p_\varepsilon
    \in
    L^\infty(S;H^1(\Omega_g^\varepsilon))
    \cap
    H^1(S;L^2(\Omega_g^\varepsilon)).
\]

\begin{lemma}[A-priori estimates]\label{lem:apriori}
Let \((u_\e,p_\e)\) be the solution of
System~\eqref{system:abstract_operator_problem}.
Then the following
estimates hold uniformly with respect to \(\e\).

\begin{itemize}
    \item If $\eta=0$, we have
    \begin{equation}\label{eq:apriori:main}
  \|u_\e\|^2_{L^\infty(S;H^1(\Omega))}
  + \|\partial_t u_\e\|^2_{L^2(S;H^1(\Omega))}
  + \|p_\e\|^2_{L^\infty(S;\,L^2(\omg))}
  + \|\nabla p_\e\|^2_{L^2(S\times\omg)}
  \lesssim 
    1.
\end{equation}
\item If $\eta=\delta\alpha$, we have
\begin{multline}
\label{eq:apriori_eta_alpha_delta}
    \|u_\e\|_{L^\infty(S;H^1(\Omega))}^2
    +
    \|\partial_tu_\e\|_{L^2(S;H^1(\Omega))}^2
    +
    \|\partial_tu_\e\|_{L^\infty(S;H^1(\omg))}^2
    +\|\partial_tw_\e\|_{L^2(S;H^1(\omg))}^2
    \\
    +
    \|\nabla p_\e\|_{L^\infty(S;L^2(\omg))}^2
    +
    \|\partial_tp_\e\|_{L^2(S;L^2(\omg))}^2
    +
    \|\partial_tq_\e\|_
        {L^2(S;H^1(\omg)^*)}^2
    \lesssim 1.
\end{multline}
\end{itemize}
\end{lemma}
\begin{proof}
    \emph{(i) Let \(\eta=0\).}
    The estimates obtained in the proof of
    \cref{thm:existence_eta0} yield
    \begin{multline*}
        \|u_\e\|_{L^\infty(S;H^1(\Omega)^3)}^2
        +
        \|\partial_tu_\e\|_{L^2(S;H^1(\omg)^3)}^2
        +
        \|p_\e\|_{L^\infty(S;L^2(\omg))}^2
        +
        \|\nabla p_\e\|_{L^2(S\times\omg)}^2
        \\
        \leq
        C\Bigl(
            \|g_\e\|_{L^2(S\times\omg)}^2
            +
            \|f_\e\|_{H^1(S;L^2(\omf)^3)}^2
            +
            \|f_\e\|_{L^2(S\times\omg)^3}^2
            +
            \|p_{\e,0}\|_{L^2(\omg)}^2
            +
            \|u_{\e,0}\|_{H^1(\Omega)^3}^2
        \Bigr).
    \end{multline*}
    The constant \(C>0\) is independent of \(\e\), since the extension, Korn, and Poincaré estimates used in the derivation are uniform in \(\e\). The right-hand side is uniformly bounded by Assumption~(A3).
    To extend the estimate for $\partial_tu_\e$ from the gel phase to the fibre phase, let
    \[
        \widetilde v_\e
        :=
        E_\e^g\bigl(\partial_tu_\e|_{\Omega_\e^g}\bigr).
    \]
    Since the traces of $\partial_tu_\e|_{\Omega_\e^f}$ and
    $\partial_tu_\e|_{\Omega_\e^g}$ agree on $\Gamma_\e$, the difference
    \[
        z_\e
        :=
        \partial_tu_\e|_{\Omega_\e^f}
        -
        \widetilde v_\e|_{\Omega_\e^f}
    \]
    has homogeneous trace on $\Gamma_\e$. Differentiating the quasistatic
    fibre equation in time, testing with $z_\e$, and using the uniform Korn
    inequality together with the uniform extension estimate, we also obtain
    \[
        \|\partial_tu_\e\|_{H^1(\Omega_\e^f)}
        \lesssim
        \|\partial_tf_\e^f\|_{L^2(\Omega_\e^f)}
        +
        \|\partial_tu_\e\|_{H^1(\Omega_\e^g)}.
    \]

    \emph{(ii) Let \(\eta=\alpha\delta\).}
    The preliminary and higher-order estimates established in
    \cref{lem:preliminary_estimate_eta_alpha,lem:higher_estimate_eta_alpha}
    imply
    \begin{multline*}
        \|u_\e\|_{L^\infty(S;H^1(\Omega)^3)}^2
        +
        \|\partial_tu_\e\|_{L^2(S;H^1(\Omega)^3)}^2
        +
        \|\partial_tu_\e\|_{L^\infty(S;H^1(\omg)^3)}^2
        +\|\partial_tw_\e\|_{L^2(S;H^1(\omg))}^2
        \\
        +
        \|p_\e\|_{L^\infty(S;H^1(\omg))}^2
        +
        \|\partial_tp_\e\|_{L^2(S;L^2(\omg))}^2
        +
        \|\partial_tq_\e\|_{L^2(S;H^1(\omg)^*)}^2
        \leq C,
    \end{multline*}
    where \(C>0\) is independent of \(\e\). Again, this follows from the
    \(\e\)-uniform geometric inequalities and the uniform bounds on the
    data and initial values in Assumption~(A3).
\end{proof}

\begin{remark}
Please note that the estimates above are uniform with respect to the microscopic parameter $\e$, but no uniformity with respect to the relaxation parameter $\delta>0$ is established.
Formally, the limit $\delta\to0$ connects the present visco-poroelastic model to the purely poroelastic setting considered in \cite{EM22}.
This is also heuristically consistent with our homogenized equations, where the additional viscous and memory contributions are expected to vanish as $\delta\to0$.
A rigorous justification of this singular limit would, therefore, require estimates that remain uniform as $\delta\to0$ (or some appropriate rescaling of the quantities that degenerate in this limit). 
Note that this is a separate singular-limit problem and it is left for future work.
\end{remark}

\section{Homogenization}
\label{sec:homog}

In this section, we rigorously pass to the limit $\e \to 0$ in
System~\eqref{eq:weak_form} and derive the effective (homogenized) equations. 
We use the method of \emph{periodic unfolding} \cite{CDG08,CDG02}, where, for a function $\varphi \in L^2(\Omega)$, the \emph{unfolding operator}
$\mathcal{T}_\e : L^2(\Omega) \to L^2(\Omega \times Y)$ is defined by
\[
  \mathcal{T}_\e(\varphi)(x, y)
  := \varphi\!\left(\e\!\left\lfloor\frac{x}{\e}\right\rfloor + \e y\right),
  \quad (x,y) \in \Omega \times Y.
\]
Here, $\lfloor \cdot \rfloor$ denotes the component-wise integer part.
We recall the key properties:
\begin{enumerate}[$(i)$]
  \item $\|\mathcal{T}_\e(\varphi)\|_{L^2(\Omega\times Y)} = \|\varphi\|_{L^2(\Omega)} + \mathcal{O}(\e)$.
  \item If $\varphi_\e \rightharpoonup \varphi$ in $L^2(\Omega)$, then
        $\mathcal{T}_\e(\varphi_\e) \rightharpoonup \varphi$ in
        $L^2(\Omega \times Y)$.
  \item If $\varphi_\e \in H^1(\Omega)$ with $\|\varphi_\e\|_{H^1(\Omega)} \leq C$,
        then there exist $\varphi \in H^1(\Omega)$ and $\varphi_1 \in L^2(\Omega; H^1_\#(Y))$
        such that, up to a subsequence,
        \[
          \mathcal{T}_\e(\varphi_\e) \rightharpoonup \varphi,
          \quad
          \mathcal{T}_\e(\nabla\varphi_\e)
          \rightharpoonup \nabla_x\varphi + \nabla_y\varphi_1
          \quad \text{in } L^2(\Omega \times Y).
        \]
\end{enumerate}
Here, $H^1_\#(Y)$ denotes the space of $Y$-periodic $H^1$ functions.
Similarly, denoting the trivial extension with zero to $\Omega$ for functions $\phi$ defined on $\Omega_\e^r$ ($r=f,g$) via $\hat\phi$, we also have 
\begin{enumerate}[$(i)$]
  \item $\|\mathcal{T}_\e(\hat\varphi)\|_{L^2(\Omega\times Y)} = \|\varphi\|_{L^2(\Omega_\e^r)} + \mathcal{O}(\e)$.
  \item If $\hat\varphi_\e \rightharpoonup \varphi$ in $L^2(\Omega)$, then
        $\mathcal{T}_\e(\hat\varphi_\e) \rightharpoonup \chi^r\varphi$ in
        $L^2(\Omega \times Y)$.
  \item If $\varphi_\e \in H^1(\Omega_\e^r)$ with $\|\varphi_\e\|_{H^1(\Omega_\e^r)} \leq C$, then there exist $\varphi \in H^1(\Omega)$ and $\varphi_1 \in L^2(\Omega; H^1_\#(Y))$
        such that, up to a subsequence,
        \[
          \mathcal{T}_\e(\hat\varphi_\e) \rightharpoonup \chi^r\varphi,
          \quad
          \mathcal{T}_\e(\nabla\varphi_\e)
          \rightharpoonup \chi^r(\nabla_x\varphi + \nabla_y\varphi_1)
          \quad \text{in } L^2(\Omega \times Y).
        \]
\end{enumerate}
Here, $\chi^r$ denotes the indicator function of the set $Y^r$.

The main result we are proving in this section is the convergence towards the following limit problem:
%

\begin{theorem}[Homogenised system]\label{thm:homog}
As $\e\to0$, the solutions $(u_\e,p_\e)$ in the sense of \cref{def:weak_solution} converge, in the sense of
Proposition~\ref{prop:twoscale} and up to a subsequence, to a solution $(u,p)$ of the following homogenised system:
Set
\begin{align*}
    \Sigma^{\rm h}(u,p)
    &:=
    \mathbb C^{\rm h}e(u)
    +\delta\mathbb D^{\rm h}e(\partial_tu)
    -A^{\rm h}p
    +\delta\partial_t\mathbb M^{\rm h}*e(\partial_tu)
    +\delta\partial_tR^{\rm h}*\partial_tp,\\
    F^{\rm h}(u,p)
    &:=
    c^{\rm h}p+A^{\rm h}:e(u)
    +\delta\bigl(
        -\partial_tR^{\rm h}*e(\partial_tu)
        +H^{\rm h}*\partial_tp
    \bigr),\\
    G^{\rm h}(u,p)
    &:=
    F^{\rm h}(u,p)-c|Y^g|p.
\end{align*}
Then
\begin{subequations}\label{eq:homog}
\begin{align}
    -\dive\Sigma^{\rm h}(u,p)
    &=f
    &&\text{in }S\times\Omega,\\
    \partial_t Q_\eta^{\rm h}(u,p)
    -\dive(K^{\rm h}\nabla p)
    &=|Y^g|g
    &&\text{in }S\times\Omega,\\
    u&=0
    &&\text{on }S\times\partial\Omega,\\
    -K^{\rm h}\nabla p\cdot n&=0
    &&\text{on }S\times\partial\Omega,
\end{align}
\end{subequations}
where
\[
    Q_\eta^{\rm h}(u,p)
    :=
    \begin{cases}
        F^{\rm h}(u,p),
        & \eta=0,\\[2mm]
        F^{\rm h}(u,p)+\delta\partial_tG^{\rm h}(u,p),
        & \eta=\alpha\delta.
    \end{cases}
\]
The solution satisfies
\[
    u\in H^1\bigl(S;H_0^1(\Omega)^3\bigr),
    \qquad
    p\in L^2\bigl(S;H^1(\Omega)\bigr),\qquad
    Q_\eta^{\rm h}(u,p)
    \in H^1\bigl(S;H^1(\Omega)^*\bigr),
\]
together with
\[
    u(0)=u_0,
    \qquad
    Q_\eta^{\rm h}(u,p)(0)=\zeta_0.
\]
The effective coefficients are given by \eqref{eq:Ceff}--\eqref{eq:Keff} and \eqref{eq:Qhom-def}--\eqref{eq:Ahom-def}, and the memory kernels by \eqref{eq:Mhom}--\eqref{eq:Hhom}.
\end{theorem}

\subsection{Limit functions, cell problems, and effective tensors}
From the uniform a-priori estimates of Section~\ref{sec:analysis}, we extract the
following convergences as $\e \to 0$, up to subsequences.

\begin{lemma}[Two-scale limits]
\label{prop:twoscale}
There exist limit functions
\[
    u\in H^1\bigl(S;H_0^1(\Omega)^3\bigr),
    \qquad
    u_1\in
    H^1\bigl(
        S;L^2(\Omega;H_\#^1(Y)^3)
    \bigr),
\]
and
\[
    p\in L^2(S;H^1(\Omega)),
    \qquad
    p_1\in
    L^2\bigl(
        S\times\Omega;
        H_\#^1(Y^g)
    \bigr),
\]
such that, up to a subsequence,
\begin{align*}
    u_\e
    &\xrightharpoonup{*}
    u
    &&\text{in }
    L^\infty(S;H_0^1(\Omega)^3),
    \\
    \mathcal T_\e(\nabla u_\e)
    &\rightharpoonup
    \nabla_xu+\nabla_yu_1
    &&\text{in }
    L^2(S\times\Omega\times Y)^{3\times3},
    \\
    \mathcal T_\e(\nabla\partial_tu_\e)
    &\rightharpoonup
    \nabla_x\partial_tu+\nabla_y\partial_tu_1
    &&\text{in }
    L^2(S\times\Omega\times Y)^{3\times3},
    \\
    \widehat p_\e
    &\rightharpoonup
    |Y^g|p
    &&\text{in }
    L^2(S\times\Omega),
    \\
    \mathcal T_\e(\widehat{\nabla p_\e})
    &\rightharpoonup
    \chi^g
    \bigl(
        \nabla_xp+\nabla_yp_1
    \bigr)
    &&\text{in }
    L^2(S\times\Omega\times Y)^3.
\end{align*}
Moreover, Assumption~\textbf{(A4)} identifies the initial trace of the displacement corrector:
\[
    u_1(0)=u_{1,0}^{\rm eq}.
\]

If $\eta=\alpha\delta$, set
\[
    w_\e:=u_\e+\delta\partial_tu_\e,
    \qquad
    w:=u+\delta\partial_tu.
\]
Then there exists
\[
    w_1^g\in
    H^1\bigl(
        S;L^2(\Omega;H_\#^1(Y^g)^3)
    \bigr)
\]
such that
\[
    \mathcal T_\e
    \bigl(
        \widehat{\nabla w_\e|_{\Omega_\e^g}}
    \bigr)
    \rightharpoonup
    \chi^g
    \bigl(
        \nabla_xw+\nabla_yw_1^g
    \bigr)
\]
in $L^2(S\times\Omega\times Y)^{3\times3}$.
Moreover,
\[
    w_1^g
    =
    \bigl(u_1+\delta\partial_tu_1\bigr)\big|_{Y^g},
    \qquad
    w_1^g(0)=w_{1,0}|_{Y^g}.
\]
In this case, we also have
\[
    \partial_tp\in L^2(S\times\Omega),
\]
and
\begin{align*}
    \widehat{\partial_tp_\e}
    &\rightharpoonup
    |Y^g|\partial_tp
    &&\text{in }L^2(S\times\Omega),
    \\
    \mathcal T_\e
    \bigl(
        \widehat{\partial_tp_\e}
    \bigr)
    &\rightharpoonup
    \chi^g\partial_tp
    &&\text{in }L^2(S\times\Omega\times Y).
\end{align*}
\end{lemma}

\begin{proof}
The estimate in Lemma~\ref{lem:apriori} implies that $(u_\e)$ is
uniformly bounded in
\[
    H^1\bigl(S;H_0^1(\Omega)^3\bigr).
\]
Hence, up to a subsequence,
\[
    u_\e\rightharpoonup u
    \quad\text{in }H^1\bigl(S;H_0^1(\Omega)^3\bigr),
\]
and, in particular,
\[
    u_\e \xrightharpoonup{*} u
    \quad\text{in }L^\infty\bigl(S;H_0^1(\Omega)^3\bigr).
\]
Moreover,
\[
    \mathcal T_\e(\nabla u_\e)
    \rightharpoonup
    \nabla_xu+\nabla_yu_1
    \quad\text{in }L^2(S\times\Omega\times Y)^{3\times3}.
\]
Moreover, \cref{lem:apriori} shows that $\partial_tu_\e$ is bounded in $L^2\bigl(S;H^1(\Omega)^3\bigr)$.
Consequently,
\(
    \mathcal T_\e(\nabla u_\e)
\)
is bounded in
\[
    H^1\bigl(
        S;L^2(\Omega\times Y)^{3\times3}
    \bigr).
\]
Consequently,
\[
    \mathcal T_\e(\nabla u_\e)
    \rightharpoonup
    \nabla_xu+\nabla_yu_1
\]
weakly in this space, and therefore
\[
    \mathcal T_\e(\nabla\partial_tu_\e)
    \rightharpoonup
    \nabla_x\partial_tu+\nabla_y\partial_tu_1
\]
in \(L^2(S\times\Omega\times Y)^{3\times3}\). In particular,
\[
    u_1\in
    H^1\bigl(
        S;L^2(\Omega;H_\#^1(Y)^3)
    \bigr).
\]
By continuity of the trace operator at \(t=0\),
\[
    \mathcal T_\e(\nabla u_{\e,0})
    \rightharpoonup
    \nabla_xu(0)+\nabla_yu_1(0).
\]
On the other hand, Assumption~{\rm (A4)} gives
\[
    \mathcal T_\e(\nabla u_{\e,0})
    \to
    \nabla_xu_0+\nabla_yu_{1,0}^{\rm eq}.
\]
Since \(u(0)=u_0\), we obtain
\[
    \nabla_yu_1(0)
    =
    \nabla_yu_{1,0}^{\rm eq}.
\]
Choosing the microscopic correctors with zero mean therefore yields
\[
    u_1(0)=u_{1,0}^{\rm eq}.
\]

For the pressure, let
\[
    \widetilde p_\e:=E_\e^g p_\e
\]
be the uniform extension to \(\Omega\). By
\cref{lem:apriori,lem:extension}, the sequence
\((\widetilde p_\e)\) is bounded in
\(L^2(S;H^1(\Omega))\). Therefore, after passing to a further
subsequence,
\[
    \widetilde p_\e\rightharpoonup p
    \quad\text{in }L^2(S;H^1(\Omega)).
\]
The corresponding unfolding compactness result on the periodic gel
phase then gives
\[
    \widehat p_\e\rightharpoonup |Y^g|p
    \quad\text{in }L^2(S\times\Omega)
\]
and
\[
    \mathcal T_\e(\widehat{\nabla p_\e})
    \rightharpoonup
    \chi^g(\nabla_xp+\nabla_yp_1)
    \quad\text{in }L^2(S\times\Omega\times Y)^3.
\]

If $\eta=\alpha\delta$, the additional estimates in \cref{lem:apriori} also imply
\[
    \partial_tp\in L^2(S\times\Omega)
\]
which yields the additional convergences by the same extension and unfolding arguments.

Finally, let
\[
    w_\e:=u_\e+\delta\partial_tu_\e.
\]
The corresponding estimate for
\(w_\e|_{\Omega_\e^g}\) gives, up to a subsequence, a corrector
\[
    w_1^g\in
    H^1\bigl(
        S;L^2(\Omega;H_\#^1(Y^g)^3)
    \bigr)
\]
such that
\[
    \mathcal T_\e
    \bigl(
        \widehat{\nabla w_\e|_{\Omega_\e^g}}
    \bigr)
    \rightharpoonup
    \chi^g
    \bigl(
        \nabla_xw+\nabla_yw_1^g
    \bigr),
    \qquad
    w=u+\delta\partial_tu.
\]
By comparison with the convergences for \(u_\e\) and
\(\partial_tu_\e\),
\[
    w_1^g
    =
    \bigl(u_1+\delta\partial_tu_1\bigr)|_{Y^g}.
\]
The same trace argument as above, together with Assumption~{\rm (A4)},
gives
\[
    w_1^g(0)=w_{1,0}|_{Y^g}.
\]
\end{proof}

As the next step, we introduce some auxiliary problems (so-called cell problems) and the effective tensors calculated via their solutions.
To that end, we set $E_{kl}:= \frac{1}{2}(e_k\otimes e_l + e_l\otimes e_k)$ for $k,l=1,2,3$.\footnote{Note that $\{E_{kl} :1\le k\le l\}$ is an orthogonal basis of $\R^{3\times 3}_{\rm sym}$.}
Now, let $\eta_{kl}\in H^1_\#(Y)^3$, $\xi_{kl} \in H^1_\#(Y^g)^3$, $q\in H^1_\#(Y)^3$, and $\omega_k \in H^1_\#(Y^g)$ be solutions to the following elliptic problems:\footnote{Note that these solutions are unique up to constants.}
\begin{subequations}
\begin{align}
\int_Y \mathcal{C}\bigl(E_{kl} + e_y(\eta_{kl})\bigr) : e_y(\psi)\di{y} &= 0
  \quad \forall\,\psi \in H^1_\#(Y)^3,\label{eq:cell:elast}\\
  \int_{Y^g} C^g\bigl(E_{kl} + e_y(\xi_{kl})\bigr) : e_y(\psi)\di{y} &= 0
  \quad \forall\,\psi \in H^1_\#(Y^g)^3,\label{eq:cell:damp}\\
  \int_{Y}\mathcal{C}(y)e_y(q):e_y(\psi)\di{y}-\alpha\int_{\Gamma}\psi\cdot n\di{\sigma}&=0\quad \forall\,\psi \in H^1_\#(Y)^3,\label{eq:cell:pressure}\\
  \int_{Y^g} K(e_k + \nabla_y\omega_k) \cdot \nabla_y\phi\di{y} &= 0
  \quad \forall\,\phi \in H^1_\#(Y^g)\label{eq:cell:perm}.
\end{align}
Here, $C(y) = C^f$ in $Y^f$ and $C(y) = C^g$ in $Y^g$ and $e_y$ denotes the symmetric gradient with respect to $y\in Y$.
In addition, we introduce the following relaxation problems: Find $\zeta_{kl}\in H^1(S;H^1_\#(Y))^3$ and $\theta\in H^1(S;H^1_\#(Y))^3$ with initial conditions $\zeta_{kl}(0)=0$ and $\theta(0)=0$ and such that
\begin{align}
    \int_{ Y}\mathcal{C}e_y(\zeta_{kl}):e_y(\psi)\di{y}
    +\delta\int_{Y^g}\mathcal C^ge_y(\partial_t\zeta_{kl}):e_y(\psi)\di{y}
        &=-\int_{Y^g}\mathcal{C}^g(E_{kl}+e_y(\eta_{kl})):e_y(\psi)\di{y},\label{eq:cell_damping_kl}\\
      \int_{ Y}\mathcal{C}e_y(\theta):e_y(\psi)\di{y}
        +\delta\int_{Y^g}\mathcal C^ge_y(\partial_t\theta):e_y(\psi)\di{y}
        &=-\int_{Y^g}\mathcal C^ge_y(q):e_y(\psi)\di{y}\label{eq:cell_damping}  
\end{align}
\end{subequations}
for all $\psi\in H^1_\#(Y)^3$.
Again, we use $\zeta=(\zeta_{k,l})$ to denote the full matrix valued function.

Note that there is no explicit time dependency in these forcing terms on the right hand side, so these problems simply describe the evolution in time with respect to a constant (in time) forcing term.
These cell problems are used to account for the relaxation part:
For example, for $f=f(t,x)$, the function
\[
z_f(t,x,y)=\int_0^t\partial_t\theta(t-s,y)f(s,x)\di{s}
\]
solves the problem 
\begin{multline*}
  \int_{ Y}\mathcal{C}e_y(z_f):e_y(\psi)\di{y}
        +\delta\int_{Y^g}\mathcal C^ge_y(\partial_tz_f):e_y(\psi)\di{y}\\
        =-f(t,x)\int_{Y^g}\mathcal C^ge_y(q):e_y(\psi)\di{y} \quad \forall\,\psi \in H^1_\#(Y)^3.
\end{multline*}

Based on these cell solutions, we define the elasticity tensor $\mathbb C^{\rm h}\in \R^{3\times3\times3\times3}$, the damping tensor $\mathbb D^{\rm h}\in \R^{3\times3\times3\times3}$, the effective pressure tensor $A^{\rm h}\in \R^{3\times3}_{\rm sym}$, the effective storage content $c^{\rm h}\in \R$, and the effective permeability tensor $K^{\rm h}\in\R^{3\times3}_{\rm sym}$ via
\begin{subequations}
\begin{align}
    \mathbb C^{\rm h}_{ijkl}:
  &= \int_Y C(y)\bigl(E_{kl} + e_y(\eta_{kl})\bigr) : e_{ij}\di{y},\label{eq:Ceff}\\
   \mathbb D^{\rm h}_{ijkl}
  &:= \int_{Y^g} C^g\bigl(E_{kl} + e_y(\xi_{kl})\bigr) : e_{ij}\di{y},\label{eq:Deff}\\
   K^{\rm h}_{kl}
  &:= \int_{Y^g} K(e_k + \nabla_y\omega_k) \cdot e_l\di{y}\label{eq:Keff},\\
   c^{\rm h} &:= c|Y^g| + \alpha\int_{Y^g} \dive_yq \,\di{y}, \label{eq:Qhom-def}\\
     A^{\rm h}_{ij}
    &:=
    \alpha |Y^g|\delta_{ij}
    -
    \int_Y \big(\mathcal{C}(y) e_y(q)\big)_{ij}\,\di{y}, \label{eq:Ahom-def}
\end{align}
and the memory kernels are given by
\begin{align}
\mathbb{M}^{\rm h}_{ij\,kl}(t)
&:= \int_Y \bigl(\mathcal{C}(y)\,e_y(\zeta_{kl}(t))\bigr)_{ij}\,\di{y}
+\delta\int_{Y^g}\bigl(\mathcal{C}^g e_y(\partial_t\zeta_{kl}(t))\bigr)_{ij}\,\di{y},\label{eq:Mhom}\\
R^{\rm h}_{ij}(t)
&:= \int_Y \bigl(\mathcal{C}(y)\,e_y(\theta(t))\bigr)_{ij}\,\di{y}
+\delta\int_{Y^g}\bigl(\mathcal{C}^g e_y(\partial_t\theta(t))\bigr)_{ij}\,\di{y},\label{eq:Rhom}\\
H^{\rm h}(t) &:= \alpha\int_{Y^g} \dive_y\bigl(\partial_t\theta(t,y)\bigr) \,\di{y}\label{eq:Hhom}.
\end{align}
\end{subequations}
Note that, for $\eta=0$, convolution terms involving $\partial_t p$ (as in \cref{thm:homog}) are understood in the distributional sense in time.
More precisely, for $k\in H^1(S)$ we set
\[
    (k*\partial_t p)(t)
    :=
    k(0)p(t)-k(t)p_0+(\partial_t k*p)(t).
\]
This definition agrees with the usual convolution whenever $\partial_t p\in L^2(S\times\Omega)$.

\begin{table}[ht]
\centering
\footnotesize
\renewcommand{\arraystretch}{1.25}
\begin{tabular}{
    >{\centering\arraybackslash}p{1.4cm}
    >{\raggedright\arraybackslash}p{2.1cm}
    >{\raggedright\arraybackslash}p{3.0cm}
    >{\raggedright\arraybackslash}p{5.5cm}
}
\hline
\textbf{Quantity}
& \textbf{Cell problem}
& \textbf{Macroscopic unit loading}
& \textbf{Physical Interpretation}
\\
\hline

$\mathbb C^h$
& $\boldsymbol\eta_{kl}$ on $Y$
& Symmetric strain $E_{kl}$
& Effective elastic stiffness
\\

$\mathbb D^h$
& $\boldsymbol\xi_{kl}$ on $Y^g$
& Symmetric strain rate $E_{kl}$
& Effective instantaneous damping tensor
\\

$K^h$
& $\omega_k$ on $Y^g$
& Pressure gradient $e_k$
& Effective Darcy permeability
\\

$c^h$
& $q$ on $Y$
& Pore pressure
& Effective storage coefficient
\\

$A^h$
& $q$ on $Y$
& Pore pressure
& Effective Biot coupling tensor
\\

$\mathbb M^h(t)$
& $\boldsymbol\zeta_{kl}(t)$ on $Y$
& Strain relaxation
& Strain-to-stress memory kernel
\\

$R^h(t)$
& $\theta(t)$ on $Y$
& Pressure relaxation
& Pressure-to-stress and strain-to-content memory kernel
\\

$H^h(t)$
& $\theta(t)$ on $Y$
& Pressure relaxation
& Pressure-to-fluid-content memory kernel
\\
\hline
\end{tabular}
\caption{Summary of the effective tensors and memory kernels.}
\label{tab:effective_coefficients}
\end{table}

For all cell functions, we choose representatives with zero mean for definiteness.
Note that the effective coefficients and memory kernels are independent of this choice, since they depend only on the corresponding gradients or symmetric gradients.

Standard arguments show that $\mathbb C^{\rm h}$ and $\mathbb D^{\rm h}$ have the usual symmetries and are coercive on $\R_{\rm sym}^{3\times3}$, while $K^{\rm h}$ is symmetric and positive definite. 
Moreover, $A^{\rm h}\in\R_{\rm sym}^{3\times3}$ and, by testing
\cref{eq:cell:pressure} with $q$,
\[
    c^{\rm h}
    =
    c|Y^g|
    +
    \int_Y \mathcal C e_y(q):e_y(q)\,\di y
    \ge c|Y^g|>0.
\]
Finally, the relaxation problem solutions satisfy (see \cref{lem:relaxation_cell})
\[
    \zeta_{kl},\theta
    \in H^2\bigl(S;H^1_\#(Y)^3\bigr).
\]
As a consequence, the memory kernels have the time regularity required below; in particular,
\[
    M^h\in H^1(S),\qquad
    R^h\in H^2(S),\qquad
    H^h\in H^1(S).
\]
which ensures that all convolution terms appearing below are well defined.

\begin{remark}[Decay of the relaxation response]
Please note that the relaxation cell problems also exhibit a natural dissipative structure. 
To see this, suppose that the problems \eqref{eq:cell_damping_kl}--\eqref{eq:cell_damping} are extended to $t\in[0,\infty)$ and set
\[
    v_{kl}:=\partial_t\boldsymbol\zeta_{kl},\qquad
    E_{kl}(t):=
    \int_{Y^g}
        \mathbb C^g e_y(v_{kl}):e_y(v_{kl})\,\di y.
\]
Differentiating \eqref{eq:cell_damping_kl} in time and testing with $v_{kl}$ gives
\[
\frac{\delta}{2}\frac{\rm d}{\di t}
\int_{Y^g}
\mathbb C^g e_y(v_{kl}):e_y(v_{kl})\di y
+
\int_Y
\mathbb C(y)e_y(v_{kl}):e_y(v_{kl})\di y
=0.
\]
Since $\mathbb C=\mathbb C^g$ on $Y^g$ and both elasticity tensors are
positive definite,
\[
    \int_Y
        \mathbb C(y)e_y(v_{kl}):e_y(v_{kl})\,\di y
    \geq E_{kl}(t).
\]
Consequently,
\[
    E_{kl}'(t)+\frac{2}{\delta}E_{kl}(t)\leq0,
\]
which implies
\[
    E_{kl}(t)\leq E_{kl}(0)e^{-2t/\delta}.
\]
The same argument applies to $\partial_t\theta$. 
This is consistent with the interpretation of the resulting convolution term as fading-memory effects.
\end{remark}

\subsection{Proof of \cref{thm:homog}}
To derive the two-scale limit equations, we first take smooth test
functions and use the standard oscillating ansatz
\[
    \psi_\e(t,x)
    =
    \psi_0(t,x)
    +
    \e\psi_1\left(t,x,\frac{x}{\e}\right)
\]
in the mechanical equation. Then
\[
    \nabla\psi_\e
    =
    \nabla_x\psi_0
    +
    \nabla_y\psi_1\left(t,x,\frac{x}{\e}\right)
    +
    \mathcal O(\e).
\]
Similarly, in the time-integrated mass-balance equation we use
\[
    \phi_\e(t,x)
    =
    \phi_0(t,x)
    +
    \e\phi_1\left(t,x,\frac{x}{\e}\right),
    \qquad \phi_\e(T)=0.
\]
Here,
\[
    \partial_t\phi_\e
    =
    \partial_t\phi_0
    +
    \e\partial_t\phi_1,
    \qquad
    \nabla\phi_\e
    =
    \nabla_x\phi_0+\nabla_y\phi_1+\mathcal O(\e).
\]
Now, using \cref{prop:twoscale}, the properties of the unfolding operator, together with the convergence of the initial fluid contents from Assumption~\textbf{(A3)}, we can see that the two-scale limit
\[
(u,u_1,p,p_1)
\]
satisfies, for almost every \(t\in S\),
\begin{multline}
\int_{\Omega}\int_Y
\mathcal C(y)
\bigl(e_x(u)+e_y(u_1)\bigr)
:
\bigl(e_x(\psi_0)+e_y(\psi_1)\bigr)
\,\di{y}\di{x}\\
+\delta
\int_{\Omega}\int_{Y^g}
\mathcal C^g
\bigl(e_x(\partial_tu)+e_y(\partial_tu_1)\bigr)
:
\bigl(e_x(\psi_0)+e_y(\psi_1)\bigr)
\,\di{y}\di{x}\\
-\alpha
\int_{\Omega}\int_{Y^g}
p\,
\bigl(
\dive_x\psi_0+\dive_y\psi_1
\bigr)
\,\di{y}\di{x}
=
\int_{\Omega}f\cdot\psi_0\di{x},
\label{eq:two-scale-limit-mechanics}
\end{multline}
for all
\[
\psi_0\in H_0^1(\Omega)^3,
\qquad
\psi_1\in
L^2\bigl(\Omega;H_\#^1(Y)^3\bigr).
\]
Moreover,
\begin{multline}
-
\int_S\int_{\Omega}\int_{Y^g}
\Bigl[
cp
+
\alpha\bigl(
\dive_xu+\dive_yu_1
\bigr)
+
\eta\bigl(
\dive_x\partial_tu
+
\dive_y\partial_tu_1
\bigr)
\Bigr]
\partial_t\varphi_0
\,\di{y}\di{x}\di{t}\\
+
\int_S\int_{\Omega}\int_{Y^g}
K
\bigl(\nabla_xp+\nabla_yp_1\bigr)
\cdot
\bigl(\nabla_x\varphi_0+\nabla_y\varphi_1\bigr)
\,\di{y}\di{x}\di{t}\\
=
\int_S\int_{\Omega}|Y^g|g\varphi_0\di{x}\di{t}
+
\int_{\Omega}\zeta_0
\varphi_0(0)\di{x},
\label{eq:two-scale-limit-pressure}
\end{multline}
for all admissible test functions
\[
\varphi_0\in H^1\bigl(S;H^1(\Omega)\bigr),
\qquad
\varphi_1\in
L^2\bigl(S\times\Omega;H_\#^1(Y^g)\bigr),
\]
such that
\[
\varphi_0(T)=0,
\qquad
\varphi_1(T)=0.
\]
The initial traces of the microscopic correctors have already been identified in Proposition~\ref{prop:twoscale}.
In particular,
\[
    u_1(0)=u_{1,0}^{\rm eq}.
\]
If $\eta=\alpha\delta$, the auxiliary corrector satisfies
\[
    w_1^g(0)=w_{1,0}|_{Y^g}.
\]
Together with the compatibility relations following from
{\rm (A3)--(A4)}, this implies
\[
    c|Y^g|\bigl(p(0)-p_0\bigr)=0,
\]
and hence $p(0)=p_0$.

The next step is to decouple the problem into its macroscopic and microscopic parts.
\subsubsection{Decoupling the mechanical part}
This is somewhat delicate because of the strong coupling between the elastic fibre phase and the viscoelastic gel phase.
Owing to this coupling, the macroscopic material response cannot, in general, be represented solely by instantaneous effective elasticity and viscosity tensors.
Indeed, the microscopic displacement is determined by the problem
\begin{multline}\label{eq:micro_limit}
\int_{\Omega}\int_Y
\mathcal C(y)
\bigl(e_x(u)+e_y(u_1)\bigr)
:e_y(\psi_1)
\,\di{y}\di{x}\\
+\delta
\int_{\Omega}\int_{Y^g}
\mathcal C^g
\bigl(e_x(\partial_tu)+e_y(\partial_tu_1)\bigr)
:e_y(\psi_1)
\,\di{y}\di{x}-\int_{\Omega}\alpha p\int_{Y^g}\dive_y\psi_1\di{y}
=0,
\end{multline}
This microscopic equilibrium couples the instantaneous elastic response of the entire composite with the viscous response of the gel phase.
The microscopic deformation therefore redistributes dynamically between the two phases: the fibre responds purely elastically, whereas the gel exhibits a time-dependent Kelvin--Voigt response. Consequently, the microscopic displacement at a given time is not determined solely by the current macroscopic strain and pore pressure, but also by their loading histories.
After eliminating the microscopic displacement, this history dependence appears in the macroscopic constitutive law in the form of memory terms.

By Proposition~\ref{prop:twoscale},
\[
    u_1(0,x,y)
    =
    u_{1,0}^{\rm eq}(x,y)
    =
    \boldsymbol\eta(y):e(u_0(x))+q(y)p_0(x).
\]
Now, owing to the linear structure of the homogenization limit, we see that (the calculations are presented in \cref{lemma:u1charact})
\[
u_1=\eta:e_x(u)+qp+\delta\partial_t\zeta*e(\partial_tu)+\delta\partial_t\theta*\partial_tp
\]
The first two terms represent the instantaneous microscopic response to the current macroscopic strain and pore pressure and he convolution terms describe the delayed response generated by the strain-rate and pressure-rate histories.

Now, we introduce the macroscopic elastic--poroelastic and viscous
stresses by
\begin{align*}
\Sigma_{\rm el}^{\rm h}(u,p)&=\int_Y
\mathcal C(y)
\bigl(e(u)+e_y(u_1)\bigr)
\di{y}
-\alpha|Y^g|p,\\
\Sigma^{\rm h}_{\rm vis}(u,p)&=
\int_{Y^g}
\mathcal C^g
\bigl(e(\partial_tu)+e_y(\partial_tu_1)\bigr)
\di{y},
\end{align*}
and get (via \cref{eq:two-scale-limit-mechanics})
\[
\int_\Omega\left(\Sigma^{\rm h}_{\rm el}(u,p)+\delta\Sigma^{\rm h}_{\rm vis}(\partial_tu,\partial_tp)\right): e(\psi_0)\di{x}=\int_\Omega f\cdot\psi_0\di{x}.
\]
It is a bit technical to express $\Sigma^{\rm h}(u,p)$ explicitly in terms of cell functions (i.e., without relying on $u_1$) but using above characterization of $u_1$ this is clearly possible.
A technical but rather straightforward calculation shows that (we refer to \cref{lemma:characterization_limit_stresses})
\begin{align*}
\Sigma^{\rm h}_{\rm el}(u,p)&=
\mathbb{C}^{\rm h}e(u)-A^{\rm h}p,\\
\Sigma^{\rm h}_{\rm vis}(\partial_tu,\partial_tp)&=
\mathbb D^{\rm h}e(\partial_tu)
+\partial_t\mathbb M^{\rm h}*e(\partial_tu)+\partial_t\mathbb{R}^{\rm h}*\partial_tp.
\end{align*}

\subsubsection{Decoupling the poroelastic part}
For the poroelastic part, we first use
\[
  p_1(t,x,y) = \sum_{k=1}^3 \partial_{x_k}p(t,x)\,\omega_k(y)
\]
and $K^{\rm h}$ to slightly simplify \cref{eq:two-scale-limit-pressure}:
\begin{multline*}
-
\int_S\int_{\Omega}\int_{Y^g}
\Bigl[
cp
+
\alpha\bigl(
\dive_xu+\dive_yu_1
\bigr)
+
\eta\bigl(
\dive_x\partial_tu
+
\dive_y\partial_tu_1
\bigr)
\Bigr]
\partial_t\varphi
\,\di{y}\di{x}\di{t}\\
+
\int_S\int_{\Omega}
K^{\rm h}\nabla p
\cdot\nabla\varphi\di{x}\di{t}
=\int_S\int_{\Omega} |Y^g|g\varphi\di{x}\di{t}
+\int_{\Omega}
\zeta_0
\varphi(0)\di{x}
\end{multline*}
for all $\varphi\in H^1(S;H^1(\Omega))$ with $\varphi(T)=0$.
Using the representation of $u_1$ obtained above, the averaged microscopic fluid content can be expressed entirely in terms of the macroscopic variables.
More precisely, \cref{lemma:pressure_macroscopic} gives
\[
    \int_{Y^g}
    \Bigl[
        cp
        +\alpha\bigl(\dive_xu+\dive_yu_1\bigr)
        +\eta\bigl(
            \dive_x\partial_tu+\dive_y\partial_tu_1
        \bigr)
    \Bigr]
    \,\di y
    =
    Q_\eta^{\rm h}(u,p),
\]
where
\[
    Q_\eta^{\rm h}(u,p)
    :=
    \begin{cases}
        F^{\rm h}(u,p),
        & \eta=0,\\[1mm]
        F^{\rm h}(u,p)
        +\delta\partial_tG^{\rm h}(u,p),
        & \eta=\alpha\delta,
    \end{cases}
\]
with
\[
    F^{\rm h}(u,p)
    :=
    c^{\rm h}p
    +A^{\rm h}:e(u)
    +\delta
    \left(
        -\partial_tR^{\rm h}*e(\partial_tu)
        +H^{\rm h}*\partial_tp
    \right)
\]
and
\[
    G^{\rm h}(u,p)
    :=
    F^{\rm h}(u,p)-c|Y^g|p.
\]
%
Consequently,
\[
-\int_S
    \left\langle
        Q_\eta^{\rm h}(u,p),\partial_t\varphi
    \right\rangle
\,\di t
+
\int_S\int_\Omega
    K^{\rm h}\nabla p\cdot\nabla\varphi
\,\di x\,\di t=
\int_S\int_\Omega
    |Y^g|g\varphi
\,\di x\,\di t+
\int_\Omega
    \zeta_0\varphi(0)
\di x.
\]
Again, the detailed calculations are given in the appendix, see \cref{lemma:pressure_macroscopic}.

\appendix
\section{Calculations}
Here, we present some of the somewhat technical calculations used in the limit analysis.
These calculations are pretty straightforward but rather tedious, so we present them here for the convenience of the reader.

We start with a short lemma concerning the regularity of the solutions to the relaxation cell problem.
\begin{lemma}
\label{lem:relaxation_cell}
For every $k,l\in\{1,2,3\}$, the problems
\eqref{eq:cell_damping_kl}--\eqref{eq:cell_damping} admit unique zero-mean solutions satisfying
\[
    \boldsymbol\zeta_{kl},\theta
    \in
    H^2\bigl(S;H_\#^1(Y)^3\bigr).
\]
\end{lemma}

\begin{proof}
Let
\[
    V:=H_\#^1(Y)^3/\R^3
\]
and define
\[
    a(v,\psi)
    :=
    \int_Y C(y)e_y(v):e_y(\psi)\,dy,
    \qquad
    b(v,\psi)
    :=
    \int_{Y^g}C^g e_y(v):e_y(\psi)\,dy.
\]
The form $a$ is coercive on $V$, whereas $b$ is non-negative.
Set
\[
    K:=\{v\in V:\ b(v,v)=0\}.
\]
The right-hand sides in \eqref{eq:cell_damping_kl}-\eqref{eq:cell_damping} vanish on $K$ which means that the solutions belong to the $a$-orthogonal complement 
\[
K^{\perp_a}:=\{v\in V\ : \ a(v,k)=0 \ \text{for all}\ k\in K\}.
\]
For $v\in K^{\perp_a}$, let $\mathcal E^g(v|_{Y^g})\in V$
denote a periodic $H^1$-extension of $v|_{Y^g}$ to $Y$.
Then
\[
    v-\mathcal E^g(v|_{Y^g})\in K,
\]
and hence, by the $a$-orthogonality,
\[
    a(v,v)
    =
    a\bigl(v,\mathcal E^g(v|_{Y^g})\bigr).
\]
Using the continuity of $a$, the periodic extension estimate, and the
periodic Korn inequality on the connected phase $Y^g$, we obtain
\[
    \|v\|_V
    \lesssim
    \|v\|_{H^1(Y^g)/\mathbb R^3}
    \lesssim
    \|e_y(v)\|_{L^2(Y^g)}
    \lesssim
    b(v,v)^{1/2}.
\]
Therefore, $b$ is coercive on $K^{\perp_a}$.
Consequently, the relaxation problems reduce to linear evolution equations of the form
\[
    \delta b(\partial_t z,\psi)+a(z,\psi)=F(\psi),
    \qquad z(0)=0.
\]
Standard Hilbert-space evolution theory therefore gives a unique
solution with the stated $H^1$-in-time regularity.
Since $F$ is independent of time, differentiating the evolution equation in time upgrades the time regularity to $H^2$.\footnote{With the same argument, we can show that the solutions are arbitrarily smooth in time, but $H^2$ is all that we need.}
\end{proof}

Now, we go on to show some of the characterizations used in the homogenization process.
We start by validating the claimed structure of $u_1$:
\[
u_1=\eta:e(u)+qp+\delta\partial_t\zeta*e(\partial_tu)+\delta\partial_t\theta*\partial_tp
\]
Localizing the microscopic limit problem \cref{eq:micro_limit} in $x\in\Omega$
\begin{multline*}
\int_Y
\mathcal C(y)
\bigl(e(u)+e_y(u_1)\bigr)
:e_y(\psi_1)
\,\di{y}\\
+\delta
\int_{Y^g}
\mathcal C^g
\bigl(e(\partial_tu)+e_y(\partial_tu_1)\bigr)
:e_y(\psi_1)
\,\di{y}-\alpha p\int_{Y^g}\dive_y\psi_1\di{y}
=0,
\end{multline*}
which is valid for all $\psi\in H^1_\#(Y)$ and for almost all $(t,x)\in S\times\Omega$.
Writing
\[
e(u) = \sum_{k,l=1}^3 e_{kl}(u)E_{kl}, \qquad \eta:e(u) = \sum_{k,l=1}^3 \eta_{kl}\,e_{kl}(u),
\]
we obtain, after inserting the representation of \(u_1\) and taking all quantities independent of \(y\) outside the corresponding cell integrals,
\begin{align*} 0={}& \sum_{k,l=1}^3 e_{kl}(u(t,x)) \int_Y \mathcal C(y) \bigl(E_{kl}+e_y(\eta_{kl}(y))\bigr) :e_y(\psi_1) \,\di{y} \\
&+ p(t,x) \left[ \int_Y \mathcal C(y)e_y(q(y)):e_y(\psi_1) \,\di{y} - \alpha \int_{Y^g} \dive_y\psi_1 \,\di{y} \right] \\
&+ \delta \sum_{k,l=1}^3 \int_0^t e_{kl}(\partial_su(s,x)) \left[ \int_Y \mathcal C(y) e_y\bigl(\partial_t\zeta_{kl}(t-s,y)\bigr) :e_y(\psi_1) \,\di{y} \right] \di{s} \\
&+ \delta \int_0^t \partial_sp(s,x) \left[ \int_Y \mathcal C(y) e_y\bigl(\partial_t\theta(t-s,y)\bigr) :e_y(\psi_1) \,\di{y} \right] \di{s} \\
&+ \delta \sum_{k,l=1}^3 e_{kl}(\partial_tu(t,x)) \int_{Y^g} \mathcal C^g \bigl(E_{kl}+e_y(\eta_{kl}(y))\bigr) :e_y(\psi_1) \,\di{y} \\
&+ \delta\,\partial_tp(t,x) \int_{Y^g} \mathcal C^g e_y(q(y)):e_y(\psi_1) \,\di{y} \\
&+ \delta^2 \sum_{k,l=1}^3 \int_{Y^g} \mathcal C^g e_y\left( \partial_t \left[ \int_0^t \partial_t\zeta_{kl}(t-s,y)\, e_{kl}(\partial_su(s,x)) \,\di{s} \right] \right) :e_y(\psi_1) \,\di{y} \\ 
&+ \delta^2 \int_{Y^g} \mathcal C^g e_y\left( \partial_t \left[ \int_0^t \partial_t\theta(t-s,y)\, \partial_sp(s,x) \,\di{s} \right] \right) :e_y(\psi_1) \,\di{y}\\
&=:\sum_{i=1}^8I_i.
\end{align*}

\begin{lemma}\label{lemma:u1charact}
The microscopic corrector is the unique solution of
\cref{eq:micro_limit} in
\[
    H^1\bigl(S;L^2(\Omega;H_\#^1(Y)^3/\mathbb R^3)\bigr)
\]
satisfying
\[
    u_1(0)=u_{1,0}^{\rm eq}.
\]
It is given by
\[
    u_1
    =
    \boldsymbol\eta:e(u)+qp
    +\delta\partial_t\boldsymbol\zeta*e(\partial_tu)
    +\delta\partial_t\theta*\partial_tp.
\]
\end{lemma}
\begin{proof}
Let
\[
    \widetilde u_1
    :=
    \boldsymbol\eta:e(u)+qp
    +\delta\partial_t\boldsymbol\zeta*e(\partial_tu)
    +\delta\partial_t\theta*\partial_tp.
\]
We first verify that $\widetilde u_1$ satisfies
\cref{eq:micro_limit}. After inserting this representation into
\cref{eq:micro_limit}, it remains to show that
\[
    \sum_{i=1}^8 I_i=0.
\]
Now, it is immediately clear that $I_1=0$ because $\eta_{kl}\in H^1_\#(Y)^3$ ($k,l=1,2,3$) solves \cref{eq:cell:elast} and $I_2=0$ because $q\in H^1_\#(Y)^3$ solves \cref{eq:cell:pressure}.

Next we look at $I_3+I_5+I_7$:
\begin{multline*}
I_3+I_5+I_7=\delta \sum_{k,l=1}^3 \int_0^t e_{kl}(\partial_su(s,x)) \left[ \int_Y \mathcal C(y) e_y\bigl(\partial_t\zeta_{kl}(t-s,y)\bigr) :e_y(\psi_1) \,\di{y} \right] \di{s} \\
+ \delta \sum_{k,l=1}^3 e_{kl}(\partial_tu(t,x)) \int_{Y^g} \mathcal C^g \bigl(E_{kl}+e_y(\eta_{kl}(y))\bigr) :e_y(\psi_1) \,\di{y} \\
+ \delta^2 \sum_{k,l=1}^3 \int_{Y^g} \mathcal C^g e_y\left( \partial_t \left[ \int_0^t \partial_t\zeta_{kl}(t-s,y)\, e_{kl}(\partial_su(s,x)) \,\di{s} \right] \right) :e_y(\psi_1) \,\di{y}.
\end{multline*}
For fixed \((t,x)\in S\times\Omega\), define
\[
Z_{kl}(t,x,y)
:=
\int_0^t
\partial_t\zeta_{kl}(t-s,y)\,
e_{kl}(\partial_su(s,x))
\,\di{s}
\]
and recall that \(\zeta_{kl}\in H^1(S;H^1_\#(Y))\cap H^2(S;H^1_\#(Y^g))\) solves \cref{eq:cell_damping_kl} with $\zeta_{kl}(0)=0$.
Differentiating \cref{eq:cell_damping_kl} with respect to time gives
\begin{equation}
\int_Y
\mathcal C(y)e_y(\partial_t\zeta_{kl}(t))
:e_y(\phi)\,\di{y}
+
\delta\int_{Y^g}
\mathcal C^g e_y(\partial_{tt}\zeta_{kl}(t))
:e_y(\phi)\,\di{y}
=0.
\label{eq:relaxation-cell-zeta-differentiated}
\end{equation}
Moreover, evaluating \cref{eq:cell_damping_kl} at \(t=0\) yields
\begin{equation}
\delta\int_{Y^g}
\mathcal C^g e_y(\partial_t\zeta_{kl}(0))
:e_y(\phi)\,\di{y}
=
-
\int_{Y^g}
\mathcal C^g
\bigl(E_{kl}+e_y(\eta_{kl})\bigr)
:e_y(\phi)\,\di{y}.
\label{eq:initial-trace-zeta}
\end{equation}
Using Leibniz' rule, we obtain
\begin{align*}
\partial_t Z_{kl}(t,x,y)
=
\partial_t\zeta_{kl}(0,y)\,
e_{kl}(\partial_tu(t,x))+
\int_0^t
\partial_{tt}\zeta_{kl}(t-s,y)\,
e_{kl}(\partial_su(s,x))
\,\di{s}.
\end{align*}
Consequently,
\begin{multline*}
\int_Y
\mathcal C(y)e_y(Z_{kl}(t,x))
:e_y(\psi_1)\,\di{y}
+\delta\int_{Y^g}
\mathcal C^g e_y(\partial_tZ_{kl}(t,x))
:e_y(\psi_1)\,\di{y}\\
=
\int_0^t
e_{kl}(\partial_su(s,x))
\Bigg[
\int_Y
\mathcal C(y)
e_y(\partial_t\zeta_{kl}(t-s))
:e_y(\psi_1)\,\di{y}
+
\delta\int_{Y^g}
\mathcal C^g
e_y(\partial_{tt}\zeta_{kl}(t-s))
:e_y(\psi_1)\,\di{y}
\Bigg]
\,\di{s}\\
+\delta\,e_{kl}(\partial_tu(t,x))
\int_{Y^g}
\mathcal C^g e_y(\partial_t\zeta_{kl}(0))
:e_y(\psi_1)\,\di{y}.
\end{multline*}
The term with the time integral vanishes by
\cref{eq:relaxation-cell-zeta-differentiated}, while the last term is,
by \cref{eq:initial-trace-zeta},
\[
-
e_{kl}(\partial_tu(t,x))
\int_{Y^g}
\mathcal C^g
\bigl(E_{kl}+e_y(\eta_{kl})\bigr)
:e_y(\psi_1)\,\di{y}.
\]
Hence,
\begin{multline*}
\int_Y
\mathcal C(y)e_y(Z_{kl})
:e_y(\psi_1)\,\di{y}
+
\delta\int_{Y^g}
\mathcal C^g e_y(\partial_tZ_{kl})
:e_y(\psi_1)\,\di{y}
\\
+
e_{kl}(\partial_tu)
\int_{Y^g}
\mathcal C^g
\bigl(E_{kl}+e_y(\eta_{kl})\bigr)
:e_y(\psi_1)\,\di{y}
=0.
\end{multline*}
Multiplying this identity by \(\delta\), summing over
\(k,l\in\{1,2,3\}\), and recalling the definition of \(Z_{kl}\), we
obtain
\[
I_3+I_5+I_7=0.
\]

Finally, we check the remaining terms,
\begin{multline}
I_4+I_6+I_8= \delta \int_0^t \partial_sp(s,x) \left[ \int_Y \mathcal C(y) e_y\bigl(\partial_t\theta(t-s,y)\bigr) :e_y(\psi_1) \,\di{y} \right] \di{s} \\
+ \delta\,\partial_tp(t,x) \int_{Y^g} \mathcal C^g e_y(q(y)):e_y(\psi_1) \,\di{y} \\
+ \delta^2 \int_{Y^g} \mathcal C^g e_y\left( \partial_t \left[ \int_0^t \partial_t\theta(t-s,y)\, \partial_sp(s,x) \,\di{s} \right] \right) :e_y(\psi_1) \,\di{y}.
\end{multline}
The argument showing $I_4+I_6+I_8=0$ is entirely analogous with $\theta$ instead of $\zeta_{kl}$.
For fixed \((t,x)\in S\times\Omega\), define
\[
Z_p(t,x,y)
:=
\int_0^t
\partial_t\theta(t-s,y)\,
\partial_sp(s,x)\,\di{s}.
\]
By Leibniz' rule,
\[
\partial_t Z_p(t,x,y)
=
\partial_t\theta(0,y)\,
\partial_tp(t,x)
+
\int_0^t
\partial_{tt}\theta(t-s,y)\,
\partial_sp(s,x)\,\di{s}.
\]
Moreover,
\[
\delta\,\partial_tp(t,x)
\int_{Y^g}
\mathcal C^g e_y(\partial_t\theta(0))
:e_y(\psi_1)\,\di{y}
=-
\partial_tp(t,x)
\int_{Y^g}
\mathcal C^g e_y(q)
:e_y(\psi_1)\,\di{y}.
\]
Therefore,
\[
\int_Y
\mathcal C(y)e_y(Z_p)
:e_y(\psi_1)\,\di{y}
+
\delta\int_{Y^g}
\mathcal C^g e_y(\partial_tZ_p)
:e_y(\psi_1)\,\di{y}
+
\partial_tp
\int_{Y^g}
\mathcal C^g e_y(q)
:e_y(\psi_1)\,\di{y}
=0.
\]
Multiplying by \(\delta\), we conclude that
\[
I_4+I_6+I_8=0.
\]

It remains to verify the initial condition. Since the convolution
integrals vanish at $t=0$, we have
\[
    \widetilde u_1(0)
    =
    \boldsymbol\eta:e(u(0))+q\,p(0).
\]
By the previously established identities
\[
    u(0)=u_0,
    \qquad
    p(0)=p_0,
\]
we obtain
\[
    \widetilde u_1(0)
    =
    \boldsymbol\eta:e(u_0)+q\,p_0
    =
    u_{1,0}^{\rm eq}.
\]
Thus, $\widetilde u_1$ satisfies both \cref{eq:micro_limit} and the prescribed initial condition.
Uniqueness can easily be checked via an energy estimate.
Therefore, $\widetilde u_1=u_1$.
\end{proof}

\begin{lemma}\label{lemma:characterization_limit_stresses}
    The macroscopic stresses are given via
    \[
    \int_\Omega\left(\Sigma^{\rm h}_{\rm el}(u,p)+\delta\Sigma^{\rm h}_{\rm vis}(\partial_tu,\partial_tp)\right): e(\psi_0)\di{x}=\int_\Omega f\cdot\psi_0\di{x}
    \]
    where
    \begin{align*}
    \Sigma^{\rm h}_{\rm el}(u,p)&=
    \mathbb{C}^{\rm h}e(u)-A^{\rm h}p,\\
    \Sigma^{\rm h}_{\rm vis}(\partial_tu,\partial_tp)&=
    \mathbb D^{\rm h}e(\partial_tu)
    +\partial_t\mathbb M^{\rm h}*e(\partial_tu)+\partial_t\mathbb{R}^{\rm h}*\partial_tp.
    \end{align*}
\end{lemma}
\begin{proof}
Choosing \(\psi_1=0\) in the two-scale limit equation
\cref{eq:micro_limit}, we obtain
\begin{multline}
\int_\Omega
\Bigg[\int_Y\mathcal C(y)\bigl(e(u)+e_y(u_1)\bigr)\,\di{y}-\alpha |Y^g|\,I_3\\
+\delta\int_{Y^g}\mathcal C^g\bigl(e(\partial_tu)+e_y(\partial_tu_1)\bigr)\,\di{y}\Bigg]:e(\psi_0)\,\di{x}
=\int_\Omega f\cdot\psi_0\,\di{x}.
\label{eq:macro-mechanics-before-identification}
\end{multline}
We now insert the representation (\cref{lemma:u1charact})
\[
u_1=
\eta:e(u)+qp
+\delta\,\partial_t\zeta*e(\partial_tu)
+\delta\,\partial_t\theta*\partial_tp.
\]
into the averaged elastic stress:
\begin{align*}
\int_Y
\mathcal C(y)
\bigl(e(u)+e_y(u_1)\bigr)
\,\di{y}
-\alpha |Y^g|p\,I
={}&
\sum_{k,l=1}^3
e_{kl}(u)
\int_Y
\mathcal C(y)
\bigl(E_{kl}+e_y(\eta_{kl})\bigr)
\,\di{y}
\\
&+
p\left[
\int_Y
\mathcal C(y)e_y(q)\,\di{y}
-\alpha |Y^g|I
\right]
\\
&+
\delta\sum_{k,l=1}^3
\int_0^t
e_{kl}(\partial_su(s))
\left[
\int_Y
\mathcal C(y)
e_y\bigl(\partial_t\zeta_{kl}(t-s)\bigr)
\,\di{y}
\right]
\,\di{s}
\\
&+
\delta\int_0^t
\partial_sp(s)
\left[
\int_Y
\mathcal C(y)
e_y\bigl(\partial_t\theta(t-s)\bigr)
\,\di{y}
\right]
\,\di{s}.
\end{align*}
By the definitions of \(\mathbb C^{\rm h}\) and \(A^{\rm h}\),
the first two lines reduce to
\[
\mathbb C^{\rm h}e(u)-A^{\rm h}p.
\]
Hence,
\begin{align}
\int_Y
\mathcal C(y)
\bigl(e(u)+e_y(u_1)\bigr)
\,\di{y}
-\alpha |Y^g|p\,I
={}&
\mathbb C^{\rm h}e(u)-A^{\rm h}p
\notag\\
&+
\delta\sum_{k,l=1}^3
\int_0^t
e_{kl}(\partial_su(s))
\left[
\int_Y
\mathcal C(y)
e_y\bigl(\partial_t\zeta_{kl}(t-s)\bigr)
\,\di{y}
\right]
\,\di{s}
\notag\\
&+
\delta\int_0^t
\partial_sp(s)
\left[
\int_Y
\mathcal C(y)
e_y\bigl(\partial_t\theta(t-s)\bigr)
\,\di{y}
\right]
\,\di{s}.
\label{eq:elastic-stress-expanded}
\end{align}
Moreover, applying Leibniz' rule to the convolution terms gives
\begin{align*}
\delta\int_{Y^g} \mathcal C^g \bigl(e(\partial_tu)+e_y(\partial_tu_1)\bigr) \,\di{y}
={}& \delta\sum_{k,l=1}^3 e_{kl}(\partial_tu) \int_{Y^g} \mathcal C^g \bigl(E_{kl}+e_y(\eta_{kl})\bigr) \,\di{y} \\ 
&+ \delta\,\partial_tp \int_{Y^g} \mathcal C^g e_y(q)\,\di{y} \\ 
&+ \delta^2\sum_{k,l=1}^3 e_{kl}(\partial_tu) \int_{Y^g} \mathcal C^g e_y\bigl(\partial_t\zeta_{kl}(0)\bigr) \,\di{y} \\
&+ \delta^2\,\partial_tp \int_{Y^g} \mathcal C^g e_y\bigl(\partial_t\theta(0)\bigr) \,\di{y} \\ 
&+ \delta^2\sum_{k,l=1}^3 \int_0^t e_{kl}(\partial_su(s)) \left[ \int_{Y^g} \mathcal C^g e_y\bigl(\partial_{tt}\zeta_{kl}(t-s)\bigr) \,\di{y} \right] \,\di{s} \\ 
&+ \delta^2\int_0^t \partial_sp(s) \left[ \int_{Y^g} \mathcal C^g e_y\bigl(\partial_{tt}\theta(t-s)\bigr) \,\di{y} \right] \,\di{s}.
\end{align*}
To identify the instantaneous viscous contribution, we differentiate the convolution terms using Leibniz' rule. 
For the strain-dependent part,
\[
\partial_t
\left[
\int_0^t
\partial_t\zeta_{kl}(t-s,y)\,
e_{kl}(\partial_su(s,x))
\,\di{s}
\right]
=
\partial_t\zeta_{kl}(0,y)\,
e_{kl}(\partial_tu(t,x))
+
\int_0^t
\partial_{tt}\zeta_{kl}(t-s,y)\,
e_{kl}(\partial_su(s,x))
\,\di{s}.
\]
Hence, the terms proportional to \(e_{kl}(\partial_tu(t,x))\) in the
averaged viscous stress are
\begin{multline*}
\delta
\sum_{k,l=1}^3
e_{kl}(\partial_tu)
\int_{Y^g}
\mathcal C^g
\bigl(E_{kl}+e_y(\eta_{kl})\bigr)
\,\di{y}
+
\delta^2
\sum_{k,l=1}^3
e_{kl}(\partial_tu)
\int_{Y^g}
\mathcal C^g
e_y\bigl(\partial_t\zeta_{kl}(0)\bigr)
\,\di{y}\\
=
\delta
\sum_{k,l=1}^3
e_{kl}(\partial_tu)
\int_{Y^g}
\mathcal C^g
\Bigl(
E_{kl}
+
e_y\bigl(
\eta_{kl}
+
\delta\partial_t\zeta_{kl}(0)
\bigr)
\Bigr)
\,\di{y}.
\end{multline*}
Evaluating the relaxation cell problem at \(t=0\) gives
\[
\int_{Y^g}
\mathcal C^g
\Bigl(
E_{kl}
+
e_y\bigl(
\eta_{kl}
+
\delta\partial_t\zeta_{kl}(0)
\bigr)
\Bigr)
:e_y(\phi)
\,\di{y}
=0
\]
for every admissible \(\phi\). 
Comparing this identity with the cell problem defining
\(\xi_{kl}\), we obtain
\[
e_y(\xi_{kl})
=
e_y(\eta_{kl})
+
\delta e_y\bigl(\partial_t\zeta_{kl}(0)\bigr)
\qquad\text{in }Y^g.
\]
This identity also determines them up to the corresponding rigid
displacement.
Consequently,
\begin{align*}
&\delta
\sum_{k,l=1}^3
e_{kl}(\partial_tu)
\int_{Y^g}
\mathcal C^g
\Bigl(
E_{kl}
+
e_y\bigl(
\eta_{kl}
+
\delta\partial_t\zeta_{kl}(0)
\bigr)
\Bigr)
\,\di{y}
\\
&\qquad=
\delta
\sum_{k,l=1}^3
e_{kl}(\partial_tu)
\int_{Y^g}
\mathcal C^g
\bigl(E_{kl}+e_y(\xi_{kl})\bigr)
\,\di{y}
\\
&\qquad=
\delta\,\mathbb D^{\rm h}e(\partial_tu).
\end{align*}
Similarly,
\[
\partial_t
\left[
\int_0^t
\partial_t\theta(t-s,y)\,
\partial_sp(s,x)
\,\di{s}
\right]
=
\partial_t\theta(0,y)\,\partial_tp(t,x)
+
\int_0^t
\partial_{tt}\theta(t-s,y)\,
\partial_sp(s,x)\,\di{s}.
\]
Thus, the instantaneous pressure-rate contribution is
\begin{align*}
\delta\,\partial_tp
\int_{Y^g}
\mathcal C^g e_y(q)\,\di{y}
+
\delta^2\,\partial_tp
\int_{Y^g}
\mathcal C^g e_y\bigl(\partial_t\theta(0)\bigr)\,\di{y}=
\delta\,\partial_tp
\int_{Y^g}
\mathcal C^g
e_y\bigl(q+\delta\partial_t\theta(0)\bigr)
\,\di{y}.
\end{align*}
On the other hand, evaluating the cell problem for \(\theta\) at
\(t=0\) and using \(\theta(0)=0\), we obtain
\[
\int_{Y^g}
\mathcal C^g
e_y\bigl(q+\delta\partial_t\theta(0)\bigr)
:e_y(\phi)\,\di{y}
=0
\]
for every admissible test function \(\phi\). Choosing
\[
\phi=q+\delta\partial_t\theta(0)
\]
and using the coercivity of \(\mathcal C^g\), it follows that
\[
e_y\bigl(q+\delta\partial_t\theta(0)\bigr)=0
\qquad\text{in }Y^g.
\]
Hence,
\[
\int_{Y^g}
\mathcal C^g
e_y\bigl(q+\delta\partial_t\theta(0)\bigr)
\,\di{y}
=0,
\]
and therefore the instantaneous pressure-rate contribution vanishes:
\[
\delta\,\partial_tp
\int_{Y^g}
\mathcal C^g
e_y\bigl(q+\delta\partial_t\theta(0)\bigr)
\,\di{y}
=0.
\]
Collecting the remaining strain-history contributions gives
\begin{multline*}
\delta\sum_{k,l=1}^3
\int_0^t
e_{kl}(\partial_su(s))
\Bigg[
\int_Y
\mathcal C(y)
e_y\bigl(\partial_t\zeta_{kl}(t-s)\bigr)
\,\di{y}+
\delta\int_{Y^g}
\mathcal C^g
e_y\bigl(\partial_{tt}\zeta_{kl}(t-s)\bigr)
\,\di{y}
\Bigg]
\,\di{s}
\\
=
\delta\,
\bigl(
\partial_t\mathbb M^{\rm h}
*
e(\partial_tu)
\bigr)(t).
\end{multline*}
Similarly,
\begin{equation*}
\delta\int_0^t
\partial_sp(s)
\Bigg[
\int_Y
\mathcal C(y)
e_y\bigl(\partial_t\theta(t-s)\bigr)
\,\di{y}
+
\delta\int_{Y^g}
\mathcal C^g
e_y\bigl(\partial_{tt}\theta(t-s)\bigr)
\,\di{y}
\Bigg]
\,\di{s}
=
\delta\,
\bigl(
\partial_tR^{\rm h}
*
\partial_tp
\bigr)(t).
\end{equation*}
This shows the stated decomposition.
\end{proof}

\begin{lemma}\label{lemma:pressure_macroscopic}
    The macroscopic effective fluid content $F(u,p)$ satisfies
    \[
    F(u,p)=F^{\rm eq}(u,p)+\delta F^{\rm mem}(u,p)+\frac\eta\alpha\partial_t\left[F^{\rm eq}(u,p)-c|Y^g|p+\delta F^{\rm mem}(u,p)\right]
    \]
    with
    \begin{align*}
    F^{\rm eq}(u,p)=c^{\rm h}p + A^{\rm h}:e(u),\qquad
    F^{\rm mem}(u,p)=-\partial_t R^{\rm h}*e(\partial_tu) +H^{\rm h}*\partial_tp.
    \end{align*}
\end{lemma}
\begin{proof}
By definition, the averaged fluid content before eliminating the
microscopic displacement corrector is
\[
F(u,p)
:=
\int_{Y^g}
\Bigl[
cp
+
\alpha\bigl(\dive_xu+\dive_yu_1\bigr)
+
\eta\bigl(
\dive_x\partial_tu+\dive_y\partial_tu_1
\bigr)
\Bigr]
\,\di{y}.
\]
Since \(u\) and \(p\) do not depend on \(y\), this may be written as
\[
F(u,p)
=
c|Y^g|p
+
\alpha\mathcal V(u,p)
+
\eta\,\partial_t\mathcal V(u,p),
\]
where
\[
\mathcal V(u,p)
:=
|Y^g|\dive_xu
+
\int_{Y^g}\dive_yu_1\,\di{y}.
\]

We now use the representation
\[
u_1
=
\sum_{k,l=1}^3
\eta_{kl}e_{kl}(u)
+
qp
+
\delta\sum_{k,l=1}^3
\partial_t\zeta_{kl}*e_{kl}(\partial_tu)
+
\delta\,\partial_t\theta*\partial_tp.
\]
Consequently,
\begin{align}
\mathcal V(u,p)
={}&
|Y^g|\dive_xu
+
\sum_{k,l=1}^3
e_{kl}(u)
\int_{Y^g}\dive_y\eta_{kl}\,\di{y}
\notag\\
&+
p\int_{Y^g}\dive_yq\,\di{y}
\notag\\
&+
\delta\sum_{k,l=1}^3
\left(
\int_{Y^g}
\dive_y\partial_t\zeta_{kl}\,\di{y}
\right)
*e_{kl}(\partial_tu)
\notag\\
&+
\delta
\left(
\int_{Y^g}
\dive_y\partial_t\theta\,\di{y}
\right)
*\partial_tp.
\label{eq:proof-fluid-content-V}
\end{align}

We first identify the instantaneous strain contribution. Testing the
cell problem for \(q\) with \(\eta_{kl}\) gives
\[
\int_Y
\mathcal C e_y(q):e_y(\eta_{kl})
\,\di{y}
=
\alpha\int_{Y^g}\dive_y\eta_{kl}\,\di{y}.
\]
On the other hand, testing the cell problem for \(\eta_{kl}\) with
\(q\) yields
\[
\int_Y
\mathcal C
\bigl(E_{kl}+e_y(\eta_{kl})\bigr)
:e_y(q)
\,\di{y}
=0.
\]
Using the symmetry of \(\mathcal C\), we therefore obtain
\[
\alpha\int_{Y^g}\dive_y\eta_{kl}\,\di{y}
=
-
\int_Y
\bigl(\mathcal C e_y(q)\bigr)_{kl}
\,\di{y}.
\]
It follows that
\[
\alpha
\left(
|Y^g|\delta_{kl}
+
\int_{Y^g}\dive_y\eta_{kl}\,\di{y}
\right)
=
A^{\rm h}_{kl}.
\]
Hence,
\[
\alpha
\left[
|Y^g|\dive_xu
+
\sum_{k,l=1}^3
e_{kl}(u)
\int_{Y^g}\dive_y\eta_{kl}\,\di{y}
\right]
=
A^{\rm h}:e(u).
\]

Moreover, by the definition of \(c^{\rm h}\),
\[
c|Y^g|p
+
\alpha p\int_{Y^g}\dive_yq\,\di{y}
=
c^{\rm h}p.
\]
Thus, the instantaneous part of the fluid content is
\[
F^{\rm eq}(u,p)
=
c^{\rm h}p
+
A^{\rm h}:e(u).
\]

We next identify the strain-memory term. Set
\[
L_{kl}(t)
:=
\alpha\int_{Y^g}
\dive_y\zeta_{kl}(t,y)
\,\di{y}.
\]
Testing the cell problem for \(q\) with \(\zeta_{kl}(t)\) gives
\[
L_{kl}(t)
=
\int_Y
\mathcal C e_y(q):e_y(\zeta_{kl}(t))
\,\di{y}.
\]
Testing the relaxation problem for \(\zeta_{kl}\) with \(q\), we find
\[
L_{kl}(t)
+
\delta\int_{Y^g}
\mathcal C^g e_y(\partial_t\zeta_{kl}(t)):e_y(q)
\,\di{y}
=
-
\int_{Y^g}
\mathcal C^g
\bigl(E_{kl}+e_y(\eta_{kl})\bigr):e_y(q)
\,\di{y}.
\]
A reciprocity argument for the two relaxation problems gives
\[
\int_{Y^g}
\mathcal C^g
\bigl(E_{kl}+e_y(\eta_{kl})\bigr)
:e_y(\partial_t\theta(t))
\,\di{y}
=
\int_{Y^g}
\mathcal C^g e_y(q):e_y(\partial_t\zeta_{kl}(t))
\,\di{y}.
\]
Using the cell problem for \(\eta_{kl}\) and the relaxation problem for
\(\theta\), we consequently obtain
\[
R^{\rm h}_{kl}(t)
=
-
\int_{Y^g}
\mathcal C^g E_{kl}:e_y(q)\,\di{y}
-
L_{kl}(t).
\]
The first term on the right-hand side is independent of time.
Therefore,
\[
\partial_tR^{\rm h}_{kl}(t)
=
-\partial_tL_{kl}(t)
=
-\alpha\int_{Y^g}
\dive_y\partial_t\zeta_{kl}(t,y)
\,\di{y}.
\]
Hence,
\begin{align*}
&\alpha\delta
\sum_{k,l=1}^3
\left(
\int_{Y^g}
\dive_y\partial_t\zeta_{kl}\,\di{y}
\right)
*e_{kl}(\partial_tu)=
-\delta\,
\partial_tR^{\rm h}*e(\partial_tu).
\end{align*}
For the pressure-memory term, the definition of \(H^{\rm h}\) gives
directly
\[
\alpha\delta
\left(
\int_{Y^g}
\dive_y\partial_t\theta\,\di{y}
\right)
*\partial_tp
=
\delta\,H^{\rm h}*\partial_tp.
\]
Therefore, defining
\[
F^{\rm mem}(u,p)
:=
-\partial_tR^{\rm h}*e(\partial_tu)
+
H^{\rm h}*\partial_tp,
\]
equation \eqref{eq:proof-fluid-content-V} becomes
\[
c|Y^g|p+\alpha\mathcal V(u,p)
=
F^{\rm eq}(u,p)
+
\delta F^{\rm mem}(u,p).
\]

Furthermore,
\[
\alpha\mathcal V(u,p)
=
F^{\rm eq}(u,p)
-
c|Y^g|p
+
\delta F^{\rm mem}(u,p).
\]
Since
\[
\eta\,\partial_t\mathcal V(u,p)
=
\frac{\eta}{\alpha}
\partial_t\bigl(\alpha\mathcal V(u,p)\bigr),
\]
we finally obtain
\[
F(u,p)
=
F^{\rm eq}(u,p)
+
\delta F^{\rm mem}(u,p)
+
\frac{\eta}{\alpha}
\partial_t
\left[
F^{\rm eq}(u,p)
-
c|Y^g|p
+
\delta F^{\rm mem}(u,p)
\right].
\]
This proves the assertion.
\end{proof}

\bibliographystyle{abbrv}
\bibliography{references}
\end{document}